\documentclass[twoside,a4paper,reqno,11pt]{amsart}
\usepackage{amsfonts,amsmath,amssymb,amsthm,float}
\usepackage{color}
\usepackage{fullpage}
\usepackage{graphicx}
\usepackage{comment}
\usepackage{longtable}
\usepackage{pdflscape}
\usepackage{enumitem}
\usepackage{array}
\usepackage{calc}
\usepackage{makecell}
\usepackage[pdfencoding=auto,psdextra]{hyperref}

\setlist[enumerate]{itemsep=3pt,topsep=3pt}
\setlist[enumerate,1]{label=\textup{(\roman*)}}
\setlist[enumerate,2]{label=\textup{(\alph*)}}

\newtheorem{maintheorem}{Theorem}
\newtheorem{maincorollary}[maintheorem]{Corollary}
\newtheorem{theorem}{Theorem}[section]
\newtheorem{lemma}[theorem]{Lemma}
\newtheorem{proposition}[theorem]{Proposition}
\newtheorem{corollary}[theorem]{Corollary}

\theoremstyle{definition}
\newtheorem{remark}[theorem]{Remark}

\DeclareMathOperator{\SL}{SL}
\DeclareMathOperator{\GL}{GL}
\DeclareMathOperator{\SO}{SO}

\DeclareMathOperator{\Sp}{Sp}
\DeclareMathOperator{\Soc}{Soc}

\DeclareMathOperator{\Hom}{Hom}
\DeclareMathOperator{\Ext}{Ext}

\numberwithin{paragraph}{section}
\newcommand{\Magma}{\textsc{Magma}}
\newcommand{\Vmin}{\ensuremath{V_{\textup{min}}}}

\setlist[enumerate]{topsep=0pt}

\title{Complete Reducibility in Bad Characteristic}
\author{Alastair J.\ Litterick and Adam R.\ Thomas}

\renewcommand\footnotemark{}

\begin{document}

\begin{abstract}
Let $G$ be a simple algebraic group of exceptional type over an algebraically closed field of characteristic $p > 0$. This paper continues a long-standing effort to classify the connected reductive subgroups of $G$. Having previously completed the classification when $p$ is sufficiently large, we focus here on the case that $p$ is bad for $G$. We classify the connected reductive subgroups of $G$ which are not $G$-completely reducible, whose simple components have rank at least $3$. For each such subgroup $X$, we determine the action of $X$ on the adjoint module $L(G)$ and on a minimal non-trivial $G$-module, and the connected centraliser of $X$ in $G$. As corollaries we obtain information on: subgroups which are maximal among connected reductive subgroups; products of commuting $G$-completely reducible subgroups; subgroups with trivial connected centraliser; and subgroups which act indecomposably on an adjoint or minimal module for $G$.
\end{abstract}

\maketitle

\setlength{\parskip}{4pt}


\section{Introduction and results} \label{sec:intro}

This paper concerns the closed subgroups of reductive algebraic groups over algebraically closed fields, the study of which dates back to work of Dynkin in the 1950s. A result of Borel and Tits \cite{MR0294349} states that such a subgroup is either itself reductive or lies in a proper parabolic subgroup of $G$. Thus attention is focused on reductive subgroups. Dynkin himself classified the reductive maximal connected subgroups of simple algebraic groups in characteristic zero \cite{MR0047629}, and this was extended into positive characteristic by Seitz \cite{MR888704,MR1048074} and Liebeck and Seitz \cite{MR2044850}.

It is natural to try and extend this classification to \emph{all} reductive subgroups. For instance, for $X$ a closed subgroup of a reductive algebraic group $G$, the coset space $G/X$ is an affine variety precisely when $X$ is reductive \cite{MR437549}. The subgroup structure of $G$ is also instrumental in studying the corresponding finite groups of Lie type \cite{MR1458329,MR4386346,craven2021maximal,craven2022maximal}.

A fundamental result states that in characteristic zero, a closed subgroup $X$ of $G$ is reductive if and only if $X$ is \emph{$G$-completely reducible} ($G$-cr), meaning that whenever $X$ is contained in a parabolic subgroup $P$ of $G$, it is contained in a Levi factor of $P$. Thus Dynkin's work provides an inductive method for classifying reductive subgroups: one runs over Levi subgroups $L$ of $G$, and studies the proper subgroups of each $L$ not lying in a proper parabolic subgroup of $L$, the so-called \emph{$L$-irreducible subgroups}.

In positive characteristic the study is significantly harder. With $P = QL$ the Levi decomposition of a parabolic subgroup of $G$, reductive subgroups of $P$ no longer necessarily have a $P$-conjugate in $L$. Instead, one must study the actions of subgroups $X_{0}$ of $L$ on the unipotent radical $Q$, and the corresponding complements in the semidirect product $QX_0$. Thus questions of \emph{non-abelian cohomology} arise.

A natural dichotomy arises: When $G$ is simple of classical type, the rank of $G$ can be arbitrarily large and classifying reductive subgroups of all $G$ is equivalent to understanding all representations of reductive groups, a famously intractable problem. On the other hand, the case when $G$ has exceptional type is both interesting (since exceptional groups exhibit the most complex behaviour) and a realistic prospect (due to the bounded rank).

So let $G$ be simple of exceptional type, over an algebraically closed field of characteristic $p > 0$. When $p$ is sufficiently large relative to the Lie type of $G$, all reductive subgroups of $G$ are again $G$-cr and we can proceed inductively. In this way, in \cite{MR1329942} Liebeck and Seitz obtain a classification of all reductive subgroups of $G$ which holds, for instance, if $p > 7$. In general, smaller prime characteristics and smaller subgroups (as measured by the Lie rank) result in more complicated cohomology and more complicated subgroup structure. This study was initiated for exceptional $G$ by D.~Stewart \cite{MR2604850,MR3075783} for $G$ of type $G_2$ or $F_4$. Here, non-$G$-cr reductive subgroups exist only when $p \le 3$. Work of the present authors \cite{Litterick2018} solves the problem when the characteristic of the underlying field is \emph{good for $G$}. Recall that for an exceptional group $G$, a prime $p$ is called \emph{bad} if $p = 2$ or $3$, or $p = 5$ with $G$ of type $E_8$; otherwise $p$ is called \emph{good}.

The main result of this paper completes the classification of reductive subgroups of $G$ whose simple factors each have rank at least $3$, with no restriction on the characteristic.

\begin{maintheorem} \label{THM:MAIN}
Let $G$ be a simple algebraic group of exceptional type over an algebraically closed field of characteristic $p > 0$. Let $X$ be a semisimple subgroup of $G$, such that $X$ is non-$G$-cr and each simple factor of $X$ has rank at least $3$. Then $p = 2$ or $3$, and $X$ is conjugate to exactly one subgroup listed in Tables \ref{tab:E6}--\ref{tab:E8p2} in Section~\ref{sec:tables}, each of which is non-$G$-cr.

For each such $X$, we give the connected centraliser $C_{G}(X)^{\circ}$ and the action of $X$ on the Lie algebra of $G$ and on a non-trivial module for $G$ of least dimension. 
\end{maintheorem}

Subgroups with factors of rank less than $3$ are more numerous and more delicate to study, and are the subject of ongoing work of the authors.

The proof of Theorem \ref{THM:MAIN} follows a similar blueprint to previous results \cite{Litterick2018,MR3075783} (cf.\ also \cite[\S 1.4.2]{LSTSurvey}) but the final result and arguments involved are much more delicate. Besides more complicated cohomology sets, useful properties such as \emph{separability} of subgroups (equality of the group-theoretic and infinitesimal centraliser dimensions) hold when $p$ is good for $G$ \cite[Theorem 1.1]{MR3042602} but often fail here. Also in good characteristic, many natural subgroups $H$ are either \emph{ascending hereditary}, meaning that connected $H$-cr subgroups of $H$ are $G$-cr, or \emph{descending hereditary}, meaning that connected $G$-cr subgroups of $H$ are $H$-cr \cite[Definition 2.47]{LSTSurvey}. For instance, \cite[Theorem 3.26]{MR2178661} states that reductive subgroups containing a maximal torus of $G$ are both ascending and descending hereditary. This fails in bad characteristic; an interesting example is given in \cite[Example 2.40]{LSTSurvey}, where two classes of $F_4$-cr subgroups of type $G_2$ are non-$B_4$-cr, non-$C_4$-cr respectively. The study of such hereditary properties is the subject of ongoing work.

\begin{remark} \label{rem:f4e6-not-dh}
Table \ref{tab:F4} shows that when $p = 2$, a group of type $F_4$ has two classes of subgroups of type $B_3$ which are non-$F_4$-cr, whereas Table \ref{tab:E6} shows that a group of type $E_6$ has a unique class of such subgroups. Comparing actions on low-dimensional modules, we see that the two $F_4$ subgroup classes are not fused in $E_6$. Thus at least one class of subgroups are $E_6$-cr, so that $F_4$ is not descending hereditary in $E_6$. This contrasts with the case $p \neq 2$, where $F_4$ is both ascending and descending hereditary, by \cite[Proposition 16.9]{MR952224} and \cite[Corollary 3.21]{MR2178661}.
\end{remark}

An interesting complication arising is the existence of non-$G$-cr subgroups which are maximal among connected reductive subgroups of $G$. For these, one cannot make use of prior results on known reductive subgroups of $G$, since all their maximal overgroups in $G$ are parabolic. For brevity, we call these subgroups \emph{MR}. The next result describes those occurring in Theorem~\ref{THM:MAIN}. 

\begin{maincorollary} \label{cor:MR}
With $G$ and $p$ as in Theorem~\ref{THM:MAIN}, let $X$ be an MR subgroup of $G$ where each simple factor of $X$ has rank at least $3$. Then either $X$ is a maximal connected subgroup of $G$ or 
\begin{enumerate}[label=\normalfont(\roman*)]
\item $G$ has type $E_6$ and $X$ is a Levi subgroup of type $D_5T_1$, \label{cor:mr-i}
\item $G$ has type $E_7$ and $X$ is a Levi subgroup of type $E_6T_1$, \label{cor:mr-ii}
\item $(G,X,p) = (E_7,D_4,2)$ and $X$ is non-$G$-cr, \label{cor:mr-iii}
\item $(G,X,p) = (E_8,D_4,2)$ and $X$ is non-$G$-cr. \label{cor:mr-iv}
\end{enumerate}
Each subgroup in \textup{(i)}--\textup{(iv)} is MR. Those in \textup{(i)}, \textup{(ii)} and \textup{(iv)} are unique up to conjugacy in $G$, and those in \textup{(iii)} form infinitely many conjugacy classes.
\end{maincorollary}
This corollary follows directly from Theorem~\ref{THM:MAIN} together with the classification of maximal connected subgroups \cite{MR2044850}. The subgroups in \ref{cor:mr-iii} were discovered in \cite{MR1367085}, and those in \ref{cor:mr-iv} are new. Here and elsewhere, we abuse terminology and write \emph{Levi subgroup of $G$} to mean a Levi factor of a parabolic subgroup of $G$.

For our next corollary, recall that a \emph{subsystem subgroup} of $G$ is a semisimple subgroup normalised by a maximal torus. These constitute a large class of subgroups of $G$ corresponding to $p$-closed subsystems of the root system of $G$. Connected reductive subgroups not contained in a subsystem subgroup are both relatively rare and less straightforward to study. The following result complements its good-characteristic analogue \cite[Corollary~6]{Litterick2018}.

\begin{maincorollary} \label{cor:overgroups}
Let $G$ and $X$ be as in the hypothesis of Theorem \ref{THM:MAIN}. Then either $X$ is contained in a proper subsystem subgroup of $G$, or $p=2$ and one of the following holds:
\begin{enumerate}[label=\normalfont(\roman*)]
\item $(G,X) = (E_6,B_3)$ and $X$ lies in a maximal subgroup $F_4$,
\item $(G,X) = (E_7,D_4)$ and $X$ is MR,
\item $(G,X) = (E_8,D_4)$ and $X$ is MR, \label{cor:overgroups-iii}
\item $(G,X) = (E_8,B_3)$ and $X$ is contained in a subgroup from part \ref{cor:overgroups-iii}.
\end{enumerate}
Each subgroup in \textup{(i)}--\textup{(iv)} lies in no proper subsystem subgroup of $G$. Those in \textup{(i)}, \textup{(iii)} and \textup{(iv)} are unique up to conjugacy in $G$, and those in \textup{(ii)} form infinitely many conjugacy classes.
\end{maincorollary}

Corollary~\ref{cor:overgroups} is proved by inspecting Tables~\ref{tab:F4}--\ref{tab:E8p2}: The subgroups $X$ listed there except (i)--(iv) above are given with an explicit embedding into a subsystem subgroup of $G$. The subgroups in (ii) and (iii) are MR so are in no proper reductive subgroup. By the Borel-de Siebenthal algorithm \cite{MR0032659}, when $G$ has type $E_6$ each subsystem subgroup of $G$ is contained in one of type $A_1 A_5$, $A_2^3$ or $D_5$. Of these, only $D_5$ can contain a subgroup of type $B_3$; but such a subsystem subgroup centralises a $1$-dimensional torus of $G$, whereas the subgroup $X$ in (i) does not. Similarly, inspecting Table~\ref{tab:E8p2} the subgroup $X$ in (iv) does not centralise a non-trivial torus in $G$, so could only be contained in a subsystem subgroup of maximal rank, with all simple factors of rank at least $3$. It remains to consider subsystem subgroups of type $A_8$, $D_8$ and $A_4^2$. The latter contains no subgroup $B_3$, and the former two act on $L(G)$ with indecomposable summands of dimension $80$, $84$ and $84$; and $120$ and $128$, respectively \cite[Lemma 11.2]{MR2883501}. This is incompatible with the stated action of $X$, whose indecomposable summands have dimensions $30$, $30$, $62$, $63$ and $63$.

Next, the main result of \cite{MR2431255} states that the product of two commuting $G$-cr subgroups is again $G$-cr, as long as the characteristic $p$ is good for $G$ or $p > 3$. We conjecture that $p \neq 2$ is in fact sufficient \cite[Remark 1.12]{LSTSurvey}. Our next result is evidence in this direction.

\begin{maincorollary} \label{cor:gcrfactors}
Let $G$ and $X$ be as in the hypothesis of Theorem \ref{THM:MAIN}. If $X$ is not simple then $p = 2$, $G$ has type $E_8$ and $X$ has type $B_3^2$. Each simple factor of $X$ is $G$-cr. 
\end{maincorollary}

The first statement here follows immediately from the classification of non-$G$-cr subgroups in Theorem~\ref{THM:MAIN}. The final statement uses the additional information contained in Tables~\ref{tab:E6}--\ref{tab:E8p2}. Specifically, the non-$G$-cr subgroups of type $B_3$ in Theorem~\ref{THM:MAIN} each have centraliser of rank at most $2$, while the factors of the non-$G$-cr subgroup $B_3^2$ each centralise one another, hence are not among these non-$G$-cr subgroups.

For our next result, recall from \cite[Lemma 3.17]{MR2178661} that the centraliser of a $G$-cr subgroup is again $G$-cr, hence reductive. In particular, when $G$ is semisimple a $G$-irreducible subgroup has trivial connected centraliser \cite[Lemma 2.1]{MR2043006}. When the characteristic is zero or sufficiently large, the converse holds: Since a reductive subgroup is $G$-cr, it is $G$-irreducible if and only if it is contained in no proper Levi subgroup of $G$, if and only if its connected centraliser is trivial. This fails in positive characteristic, and our next result describes the extent of this failure. This extends \cite[Corollary 9]{Litterick2018}. The notation for the embeddings is given in Section~\ref{sec:embed}.

\begin{maincorollary} \label{cor:trivcentraliser}
Let $G$ and $X$ be as in the hypothesis of Theorem \ref{THM:MAIN}, and suppose that $C_{G}(X)^{\circ}$ is reductive. Then $C_{G}(X)^{\circ} = 1$.

This occurs for precisely two subgroups $X$, up to conjugacy. Namely, $p = 2$, $G$ has type $E_8$ and $X$ is either a subgroup $A_3 < D_8$ via $T(101)$, or $B_4 < D_8$ via $0001$.
\end{maincorollary}

Our final result highlights an interesting phenomenon regarding the actions of reductive subgroups of $G$ on low-dimensional $G$-modules. Namely, one typically expects that a proper reductive subgroup will act with many indecomposable direct summands on a given irreducible $G$-module. However, we exhibit some subgroups which act indecomposably. This complements prior results such as \cite{MR2063402} which classify subgroups acting irreducibly. In the following, we write $\Vmin$ or $V_{n}$ for a non-trivial $G$-module of least dimension, which is $n$.

\begin{maincorollary} \label{cor:indecomp}
Let $G$ and $X$ be as in the hypothesis of Theorem \ref{THM:MAIN} and let $V = L(G)$ or $\Vmin$. Then $V \downarrow X$ is indecomposable if and only if $p = 2$ and $(G,X,V)$ is one of the following:
\begin{enumerate}[label=\normalfont(\roman*)]
    \item $(G,X) = (F_4,B_3)$ with $X$ in a short-root subgroup $\tilde{D}_4$ and $V = V_{26}$,
    \item $(G,X) = (E_6,B_3)$ with $V = L(G)$,
    \item $(G,X) = (E_7,D_4)$ with $X$ an MR subgroup and $V = V_{56}$ or $V = L(G)$.
\end{enumerate}
\end{maincorollary}

\subsection*{Layout of the paper} After defining relevant notation in Section~\ref{sec:notation} and giving preliminary results in Section~\ref{sec:preliminaries}, Theorem~\ref{THM:MAIN} is proved in Sections~\ref{sec:A3}--\ref{sec:D4}. Each section corresponds to a Lie type of non-$G$-cr subgroups, split into subsections according to the strategy described in Section~\ref{sec:enumerating}. Section~\ref{sec:actions} then determines the restrictions $L(G) \downarrow X$ and $\Vmin \downarrow X$ for each $X$ arising in Theorem~\ref{THM:MAIN}, and the tables of these subgroups and associated data are given Section~\ref{sec:tables}.


\section{Notation} \label{sec:notation}

\subsection{Algebraic groups, roots, weights.} Throughout, all groups are linear algebraic groups over an algebraically closed field $K$ of characteristic $p > 0$. We take the `group variety' viewpoint \cite{MR1102012}, so that our algebraic groups are the $K$-points of Zariski-closed subgroups of general linear groups over $K$. All subgroups mentioned are closed and all group homomorphisms are variety morphisms. The terms `simple' and `semisimple' always refer to connected groups. Group actions are on the left, denoted by a dot, i.e.~$(g,q) \mapsto g \cdot q$ for $g$ in a group $G$ and $q$ in a $G$-set.

For a reductive algebraic group $G$, fix a maximal torus $T$ of $G$ and corresponding root system $\Phi$. Let $\Pi = \{ \alpha_1, \ldots, \alpha_l \}$ be a base of simple roots and let $\{\lambda_1,\ldots,\lambda_{l}\}$ be the corresponding fundamental dominant weights, with the Bourbaki ordering \cite[Ch.\ VI, Planches I-IX]{MR0240238}. The notation $a_1 a_2 \ldots a_l$ will indicate a root $\sum a_{i}\alpha_{i}$ or a weight $\sum a_{i}\lambda_{i}$; context will prevent ambiguity.

The Weyl group $W = N_G(T) / T$ acts on $\mathbb{Z}\Phi \otimes \mathbb{R}$. For $\alpha \in \Phi$, let $s_{\alpha}$ denote the reflection in the hyperplane perpendicular to $\alpha$, and let $U_{\alpha} = \{ x_{\alpha}(t) \, : \, t \in K \}$ be the root subgroup corresponding to $\alpha$. For $t \in K^{\ast}$, define
\begin{align*}
n_{\alpha}(t) &= x_{\alpha}(t)x_{-\alpha}(-t^{-1})x_{\alpha}(t) \in N_{G}(T),\\
h_{\alpha}(t) &= n_{\alpha}(t)n_{\alpha}(-1) \in T,
\end{align*}
so that $n_{\alpha}(t)$ maps to $s_{\alpha} \in W$, \cite[\S 6.4]{MR0407163}. Set $n_{\alpha} = n_{\alpha}(1)$ and for a simple root $\alpha_{i}$ let $s_{i} = s_{\alpha_{i}}$, $n_{i} = n_{\alpha_i}$ and $h_{i}(t) = h_{\alpha_i}(t)$.

The notation $\bar{X}$ denotes a subgroup of $G$ generated by long root subgroups of $G$. If $\Phi$ has multiple root lengths then $\tilde{X}$ denotes a subgroup generated by short root subgroups of $G$.

\subsection{Modules.} For a reductive algebraic group $G$ and a dominant weight $\lambda$ of $G$, let $V_{G}(\lambda)$ denote the irreducible $G$-module of highest weight $\lambda$. When $G$ is clear we write simply $\lambda$ for the module; in particular $0$ denotes the trivial irreducible $G$-module. The corresponding Weyl module for $G$ is denoted $W(\lambda)$ or $W_{G}(\lambda)$, and the tilting module is denoted $T(\lambda)$ or $T_{G}(\lambda)$. We write $V^{*} = \Hom_{K}(V,K)$ for the dual of a $G$-module $V$. If $G = G_1 G_2 \ldots G_k$ is a commuting product of reductive algebraic groups then $(V_1, \ldots, V_k)$ denotes the $G$-module $V_1 \otimes \cdots \otimes V_k$, where $V_i$ is a $G_i$-module for each $i$.

Let $F : G \to G$ be a Frobenius endomorphism which acts on root elements via $x_{\alpha}(t) \mapsto x_{\alpha}(t^p)$ for $\alpha \in \Phi$, $t \in K$. If $V$ is a $G$-module afforded by a representation $\rho : G \rightarrow GL(V)$ then $V^{[r]}$ denotes the module afforded by $\rho^{[r]} := \rho \circ F^r$. Let $M_1, \ldots, M_k$ be $G$-modules and $m_1, \ldots, m_k$ be positive integers. Then $M_1^{m_1} / \ldots / M_k^{m_k}$ denotes a $G$-module having the same composition factors as $M_1^{m_1} \oplus \ldots \oplus M_k^{m_k}$. Furthermore, $V = M_1 | \ldots | M_k$ denotes a $G$-module with socle series $V = V_{1} > V_{2} > \ldots > V_{k+1} = \{0\}$, so that $\Soc(V/V_{i+1}) = V_{i}/V_{i+1} \cong M_{i}$ for $1 \le i \le k$. We denote by $L(G)$ the adjoint $G$-module. For $G$ of type $F_4$, $E_6$ and $E_7$ we respectively write $V_{26} = V_{G}(\lambda_4)$, $V_{27} = V_{G}(\lambda_1)$ and $V_{56} = V_{G}(\lambda_7)$ for these non-trivial modules of least dimension.

\subsection{Parabolic and Levi subgroups.} Let $J = \{\alpha_{j_1}, \alpha_{j_2}, \ldots, \alpha_{j_r}\} \subseteq \Pi$ and let $\Phi_J = \Phi \cap \mathbb{Z}J$. Then the standard parabolic subgroup of $G$ corresponding to $J$ is $P_J = \left< T, U_\alpha \, : \, \alpha \in \Phi_J \cup \Phi^+ \right>$. This has Levi decomposition $P_J = Q_JL_J$ where $Q_J = R_{u}(P) = \left< U_\alpha \, : \, \alpha \in \Phi^+ \setminus \Phi_J \right>$, and $L_J = \left< T, U_\alpha \, : \, \alpha \in \Phi_J \right>$ is the standard Levi factor of $P_J$. If the derived subgroup $L'$ has Lie type $\mathbf{X}$, we call $P_J$ and $L_J$ an $\mathbf{X}$-parabolic subgroup and $\mathbf{X}$-Levi subgroup of $G$ respectively. For instance when $G$ has type $E_6$ the subgroup $P_{2345}$ is a $D_{4}$-parabolic subgroup of $G$. Recall that `Levi subgroup of $G$' means any Levi factor of a parabolic subgroup of $G$.

\subsection{Shape modules.} \label{sec:shape-modules} For a subset $J \subseteq \Pi$, the \emph{level} of $\gamma = \sum_{\alpha \in \Pi} c_{\alpha} \alpha$ is the sum $\sum_{\alpha \in \Pi \setminus J} c_{\alpha}$, and the \emph{shape} of $\gamma$ (with respect to $J$) is $\sum_{\alpha \in \Pi \setminus J} c_{\alpha}\alpha$. Thus for a standard parabolic subgroup $P_J = Q_J L_J$ as above, $Q_J$ is generated by root subgroups corresponding to roots of positive level. For each $i \ge 1$ we define
\[ Q(i) = \left< U_\gamma : \gamma \text{ has level} \ge i \right>, \]
an $L$-stable normal subgroup of $Q_J$. The \emph{$i$-th level of $Q$} is $Q(i)/Q(i+1)$, which is central in $Q/Q(i+1)$. By \cite[Theorem 2 and Remark 1]{MR1047327}, each level is an $L$-module, and is a direct sum of \emph{shape modules} $V_S$, where $S$ runs over shapes of level $i$. Each $V_S$ is either irreducible or indecomposable of length $2$, the latter only if $(G,p) = (G_2,2)$, $(G_2,3)$ or $(F_4,2)$. Moreover the high weight(s) of $L$ on $V_S$ can be determined combinatorially, cf.\ \cite[Lemma 3.1]{MR1329942}.

\subsection{Centralisers.} \label{sec:centnot} When discussing subgroup centralisers, we write $U_{i}$ for an $i$-dimensional unipotent group, and $T_{j}$ for a $j$-dimensional torus. For instance, $C_G(X)^{\circ} = U_{5} T_{1}$ means that $C_{G}(X)^{\circ}$ has a $5$-dimensional unipotent radical, with quotient a 1-dimensional torus.

\subsection{Roots of $G$ and its Levi subgroups.} \label{sec:roots-levis} If $L$ is a standard Levi subgroup of $G$ then roots of $L$ are roots of $G$ and we must be consistent in how we identify these. If $L_{0}$ is a simple factor of $L$ then let $\Psi = \{\alpha_{1}',\ldots,\alpha_{m}'\}$ be the simple roots of $L_{0}$, a subset of $\Pi$. Linearly order $\Pi$ in such a way that $\alpha_{i} < \alpha_{j}$ whenever $i < j$. If $L_{0}$ has Lie type $A_{m}$, the embedding $L \to G$ is chosen such that $\alpha_{1}'$ is the least simple root of $G$ contained in $\Psi$. If $L_{0}$ has type $E_{6}$ or $E_{7}$ then $\alpha_{i}' = \alpha_{i}$ for all $i$. If $L_{0}$ has type $D_{4}$ then $(\alpha_{1}',\alpha_{2}',\alpha_{3}',\alpha_{4}') = (\alpha_{2},\alpha_{4},\alpha_{3},\alpha_{5})$. If $L_{0}$ has type $D_{5}$ then $(\alpha_{1}', \alpha_{2}',\alpha_{3}',\alpha_{4}',\alpha_{5}') = (\alpha_{1},\alpha_{3},\alpha_{4},\alpha_{5},\alpha_{2})$ or $(\alpha_{6},\alpha_{5},\alpha_{4},\alpha_{3},\alpha_{2})$. If $L_{0}$ has type $D_{6}$ then $(\alpha_{1}',\alpha_{2}',\alpha_{3}',\alpha_{4}',\alpha_{5}',\alpha_{6}') = (\alpha_{7},\alpha_{6},\alpha_{5},\alpha_{4},\alpha_{3},\alpha_{2})$. Finally, if $L$ has multiple components of the same type, then these components are ordered according to the position of their least simple root ``$\alpha_{1}'$'' as an element of $\Pi$. For instance, if $G = E_{7}$ and $L = L_{12567}$, then $L$ is a Levi subgroup of type $A_{1}A_{1}A_{3}$ and the first $A_{1}$ factor corresponds to $\alpha_1$, and the second to $\alpha_{2}$.

When $G$ is simple of type $E_7$ or $E_8$, we will need to distinguish between certain isomorphic subsystem subgroups of $G$. For type $E_7$, there are two conjugacy classes of $A_5$-Levi subgroups, with representatives $A_{5} = L_{24567}$ and $A_{5}' = L_{34567}$. By \cite[Table 8.2]{MR1329942} we have $C_{G}(A_{5})^{\circ} = A_{2}$ and $C_{G}(A_{5}')^{\circ} = A_{1}T_{1}$. Furthermore $A_5'$ is contained in a subgroup $E_6$ whereas $A_5$ is not. In $E_8$ there are two conjugacy classes of subgroups $A_7$, one given by the $A_7$-Levi subgroup $L_{1345678}$, which we denote $A_7$, and one denoted $A_{7}'$ contained in an $E_7$-Levi subgroup of $G$. Again by [\emph{loc.\ cit.}] we have $C_{G}(A_{7})^{\circ} = T_1$, and $C_{G}(A_{7}')^{\circ} = A_{1}$.

\subsection{Embeddings into classical groups.} \label{sec:embed} Finally, for a reductive algebraic group $X$, a rational $X$-module $V$ corresponds to a representation from $X$ into $H = \SL(V)$, $\SO(V)$ or $\Sp(V)$; for brevity, we write `$X \le H$ via $V$' for this. When $V = V_{X}(\lambda)$ is irreducible we take this further and write `$X \le H$ via $\lambda$.' We occasionally talk about chains `$X \le Y \le Z$ via $\lambda$', in which case the embedding refers to the first inclusion $X \to Y$; the inclusion $Y \to Z$ will be clear from context.


\section{Preliminary Results} \label{sec:preliminaries}

\subsection{A strategy for enumerating subgroups} \label{sec:enumerating}

From now on, $G$ denotes a simple algebraic group of exceptional Lie type over the algebraically closed field $K$. To prove Theorem~\ref{THM:MAIN} we follow the strategy developed in \cite{MR3075783} and \cite{Litterick2018}. Each class of non-$G$-cr subgroups has a representative $X$ in some standard parabolic subgroup $P$ of $G$; take $P$ to be minimal subject to containing $X$, and let $X_{0}$ be the image of $X$ in the standard Levi factor $L$ of $P$. By minimality, $X_{0}$ is $L$-irreducible. If $X$ is reductive then $X$ is a complement (as abstract groups) to $Q := R_u(P)$ in the semidirect product $QX_{0}$. The cohomology set $H^{1}(X_0,Q)$ can be used to parametrise $Q$-conjugacy classes of such complements. The following now tells us that non-trivial elements of $H^{1}(X_0,Q)$ give rise to non-$G$-cr subgroups, and non-associated parabolic subgroups give rise to disjoint $G$-conjugacy classes of non-$G$-cr subgroups. Here, two parabolic subgroups are \emph{associated} if their Levi factors are $G$-conjugate to one another.
\begin{lemma} [{\cite[Theorem 5.8 and Proposition 5.14]{MR3042598}, \cite[Lemma~3.26]{Litterick2018}}] \label{lem:associated}
Let $X$ be a closed subgroup of $G$, let $P$ be a parabolic subgroup of $G$ containing $X$, and let $X_0$ be the image of $X$ under projection to a Levi factor of $P$. Then $X$ is $G$-conjugate to $X_0$ if and only if $X$ is $R_u(P)$-conjugate to $X_0$.

If $P_{1}$ and $P_{2}$ are minimal among parabolic subgroups of $G$ containing $X$ then $P_1$ and $P_2$ are associated, and the images of $X$ in their respective Levi factors are $G$-conjugate.
\end{lemma}
The second part of this lemma allows us to break our enumeration of non-$G$-cr subgroups down, one subsection for each pair $(L,X_0)$ with $L$ representing a $G$-conjugacy class of Levi subgroups and $X_0$ an $L$-irreducible subgroup.

At this point we need two essential ingredients. Firstly, we need a description of the $L$-irreducible subgroups $X_{0}$ occurring in each $L$. This has been accomplished explicitly by the present authors in \cite{Litterick2018a} and \cite{ThomasIrreducible}, which give the irreducible subgroups in each simple factor of each Levi subgroup $L$ of $G$; a general $L$-irreducible semisimple subgroup is then a central product of these. Secondly, we must describe the sets $H^{1}(X_{0},Q)$ for each such $L$-irreducible subgroup $X_{0}$, which parametrise the complements to $Q$ in the semidirect product $QX_{0}$, up to $Q$-conjugacy.

\subsection{Cohomology and complements in semidirect products} \label{sec:complements}

Recall that when $X_0$ and $Q$ are algebraic groups with $X_0$ acting on $Q$ via a morphism $X_0 \times Q \to Q$, complements to $Q$ in the semidirect product $QX_0$ are in bijection with rational \emph{1-cocycles}, variety morphisms $\phi \, : \, X_0 \to Q$ such that $\phi(xy) = \phi(x)(x \cdot \phi(y))$ for all $x,y \in X_0$. Here, a \emph{complement} $X$ to $Q$ is by definition a closed subgroup of $QX_0$ satisfying (i) $QX = QX_0$, (ii) $Q \cap X = 1$, and (iii) $L(Q) \cap L(X) = \{0\}$, see \cite[Definition 3.2.1]{MR3105754} and \cite[4.3.1]{MR2753264}. Two cocycles $\phi$, $\psi$ are \emph{cohomologous} if there exists $q \in Q$ such that $\phi(x) = q^{-1}\psi(x)(x \cdot q)$ for all $x \in X_0$. This defines an equivalence relation on the set $Z^{1}(X_0,Q)$ of 1-cocycles, and the corresponding quotient is denoted $H^{1}(X_0,Q)$, and parametrises complements up to $Q$-conjugacy. The set $H^{1}(X_0,Q)$ has a distinguished element, which is the class of the map sending every element of $X_0$ to the identity of $Q$. In general $H^{1}(X_0,Q)$ is only a pointed set, but if $Q$ is a $KX_0$-module then $H^{1}(X_0,Q)$ is naturally a $K$-vector space.

By \cite[Lemma 3.6.1]{MR3105754}, a subgroup satisfying (i) and (ii) automatically satisfies (iii) when $X_0$ is connected reductive, $Q$ is unipotent and $p \neq 2$. A subtlety occurs in characteristic $2$, however. The issue is that closed reductive subgroups may arise as \emph{abstract} complements in semidirect products, and thus the condition on Lie algebras must be relaxed if we are to account for all such subgroups. Fortunately, we have good control over this situation. When $X_0$ is defined over $\mathbb{F}_p \subset K$, we let $Q^{[1]}$ denote the algebraic $X_0$-group obtained by twisting the action of $X_0$ by a Frobenius morphism $X_0 \to X_0$ induced by the $p$-th power field automorphism $K \to K$. Then we get induced maps $Q \to Q^{[1]}$ and $H^{1}(X_0,Q) \to H^{1}(X_0,Q^{[1]})$, and the latter is always injective and is usually an isomorphism of pointed sets \cite[Theorem 3.5.6]{MR3105754}. The following is now a particular case of \cite[Lemma 3.6.1]{MR3105754}.
\begin{lemma} \label{lem:bn-cn}
	Let $X_0$ be a simple algebraic group acting on a unipotent algebraic group $Q$ via a morphism $X_0 \times Q \to Q$, and suppose that $Q$ has a filtration by normal $X_0$-stable subgroups whose successive quotients are $X_0$-modules. Let $X$ be an abstract complement to $Q$ in the semidirect product $QX_0$. Then either:
	\begin{enumerate}
		\item $X$ is a complement to $Q$ as algebraic groups, or
		\item $p = 2$, $X_0$ has type $C_{n}$, $X$ has type $B_n$, and some $X_0$-composition factor in the filtration of $Q$ has high weight $\lambda_1$.
	\end{enumerate}

	In \textup{(ii)}, if the groups and morphisms are all defined over the prime field then $X$ corresponds to an element of $H^{1}(X_0,Q^{[1]})$ which is not in the image of $H^{1}(X_0,Q) \to H^{1}(X_0,Q^{[1]})$. Thus there is a bijection between $H^{1}(X_0,Q^{[1]})$ and the $Q$-conjugacy classes of closed, connected, reductive subgroups of $QX_0$ which are abstract complements to $Q$.
\end{lemma}

Thus to classify non-$G$-cr subgroups of $G$, we need to understand the sets $H^{1}(X_0,Q)$ for various groups $X_0$, and also $H^{1}(X_0,Q^{[1]})$ when $X_0$ has type $C_n$ and $p = 2$. In many cases, this can be reduced to abelian cohomology calculations, using a long exact sequence. In particular:
\begin{lemma}[{\cite[\S I.5, Propositions 36 and 43]{MR1867431}}] \label{lem:exactseq}
Let $Q$ be an algebraic $X_0$-group and $R$ an $X_0$-stable subgroup of $Q$. Then:
	\begin{enumerate}
		\item There is an exact sequence of pointed sets
		\[ 0 \to R^{X_0} \to Q^{X_0} \to (Q/R)^X_0 \to H^{1}(X_0,R) \to H^{1}(X_0,Q) \]
		where $-^{X_0}$ denotes the fixed-point subgroup under the action of $X_0$.
		\item If moreover $R$ is central in $Q$, then there is an exact sequence of pointed sets
		\[ 0 \to R^{X_0} \to Q^{X_0} \to (Q/R)^{X_0} \to H^{1}(X_0,R) \to H^{1}(X_0,Q) \to H^{1}(X_0,Q/R) \to H^{2}(X_0,R) \]
		where $H^{2}(X_0,R)$ is the second cohomology group defined for example in \cite[\S II.4.2]{MR2015057}.
	\end{enumerate}
\end{lemma}
Note that these are only exact sequences of pointed sets, meaning that the image of each map is the pre-image of the distinguished element under the next. However, this is already sufficient to deduce much information about $H^{1}(X_0,Q)$. For instance, it also allows us to relate $M$-complete reducibility with $G$-complete reducibility, when $X \le M \le G$ is a chain of reductive subgroups:
\begin{corollary} \label{cor:mcrgcr}
Let $M \le G$ be two reductive algebraic groups, and let $P_M = Q_M L_M$ be a parabolic subgroup of $M$ containing a reductive subgroup $X$, such that $X$ lies in no Levi factor of $P_M$, and write $X_0$ for the image of $X$ in $L_M$. Furthermore, let $P = QL$ be a parabolic subgroup of $G$ such that $P_M = M \cap P$ and $Q_M = M \cap Q$.

If $(Q/Q_M)^{X_0} = \{0\}$ then $X$ is non-$G$-cr.
\end{corollary}

\proof The hypotheses imply that the map $(Q/Q_M)^{X_0} \to H^{1}(X_0,Q_M)$ is zero, and since $X$ is non-$M$-cr, the latter set contains non-zero points. By exactness, these do not lie in the kernel of $H^{1}(X_0,Q_M) \to H^{1}(X_0,Q)$, and since $X$ corresponds to such a point, $X$ remains non-$G$-cr. \qed

\subsection{Approximating $H^{1}(X_0,Q)$} \label{sec:rho}

We now recall a method, developed in \cite{MR3075783} and applied extensively in \cite{Litterick2018}, for approximating $H^{1}(X_0,Q)$. When $Q$ is the unipotent radical of a parabolic subgroup $P$, recall that $Q$ has a filtration by normal subgroups $Q(i)$ $(i = 1,\ldots,r)$. We write $V_i = Q(i)/Q(i+1)$ for the corresponding levels, which are modules for a Levi subgroup $L$ of $P$ hence also for $X_0$ whenever $X_0 \le L$. One can therefore form the direct sum
\[ \mathbb{V} = \mathbb{V}_{X_0,Q} := \bigoplus_{i = 1}^{r} H^{1}(X_0,V_i). \]
For $i = 1,\ldots,r$ one now defines partial maps $\rho_i : \mathbb{V} \to H^{1}(X_0,Q/Q(i+1))$ inductively \cite[Definition~3.2.5]{MR3075783}. Let $\mathbf{v} = ([\gamma_1],\ldots,[\gamma_r]) \in \mathbb{V}$. For $i = 1$ let $\rho_1(\mathbf{v}) = [\gamma_1]$. For $i > 1$, lift $\rho_{i-1}(\mathbf{v})$ to some element $[\Gamma_i] \in H^{1}(X_0,Q/Q(i+1))$ under the natural map $H^{1}(X_0,Q/Q(i+1)) \to H^{1}(X_0,Q/Q(i))$, \emph{when this is possible}, and then set $\rho_i(\mathbf{v}) = [\Gamma_i][\gamma_i]$. As long as $[\Gamma_i]$ exists, this product is well-defined since $Q(i)/Q(i+1)$ is central in $Q/Q(i+1)$. If such a lift $[\Gamma_i]$ does not exist, declare $\rho_i$ undefined at $\mathbf{v}$.

Each of these partial maps $\rho_i : \mathbb{V} \to H^{1}(X_0,Q/Q(i+1))$ turns out to be surjective (cf.~\cite[Proposition 3.2.6]{MR3075783}, \cite[Lemma 3.10]{Litterick2018}), and we set $\rho = \rho_r$. As a particular corollary, if $H^{1}(X_0,V_i) = \{0\}$ for all levels $V_i$ then $H^{1}(X_0,Q) = \{0\}$, and all complements to $Q$ in $QX$ are $Q$-conjugate to $X$. Similarly, in the setting of Lemma~\ref{lem:bn-cn}(ii) with $X_0$ of type $C_n$, if $H^{1}(X_0,V_i^{[1]}) = \{0\}$ for all $i$ then no (abstract) complements of type $B_n$ arise.

Finally, $Q$ acts on itself, and if $q \in Q^{X_0}$ then $q$ serves to fuse together classes corresponding to different elements of $\mathbb{V}$. A full account is given in \cite[pp.~5292--5294]{Litterick2018}. Complete details will be given when we make use of this action in our proofs; for the moment, we note only that if $q$ maps to a point in $V_i^{X_0}$, then conjugation by $q$ induces $X_0$-module homomorphisms $V_j \to V_{i+j}$ and linear maps $H^{1}(X_0,V_j) \to H^{1}(X_0,V_{i+j})$ for each $j$, and thereby a map $\mathbb{V} \to \mathbb{V}$. Typically, this will fuse together classes in such a way that a subgroup class represented by $(k_1,\ldots,k_m) \in \mathbb{V}$ will also be represented by an element with $k_i k_j = 0$ for particular indices $i$ and $j$.

\subsection{From $Q$-conjugacy to $G$-conjugacy} \label{sec:qgconj}

The theory so far describes non-$G$-cr semisimple subgroups in parabolic subgroups $P = QL$ \emph{up to $Q$-conjugacy}; this is not the end of the story as elements of $G$ can fuse classes together. We recall some details from \cite[Section 3.5]{Litterick2018}.

Firstly, the torus $Z(L)$ acts on $Q$ and centralises each subgroup $X_0 \le L'$, inducing an action on $H^{1}(X_0,Q)$. Fixing a basis of each space $H^{1}(X_0,V_i)$ where $V_i$ is a level of $Q$, the map $\rho$ allows us to parametrise complements to $Q$ in $QX_0$ by certain $m$-tuples $(k_1,k_2,\ldots,k_m) \in \mathbb{V}$ for some $m$. The definition of $\rho$ involves many arbitrary choices of lift, but subject to appropriate choices, the action of $Z(L)$ can be viewed as multiplying the coordinates $k_i$ by appropriate scalars, fusing together all of the corresponding subgroup classes. The action of $Z(L)$ is fixed-point-free on the non-identity elements of $Q$ since $C_{G}(Z(L)) = L$, and thus this action will be non-trivial in each coordinate of $\mathbb{V}$.

Also, elements of $G$ can fuse classes of subgroups occurring in standard parabolic subgroups $P_I$ and $P_J$, for not necessarily distinct subsets $I$ and $J$ of simple roots. By Lemma~\ref{lem:associated}, the corresponding Levi subgroups $L_I$ and $L_J$ are $G$-conjugate in this case. In particular, we will be concerned with elements $w$ of the Weyl group $W(G) = N_G(T)/T$, with pre-image $\dot{w} \in N_G(T)$, such that ${\dot w} \cdot L_I = L_J$. In this case, $\dot{w}$ sends each root subgroup in the unipotent radical of $P_{I}$ to another root subgroup of $G$, which may or may not lie in $P_J$. Thus if we have a subgroup $X \le P_I$, and if $X \le U L_I$ for some subgroup $U \le Q_{I}$ with $\dot{w} \cdot U \le Q_{J}$, then $X$ is $G$-conjugate to a subgroup of $P_{J}$. Again, when we use arguments of this form in Sections~\ref{sec:A3}--\ref{sec:D4}, we will give full details of the relevant Weyl group elements and how they fuse classes together.

\subsection{Lie types of non-$G$-cr subgroups} \label{sec:subtypes}

We now describe those Lie types $\mathbf{X}$ of semisimple groups which admit non-zero cohomology on a relevant module.

\begin{lemma} \label{lem:subtypes}
Let $G$ be a simple algebraic group of exceptional type in characteristic $p$, such that $G$ has a non-$G$-cr semisimple subgroup of type $\mathbf{X}$ with simple components each of rank $3$ or more. Then $(G,\mathbf{X},p)$ appears in Table \ref{tab:subtypes}.
\end{lemma}

\begin{table}[htbp]
\caption{Types of semisimple non-$G$-cr subgroups} \label{tab:subtypes}
\centering
\begin{tabular}{
r|cccc
}
$\mathbf{X}$ \textbackslash\ $G$	& $E_{8}$		& $E_{7}$	& $E_{6}$	& $F_{4}$	\\ \hline
$A_{3}$, $C_4$, $D_4$	& $2$ 			& $2$		&			&			\\
$B_{3}$ 	 			& $2$ 			& $2$ 		& $2$ 		& $2$ 		\\
$C_{3}$ 	 			& $3$ 			&			&			&			\\
$B_{3}^2$, $B_{4}$ 		& $2$			&			&			& 
\end{tabular}
\end{table}

\begin{remark}
In the course of proving Theorem~\ref{THM:MAIN} we will prove the converse to Lemma~\ref{lem:subtypes}, that non-$G$-cr subgroups of each type do indeed exist. This was shown already in \cite{Stewart21072013} for simple types, and exhibiting a non-$G$-cr subgroup of type $B_3^2$ (Section~\ref{sec:B3}) completes the argument.
\end{remark}

\begin{proof} The simple types $\mathbf{X}$ which occur are given in \cite[Theorem 1]{Stewart21072013}, so we need only consider types with more than one component. For such a type $\mathbf{X}$, suppose that $G$ has a parabolic subgroup $P = QL$ with an $L$-irreducible subgroup $X_0$, also of type $\mathbf{X}$ (or of type $C_{n}$ when $\mathbf{X} = B_n$), with $H^{1}(X_0,Q) \neq \{0\}$. From the exact sequence of cohomology (Lemma~\ref{lem:exactseq}), we see that if $H^{1}(X_0,Q(i)/Q(i+1)) = \{0\}$ for each $i$, then $H^{1}(X_0,Q) = \{0\}$. Thus for a non-$G$-cr subgroup of type $\mathbf{X}$ to exist, we need $H^{1}(X_0,V) \neq \{0\}$ for some indecomposable $X_0$-module direct summand $V$ of a level of $Q$.

In \cite{Litterick2018a} and \cite{ThomasIrreducible}, an explicit list is given of the irreducible subgroups in each simple factor of a Levi subgroup of $G$. These are given up to conjugacy in the simple factor, except for factors of type $D_7$ in \cite{Litterick2018a} when conjugacy under a graph automorphism is also allowed. Given this, in each Levi subgroup $L$ of $G$, the $L$-irreducible semisimple subgroups are precisely the products of irreducible subgroups in each simple factor. In summary, the non-simple, semisimple $L$-irreducible subgroups $X_0$ arising in this way are as follows.

\begin{center}
\begin{tabular}{c|l}
$L$ & $X_0$ \\ \hline
$A_{3}^2$ & $X_0 = L'$ \\
$A_{3} A_{4}$ & $X_0 = L'$ \\
$D_{6}$ & $A_{3}^{2}$ via $(010,0) + (0,010)$ \\
$D_{7}$ & $A_{3}D_{4}$ via $(010,0) + (0,\lambda_1)$ \\
& $B_{3}^{2}$ via $(100,0) + (0,100)$ $(p \neq 2)$ or $0|((100,0) + (0,100))|0$ $(p = 2)$ \\
& $A_{3}B_{3}$ via $(010,0) + (0,100) + 0$ $(p \neq 2)$ or $(010,0) + (0,T(100))$ $(p = 2)$ \\
& $A_{3}B_{3}$ via $(010,0) + (0,001)$
\end{tabular}
\end{center}
It remains to show that $H^{1}(X_0,V) = \{0\}$ for each such $X_0$ and each relevant $V$, except for $X_0$ of type $B_3^2$ with $p = 2$. By the K\"{u}nneth formula for cohomology, if $U_i$ is a module for a reductive group $Y_i$ $(i = 1,2)$ then $H^{1}(Y_1 \times Y_2, U_1 \otimes U_2) \cong \bigoplus_{j = 0,1}H^{j}(Y_1,U_1) \otimes H^{1-j}(Y_2,U_2)$.

As an $L$-module, each level of $Q$ is a direct sum of shape modules $V_S$ as described in Section~\ref{sec:shape-modules}. The structure of such shape modules is known from prior work, e.g.~\cite[Lemma~3.1]{MR1329942}. Explicitly, up to taking duals, for a simple factor of $L$ of type $A_n$ we obtain irreducible modules $\lambda_1$, $\lambda_2$, $\lambda_3$; for type $D_n$ $(n \ge 4)$ we obtain $\lambda_1$, $\lambda_{n-1}$ and $\lambda_n$; for type $B_3$ we obtain $W(100)$ and the irreducible module $001$; for type $C_3$ we obtain $W(010)$; for type $E_6$ we obtain $\lambda_1$; and for type $E_7$ we obtain $\lambda_7$.

Once these shape modules $V_S$ are known, determining the action of the $L$-irreducible subgroups $X_0$ on $V_S$ is usually a matter of routine weight-space calculations. When $X_0$ is not simple, each simple factor of $X_0$ is contained in a proper Levi subgroup of $L$ (hence of $G$) whose derived subgroup is simple of type $A_{3}$, $A_{4}$ or $D_{4}$. The non-zero shape modules occurring have dimension at most $10$, and each $L$-module occurring in the filtration of $Q$ then restricts to $X_0$ as a tensor product of modules for the factors of $X_0$. In particular, in each decomposition $U_{1} \otimes U_{2}$ above, each factor $U_{i}$ has dimension at most $10$.

For most irreducible $X_0$-modules $V$ of relevance to us, the first cohomology group $H^{1}(X_0,V)$ has already appeared in the literature. For example, \cite[Lemma~4.6]{Stewart21072013} collates many such results. Generally, such cohomology groups can be determined from the structure of the corresponding Weyl module and a dimension-shifting argument, see \cite[Proposition 4.14]{MR2015057} and the subsequent remark there. In turn, the structure of Weyl modules can be computed either using S.\ Doty's GAP routine \cite{Dot1} or by using the Weyl character formula together with the known weights of irreducible modules of low dimension \cite{MR1901354}. As an elementary example, the Weyl module $W_{B_3}(100)$ has structure $100|0$, and dimension-shifting gives $H^{1}(B_3,100) \cong H^{0}(B_3,0) \cong K$.

In the end, for a simple factor $Y$ of $X_0$, we find that the only irreducible $Y$-modules $U$ with non-zero cohomology (and dimension at most $10$) occur for factors of type $B_{3}$ and $U \cong V_{B_3}(100)$. Now, with the exception of $B_{3}^{2} < D_{7}$ when $p = 2$, each such factor $B_{3}$ is contained in a proper Levi subgroup of type $D_{4}$, acting on the three $8$-dimensional modules as a tilting module $T(100)$ or $001$. Such a module has zero cohomology \cite[Theorem 1.1]{DonkinFiltration}, and thus if a non-$G$-cr subgroup $X$ of type $\mathbf{X}$ exists, then $p = 2$ and $X = B_{3}^2 < D_{7}$.
\end{proof}

Now that we have determined the possible types of non-$G$-cr semisimple subgroups $X$ occurring, we next describe the relevant embeddings $X_0 \to L$ for Levi subgroups $L$ of $G$ with $L$-irreducible image; this follows directly from the main results of \cite{Litterick2018a} and \cite{ThomasIrreducible}.

\begin{lemma} \label{lem:gcrsubs}
With $(G,\mathbf{X},p)$ as in Table~\ref{tab:subtypes}, let $L$ be a Levi subgroup of $G$ and $X_0 \le L$ be an $L$-irreducible semisimple subgroup, where $X_0$ has type $\mathbf{X}$, or $X_0$ has type $C_n$ when $\mathbf{X} = B_n$ $(n = 3 \text{ or } 4)$. Then either:
\begin{enumerate}[label=(\roman*)]
	\item $p = 3$ and $(X_0,L') = (C_3,D_7)$ with $H \to L'$ via $010$, or
	\item $p = 2$ and the inclusion $X_0 \to L'$ is described in Table~\ref{tab:gcrsubs}.
\end{enumerate}
\end{lemma}

\begin{table}[htbp]
\caption{Levis containing relevant irreducible subgroups $X_0$ when $p=2$} \label{tab:gcrsubs}
\begin{minipage}{0.56\linewidth}
\small
\begin{tabular}{c|l}
$L'$ & Embedding $X_0 \to L'$ \\ \hline
$B_{3}$ & $A_{3}$ via $010 + 0$ \\
$C_{3}$ & $A_{3}$ via $010$ \\
$D_{4}$ & $B_{3}$ via $T(100)$ \\
		& $B_{3}$ via $001$ (2 classes) \\
$A_{5}$ & $C_{3}$ via $100$ \\
		& $A_{3}$ via $010$ \\
$D_{5}$ & $B_{4}$ via $T(1000)$ \\
$D_{6}$ & $A_{3}$ via $010 + 010^{[r]}$ ($r \neq 0$; $2$ classes) \\
$A_{3}^2$ & $A_{3}$ via $(100^{[r]},100^{[s]})$ or $(100^{[r]},001^{[s]})$ $(rs=0)$
\end{tabular}
\end{minipage}
\begin{minipage}{0.36\linewidth}
\small
\begin{tabular}{c|l}
$L'$ & Embedding $X_0 \to L'$ \\ \hline
$E_{6}$	& $C_{4} < F_4$  \\
		& $D_{4} < C_4 < F_4$  \\
$A_{7}$ & $C_{4}$ via $1000$ \\
		& $D_{4}$ via $1000$ \\
		& $B_{3}$ via $001$ \\
$D_{7}$ & $C_{3}$ via $010 + 0$ \\
		& $B_{3}^{2}$ via $0|((100,0) + (0,100))|0$ \\
		& $B_{3}$ via $0|(100 + 100^{[r]})|0$ $(r \neq 0)$ \\
		& $A_{3}$ via $101$
\end{tabular}
\end{minipage}
\end{table}

Finally, with Lemma~\ref{lem:gcrsubs} telling us which parabolic subgroups $P = QL$ to focus on, we consider the action of each relevant $L$-irreducible subgroup $X_0$ on the unipotent radical $Q$; in particular we need to know how $X$ acts on each shape module $V_S$ occurring in the levels of $Q$.

\begin{proposition} \label{prop:h1}
Let $X_0$ be a $G$-cr semisimple subgroup of $G$ of type $\mathbf{X}$ as in Lemma~\ref{lem:gcrsubs}, and let $P = QL$ be a parabolic subgroup of $G$, where $X_0 \le L$ is $L$-irreducible. If $H^{1}(X_0,V^{[1]}) \neq \{0\}$ for some $X_0$-summand $V$ of a shape module $V_S$ occurring in a level of $Q$, then either:
\begin{enumerate}[label=(\roman*)]
	\item $p = 3$ and $(X_0,L',V_S,V) = (C_3,D_7,V_{D_7}(\lambda_1),V_{C_3}(010))$ with $\dim H^{1}(X_0,V) = 1$, or
	\item $p = 2$ and $X_0$, $L'$, $V_S$ and $V$ appear in Table~\ref{tab:h1}.
\end{enumerate}
In (ii), if $(X_0,V) = (D_4,V_{D_4}(\lambda_2))$ then $\dim H^{1}(X_0,V) = 2$, otherwise $\dim H^{1}(X_0,V^{[1]}) = 1$. In the latter case we also have $\dim H^{1}(X_0,V) = 1$ unless $X_0$ has type $C_n$ and $V_S \downarrow X_0$ has a high weight $\lambda_1$ (cf.~Lemma \ref{lem:bn-cn}).
\end{proposition}

\begin{table}[htbp]
\caption{Shape modules with non-zero cohomology, $p = 2$} \label{tab:h1}
{\centering \small
\begin{tabular}{c*{4}{|c}}
$G$ 	& $L'$	& $X_0 \le L$			& $V_S \downarrow L'$				& $V_S \downarrow X_0$ \\ \hline
$F_4$ 	& $B_3$ & $X_0 = L'$						& $100|0$	& $100|0$ \\
	 	& $C_3$ & $X_0 = L'$						& $001|100$ 	& $001|100$ \\ \hline
$E_6$ 	& $A_5$ & $C_{3}$ via $100$				& $\lambda_1, \lambda_5$	& $100$ \\ \hline
$E_7$ 	& $A_5$ & $C_{3}$ via $100$				& $\lambda_1, \lambda_5$	& $100$ \\
		&		&								& $\lambda_3$ 				& $100|001|100$ \\
	 	& 		& $A_{3}$ via $010$				& $\lambda_2, \lambda_4$  	& $101 + 0$ \\
	 	& $E_6$ & $C_4 < F_4$ 		& $\lambda_1, \lambda_6$	& $0100 + 0$ \\
	 	& 		& $D_4 < C_4 < F_4$				& $\lambda_1, \lambda_6$ 	& $0100 + 0$ \\ \hline
$E_8$ 	& $A_5$ & $C_{3}$ via $100$				& $\lambda_1, \lambda_5$	& $100$ \\
		&		&								& $\lambda_3$ 				& $100|001|100$ \\
	 	& 		& $A_{3}$ via $010$				& $\lambda_2, \lambda_4$ 	& $101 + 0$ \\
	 	& $A_3^2$ & $A_3$ via $(100,100^{[1]})$ & $(010,100)$  & $210$ \\
	 	&		&	$A_3$ via $(100^{[1]},100)$ & $(100,010)$	& $210$  \\
	 	& $E_6$ & $C_4 < F_4$ 					& $\lambda_1, \lambda_6$	& $0100 + 0$ \\
	 	& 		& $D_4 < C_4 < F_4$				& $\lambda_1, \lambda_6$ 	& $0100 + 0$ \\
	 	& $A_7$ & $C_4$ via $1000$				& $\lambda_1$ 	& $1000$ \\
	 	&		& $D_4$ via $1000$				& $\lambda_2$ 	& $0|0100|0$ \\
	 	&		& $B_3$ via $001$				& $\lambda_2$ 	& $0|100|010|100|0$  \\
	 	& $D_7$ & $B_3^2$ via $0|((100,0) + (0,100))|0$	& $\lambda_1$ 	& $0|((100,0) + (0,100))|0$ \\
	 	& 		& $B_3 < B_3^2$ via $0|(100 + 100^{[r]})|0$ $(r \neq 0)$	& $\lambda_1$ 	& $0|(100 + 100^{[r]})|0$ \\
	 	& 		& $A_3$ via $101$				& $\lambda_1$ 	& $101$
\end{tabular}
}
\end{table}

\begin{remark} \label{rem:nearly-all-exist}
In the course of proving Theorem~\ref{THM:MAIN} we will show, with two exceptions, that whenever $X_0$ has non-zero cohomology on a module $V$, this does in fact give rise to a non-$G$-cr subgroup of $G$. The exceptions occur in the $A_3^2$-parabolic subgroups of $E_8$, where $X_0$ has type $A_3$ and $V = V_{X}(210)$ occurs as a summand of a shape module, but $H^{1}(X_0,Q)$ nevertheless vanishes. This is caused by a shape module with non-zero second cohomology group, which by Lemma~\ref{lem:exactseq} obstructs the lifting of cocycles from a level of $Q$ to $Q$ itself.
\end{remark}

\begin{proof}[Proof of Proposition~\ref{prop:h1}]
As discussed in the proof of Lemma~\ref{lem:subtypes}, each shape module $V_S$ occurring is a Weyl module with a small range of possible high weights. In most cases, determining the action $V_S \downarrow X_0$ is a routine weight-space calculation, given the known embedding of $X_0$ into $L'$. For instance, when $L'$ has type $A_n$, all modules lie in the tensor algebra of the natural module $\lambda_1$, and thus $V_S \downarrow X_0$ can be calculated by applying constructions to $\lambda_1 \downarrow X_0$ such as taking tensor products, summands etc. In other cases, the weights of $X_0$ on $V_{L'}(\lambda)$ can be determined by identifying a maximal torus of $X_{0}$ as a sub-torus of a maximal torus of $L'$. In yet further cases, the action follows from previous work of Seitz \cite[Theorem 1]{MR888704}. For example, in the case of $C_3 < D_7$ via $010$ when $p = 2$, the action of $C_3$ on the shape module $V_S = V_{D_7}(\lambda_6)$ is given by \cite[Case IV$_8$, p.~283]{MR888704}. Similarly, for $A_3 < D_7$ via $101$ when $p = 2$, the action of $X_0$ on $V_{D_7}(\lambda_6)$ is given by \cite[Case S$_7$, p.~283]{MR888704}.

We now give details of two of the most involved cases, both with $G$ of type $E_8$. When $X_0$ has type $B_3$, embedded in $A_7$ via $001$ when $p = 2$, the three shape modules occurring have high weights $\lambda_1$, $\lambda_3$ and $\lambda_6$, hence restrict to $X_0$ as $001$, $\wedge^3(001)$ and $\wedge^2(001)$ respectively. Now $001$ is tilting for $X_0$ and therefore has trivial cohomology. The exterior cube $\wedge^3(001)$ is self-dual of dimension $56$ and its weights are sums of triples of pairwise-distinct weights of $001$; it follows that $\wedge^3(001)$ has a high weight $101$. Inspecting \cite{MR1901354}, we find that $W_{X_0}(101)$ is irreducible of dimension $48$ in characteristic $2$; the highest remaining weight is $001$. Since $101$ and $001$ are tilting in characteristic $2$, it follows that $\wedge^3(001)$ splits as a direct sum $101 + 001$, and has trivial cohomology. Finally, similar calculations show that $\wedge^2(001)$ has composition factors $0^2/100^2/010$. In this case, the module turns out to be uniserial with socle series $0|100|010|100|0$. General theory tells us, for instance, that $\wedge^2(001)$ admits a unique non-zero $X_0$-homomorphism to $K$, up to scalars, since $001$ supports a unique nondegenerate alternating bilinear form up to similarity. However the most straightforward way to determine the submodule structure is computationally, working with explicit matrices generating a sufficiently large finite subgroup such as $\SO_{7}(4) \le X_0$, for which the module will still be uniserial.

Now that we know $\wedge^2(001) = 0|100|010|100|0$, we can work out its first cohomology group. The inclusion $100|0 \to \wedge^2(001)$ gives a long exact sequence of cohomology, part of which is
\[ H^{0}(X_0,0|100|010) \to H^{1}(X_0,100|0) \to H^{1}(X_0,\wedge^2(001)) \to H^{1}(X_0,0|100|010). \]
The first group here is zero, as is the right-hand group since $0|100|010$ is the dual of a Weyl module \cite[Lemma~7.1.2(ii)]{MR2044850} and such modules have zero first cohomology \cite[Theorem 1]{MR804233}. Thus the middle groups are isomorphic, and $H^{1}(X_0,100|0) = K$ since $T(100) = 0|100|0$, which has zero cohomology [\emph{loc.\ cit.}].

For a final example, again with $G$ of type $E_8$, consider $X_0$ of type $B_3^2$, embedded into $L'$ of type $D_7$ via the module $V = 0|( (100,0) + (0,100)) | 0$. Then $\dim H^{1}(X_0,V) = 1$: Consider the short exact sequence $\{0\} \to (100,0)|0 \to V \to 0|(0,100) \to \{0\}$ and the associated long exact sequence. We have $H^{0}(X_0,0|(100,0)) = H^{1}(X_0,0|(100,0)) = \{0\}$, the latter since it is the dual of a Weyl module. It follows that $H^{1}(X_0,(100,0)|0) \cong H^{1}(X_0,V) \cong K$, as claimed. Finally, the module $V_{\lambda_6}(\lambda_6) \downarrow X_0$ is a tensor product of spin modules for each factor, by \cite[Proposition 2.7]{MR1329942}, in particular it is irreducible and has zero cohomology, so does not appear in Table~\ref{tab:h1}.
\end{proof}

We single out a particular case of Proposition~\ref{prop:h1}, both because it is a more involved calculation, and because we require additional details when classifying subgroups of type $B_3$ in Section~\ref{sec:B3}.
\begin{lemma} \label{lem:c3-wedge-cube-h1}
Let $p = 2$, let $X_0$ be simple of type $C_3$ and let $M = \bigwedge^{3} V_{X_0}(100)$. Then $H^{1}(X_0,M)$ vanishes, and $H^{1}(X_0,M^{[1]})$ is $1$-dimensional. The inclusion $200 \to M^{[1]}$ induces an isomorphism of first cohomology groups, and the map $M^{[1]} \to 200$ induces the zero map on the first cohomology groups.
\end{lemma}

\begin{proof}
The first statement holds since $M = 100|001|100$, which has no composition factors with a non-zero first cohomology group. Note, however, that $\dim H^{1}(X_0,200) = 1$. Now, direct computation (e.g.~using \cite{Dot1}) shows that the Weyl module $W(002) = 002|200|020|200|000|101|000$, a uniserial $X_0$-module. Thus $W(002)$ has no indecomposable quotient with composition factors $002/200/000$. Taking duals, we see that $H^{1}(X_0,200|002)$ vanishes.

Next, the short exact sequence $\{0\} \to 200 \to M^{[1]} \to 200|002 \to \{0\}$ induces an exact sequence $\{0\} \to H^{1}(X_0,200) \to H^{1}(X_0,M^{[1]}) \to H^{1}(200|002) = \{0\}$, which shows that $H^{1}(X_0,M^{[1]}) \cong H^{1}(X_0,200) \cong K$. This also shows that the map $200 \to M^{[1]}$ induces an isomorphism on first cohomology groups.

Finally, the submodule $002|200$ of $M^{[1]}$ has one-dimensional first cohomology group, which we see from the short exact sequence $\{0\} \to 200 \to 002|200 \to 002 \to \{0\}$ inducing $\{0\} \to K \to H^{1}(X_0,002|200) \to H^{1}(X_0,002) = \{0\}$. Then the exactness of $\{0\} \to H^{1}(X_0,002|200) \to H^{1}(X_0,M^{[1]}) \to H^{1}(X_0,200)$ shows that the map $M^{[1]} \to 200$ induces the zero map on cohomology groups.
\end{proof}

\subsection{Exhibiting non-$G$-cr subgroups} \label{sec:exhibiting}

In many places, studying first cohomology will limit the number of classes of non-$G$-cr subgroups of a given type occurring. Here, we collect results which will be useful for showing that such subgroups do in fact exist. In places, we will be aware of non-$M$-cr subgroups when $M$ is a proper reductive subgroup of $G$, and it will be helpful to know when such a subgroup is also non-$G$-cr; this is not automatic.

\begin{lemma}[{\cite[Lemma~2.6 and Corollary~3.21]{MR2178661}}] \label{lem:BMR}
Let $S$ be a linearly reductive subgroup of a connected reductive algebraic group $G$. Then $S$ is $G$-cr, and if $H = C_G(S)^\circ$ then a subgroup of $H$ is $H$-cr if and only if it is $G$-cr.
\end{lemma}

In the particular case that $S$ is a torus, $C_G(S)$ is a Levi subgroup of $G$ and the result in this case was first proved in \cite[Proposition~3.2]{MR2167207}.

The following is a minor generalisation of an example of M.\ Liebeck \cite[Example 3.45]{MR2178661}.
\begin{proposition} \label{prop:orthog_sum}
Let $X$ be a group with a faithful, irreducible module $V$ of even dimension in characteristic $p = 2$, supporting a nondegenerate bilinear form.

Then the orthogonal direct sum $W = V \perp V$ realises $X$ as a non-$H$-cr subgroup of $H = \Sp(W)$, and of $\SO(W)$ if $V$ is orthogonal.
\end{proposition}

\begin{proof}
We need only prove that $X$ is non-$H$-cr. Write $V'$, $V''$ for the left and right summands of $W$ above, and fix an isometric $X$-module isomorphism $\phi : V' \to V''$. Then the proper, non-zero $X$-stable subspaces of $W$ (other than $V''$ itself) all have the form $V_\lambda$, where
\[ V_\lambda = \{ v + \lambda \phi(v) \, : \, v \in V' \}. \]
Now, since $\phi$ is an isometry and $V'$ is orthogonal to $V''$, we have $(v + \lambda\phi(v),w + \lambda \phi(w)) = (1 + \lambda^2)(v,w)$. Since the bilinear form is nondegenerate and $p = 2$, this expression is zero for all $v,w \in V$ if and only if $\lambda = 1$. Thus $X$ preserves a unique nonzero totally isotropic subspace of $W$, hence $X$ lies in a unique parabolic subgroup of $H = \Sp(W)$ and is therefore non-$H$-cr. Similarly, if $V$ supports a nondegenerate quadratic form $q$ then $q(v + \lambda\phi(v)) = (1 + \lambda^2)q(v)$ for all $v \in V'$, and again we find that $V_{1}$ is the unique nonzero $X$-stable totally singular subspace of $W$, so $X$ is non-$\SO(W)$-cr.
\end{proof}

\begin{remark} \label{rem:a3d6-d4d8}
We will need this proposition in three situations: when $X = A_3 < H = D_6$ via $010 + 010$, when $X = B_3 < H = D_8$ via $001 + 001$ and when $X = D_4 < H = D_8$ via $1000 + 1000$. In each case this leads to two non-$H$-conjugate, non-$G$-cr subgroups of $H$, as can be seen for instance via their inequivalent actions on the half-spin modules for $H$; these subgroups will be conjugate under an outer automorphism of $H$. In general, one can prove using the orbit-stabiliser theorem, and the fact that $X$ is unique up to conjugacy in the full isometry group of the form on $W = V \perp V$, that either one or two $H$-conjugacy classes of subgroups will occur in this way.
\end{remark}

\subsection{Calculating centralisers} \label{sec:centralisers}

We now discuss some details involved in calculating $C_G(X)^{\circ}$ for each $X$ occurring in Theorem~\ref{THM:MAIN}. Let $P$ be minimal among parabolic subgroups of $G$ containing $X$, with Levi decomposition $P = QL$, and let $X_0$ be its image in $L$. To begin, note that elements of $X$ have the form $\phi(x)x$ for some cocycle $\phi : X_0 \to Q$ and $x \in X_0$. Thus when $Q$ is abelian, or more generally when the image $\phi(X)$ is central in $Q$, the fixed-point subgroups $Q^X$ and $Q^{X_0}$ coincide. In this case, $Q^{X_0}$ provides a connected unipotent subgroup of $C_G(X)$. Also when $Q$ is abelian, the group $Q^{X_0}$ itself is straightforward to determine from knowledge of its weights as an $X_0$-module, using \cite{MR1047327}. When $Q$ is not abelian, however, two issues arise: The group $Q^{X_0}$ need not simply correspond to the trivial $X_0$-composition factors in each level, and also the image $\phi(X_0)$ need not commute with $Q^{X_0}$, so determining $Q^{X}$ is more subtle.

To overcome the first issue, one can use the fact that $X_0$ is $G$-cr, which places useful constraints on the structure of $C_G(X_0)$; as discussed for instance in \cite[\S 6.2]{Litterick2018a}. In many cases of interest to us, Liebeck and Seitz \cite[p.\ 333]{MR1274094} have already determined $C_G(X_0)^{\circ}$, and we are able to use this to determine $Q^{X_0}$ and hence $Q^{X}$.

For the issue regarding the relationship between $Q^{X}$ and $Q^{X_0}$, this often comes down to showing that $\phi(X_0)$ commutes with specific elements of $Q^{X_0}$, for instance by using the Chevalley commutator formula. Frequently, this reduces to the fact that some subgroup $Q(j)$ is abelian and contains both $Q^{X_0}$ and $\phi(X_0)$. The most intricate calculations occur when $G$ has type $E_8$ and $X$ has type $B_3$, in which case we explicitly construct the relevant cocycles $\phi$ in terms of root elements of $G$ (Proposition~\ref{prop:B3inE8gens}), and thereby determine $Q^{X}$ directly.

The following will also be useful in calculating $C_{G}(X)^{\circ}$ for each subgroup $X$ in Theorem~\ref{THM:MAIN}.

\begin{lemma} \label{lem:rankofcents}
Let $G$ be a reductive algebraic group of rank $r$ with non-$G$-cr connected subgroup $X$, contained minimally in a parabolic subgroup $P = QL$ with $L$ of semisimple rank $s$. Then $C_G(X)$ has rank at most $r-s-1$.  
\end{lemma}

\begin{proof}
Suppose that $S \le C_G(X)$ is a torus of rank $r-s$. To obtain a contradiction, it suffices to prove that $X$ is contained in a parabolic subgroup of $G$ whose Levi subgroups have semisimple rank less than $s$, since all minimal parabolic overgroups of $X$ are associated (Lemma \ref{lem:associated}).    

Let $M = C_G(S)$. Then $M$ is a Levi subgroup of $G$ containing $X$. Since $X$ is non-$G$-cr, it follows from Lemma \ref{lem:BMR} that $X$ is non-$M$-cr. In particular, $X < P_M = Q_M L_M$ where $P_M$ is a proper parabolic subgroup of $M$ and $L_M$ is a Levi subgroup of $M$ with semisimple rank less than $s$. It follows from \cite[Propositions~2.6.6, 2.6.7]{MR794307}, that there exists a parabolic subgroup $\hat P$ of $G$ such that $P_M < \hat P$ and $L_M$ is a Levi subgroup of $\hat{P}$. Since $X < P_M < \hat{P}$, we have reached a contradiction. 
\end{proof}

\begin{lemma} \label{lem:unip-cent}
Let $X$ be a reductive subgroup of $G$. Then there exists a parabolic subgroup $P$ with Levi decomposition $P = QL$, minimal subject to containing $X$, such that $Q^{X}$ is a maximal connected unipotent subgroup of $C_{G}(X)$.

Furthermore, let $X_0$ be the image of $X$ under the projection $P \to L$, and let $n$ be the total number of trivial $X_0$-composition factors across all the levels $Q(i)/Q(i+1)$, $i \ge 0$. Then $\dim Q^{X} \le n$.
\end{lemma}

\begin{proof}
The existence of $P$ follows directly from the Borel--Tits Theorem; if $U$ is a maximal connected unipotent subgroup of $C_{G}(X)$ then $U$ is the unipotent radical of $UX$ and thus there exists a parabolic subgroup $P$ of $G$ which contains $X$, and contains $U$ in its unipotent radical. Since $R_u(P_1) \ge R_u(P_2)$ whenever $P_1 \le P_2$ are parabolic subgroups, we can take $P$ to be minimal subject to containing $X$.

For the latter statement, the filtration $Q = Q(1) \ge Q(2) \ge \ldots$ is also a filtration by $X$-stable subgroups, since each subgroup is normal in $Q$ and $X_0$-stable. Now the relation $[Q(i),Q(j)] \subseteq Q(i+j)$ implies that the induced action of $Q$ on each level $Q(i)/Q(i+1)$ is trivial. Since elements of $X$ have the form $\phi(x)x$ with $x \in X_0$ and $\phi : X_0 \to Q$ a cocycle, it follows that $X$ and $X_0$ have precisely the same composition factors on each level, in particular the same number of trivial composition factors. Now the filtration of $Q$ by $X_0$-stable normal subgroups $Q(i)$ gives rise to a filtration of $Q^{X}$ by $X$-stable subgroups, and the corresponding quotients are trivial $X$-modules. The claim follows.
\end{proof}

\begin{lemma} \label{lem:maxparab}
Let $P$ be a maximal parabolic subgroup of $G$, and suppose that $P$ is minimal among parabolic subgroups containing a given non-$G$-cr subgroup $X$ of $G$. Then:
\begin{enumerate}
\item $P$ is the unique proper parabolic subgroup of $G$ containing $X$, \label{maxparab-i}
\item $N_G(X) \le P$ and $C_G(X)^\circ \le Q$. \label{maxparab-ii}
\end{enumerate}
\end{lemma}
\proof \ref{maxparab-i} By minimality, the image $\pi(X)$ is $L$-irreducible. Now let $P_1$ be any proper parabolic subgroup of $G$ containing $X$; we will show that $P_1 = P$. Standard results on intersections of parabolic subgroups show that the Levi factor of $P_1$ is $G$-conjugate to $L$ \cite[Lemma 3.26]{Litterick2018}. We can then assume that $P$ and $L$ are standard, and by \cite[\S 8]{MR794307} we can assume that $P_1 = n \cdot P_2$ for a standard parabolic subgroup $P_2$, where $n \in N_G(T)$ comes from a certain set of \emph{distinguished double coset representatives} of the Weyl group. Now conjugation by $n$ sends the standard Levi factor $L_2$ of $P_2$ to $L$. By maximality, $L$ and $L_2$ correspond to removing a single node from the Dynkin diagram of $G$, hence $n$ sends all the simple root subgroups in $L_2$ to positive roots; call the remaining simple root $\alpha$. If $n \cdot \alpha$ is also positive then $n$ preserves the set of positive roots of $G$, hence induces the identity element of $N_G(T)/T$, so $P_1 = P_2$. Then $P$ and $P_1$ are both standard and both minimal with respect to containing $X$, hence $P = P_1$ as claimed. If instead $n \cdot \alpha$ is negative then $n$ sends all roots of $G$ with positive $\alpha$-coefficient to negative roots. These are precisely the roots occurring in the unipotent radical of $P_2$, whereas the roots occurring in the unipotent radical of $P$ are all positive; hence $P \cap P_1 = P \cap (n \cdot P_2) = L$. However, this intersection contains $X$, and by assumption $X$ does not lie in $L$; this contradiction shows that this latter case does not occur, and thus $P_1 = P$.

\ref{maxparab-ii} By part \ref{maxparab-i} any element of $G$ normalising $X$ also normalises $P$, hence $N_G(X) \le P$. Now by Lemma~\ref{lem:rankofcents}, $C_{G}(X)^{\circ}$ has rank $0$ hence is unipotent, hence Lemma~\ref{lem:unip-cent} gives us a (proper) parabolic subgroup $P_0$ containing $X$ and having $C_{G}(X)^{\circ}$ as the fixed-point subgroup of its unipotent radical under $X$; by part \ref{maxparab-i} we have $P_0 = P$. \qed


\section{Type \texorpdfstring{$A_3$}{A3}} \label{sec:A3}

\subsection{\texorpdfstring{$G$}{G} of type \texorpdfstring{$E_7$}{E7}}

\subsubsection{\texorpdfstring{$L'$}{L'} of type \texorpdfstring{$A_5$}{A5}} \label{sec:a3ine7a5}

There are three standard $A_{5}$-parabolic subgroups of $G$, namely $P_{13456}$, $P_{34567}$ and $P_{24567}$. The first two are associated, with Levi subgroup labelled $A_5'$ according to the convention discussed in Section~\ref{sec:roots-levis}, while $P_{24567}$ forms its own association class, and has Levi subgroup labelled simply $A_5$. By Lemma~\ref{lem:associated}, then, non-$G$-cr subgroups with irreducible image in the standard Levi subgroups $L_{13456}$ and $L_{34567}$ are not $G$-conjugate to those with irreducible image in $L_{24567}$.

First let $P = P_{13456}$ or $P_{34567}$, with Levi decomposition $P = QL$. Let $X_0$ be a simple subgroup of type $A_3$ contained in the derived subgroup $L'$, with $X_0 \le L'$ via $010$ in the notation of Section~\ref{sec:embed}. As described in Section~\ref{sec:shape-modules}, the levels of the unipotent radical $Q$ are modules for $L'$, whose high weights can be determined combinatorially. The $L'$-module structure of these levels, as well as their restrictions to $X_0$, are as follows.
\begin{center}
\begin{tabular}{*{5}{c}}
 & $Q/Q(2)$ & $Q(2)/Q(3)$ & $Q(3)/Q(4)$ & $Q(4)$ \\ \hline
$L_{13456}'$ & $\lambda_3 + \lambda_5$ & $\lambda_2 + 0$ & $\lambda_5$ \\
$X_0 \le L_{13456}$ & $(010|(200 + 002)|010) + 010$ & $101 + 0^2$ & $010$ \\ \hline
$L_{34567}'$ & $\lambda_1 + \lambda_2$ & $\lambda_3$ & $\lambda_5$ & $0$ \\
$X_0 \le L_{34567}$ & $010 + 101 + 0$ & $010|(200 + 002)|010$ & $010$ & $0$ \\
\end{tabular}
\end{center}
Here we have used $\lambda_2 \downarrow X_0= \bigwedge^{2}(010) = 101 + 0$, and $\lambda_3 \downarrow X_0 = \bigwedge^{3}(010) = 010|(200 + 002)|010$. By Proposition \ref{prop:h1}, $H^{1}(X_0,101) \cong K$ and all other summands in a level of $Q$ have zero first cohomology group for $X_0$. Hence $\mathbb{V} \cong K$ (for each of the two choices of $P$).

The non-trivial torus $Z(L)$ centralises $X_0$ and acts on $Q$ without fixed points, hence induces a non-trivial action on $\mathbb{V} = K$. As discussed in Section~\ref{sec:rho}, it follows that there is at most one class of non-$G$-cr subgroups occurring in each parabolic. In $Q_{13456}$ and $Q_{34567}$ the modules of high weight $\lambda_2$ are respectively generated as an $X_0$-module by the images of $U_{0101111}$ in level $2$ and $U_{\alpha_{2}}$ in level $1$. Conjugation by the element $n_{1011111}n_{1010000} \in N_{G}(T)$ maps $L_{13456}$ to $L_{34567}$ and sends $U_{0101111}$ to $U_{\alpha_{2}}$. We deduce that a non-$G$-cr subgroup in one of these parabolic subgroups is conjugate to a subgroup of the other. In particular, any non-$G$-cr subgroup arising here has a conjugate contained in $P_{13456}$. In this case, $Q(3)$ contains $U_{0101111}$ and so any non-$G$-cr subgroup arising here has a conjugate contained in $Q(3)X_0$. Note also that $[Q(3),Q(3)] = \{0\}$ by the Chevalley commutator formula, hence $Q(3)$ is abelian. Since $Q(3)/Q(4)$ has two trivial composition factors and since $H^{1}(X_0,010) = \{0\}$, it follows that each complement to $Q$ in $QX_0$ centralises a $2$-dimensional unipotent subgroup, namely $Q(3)^{X_0}$.

We now exhibit an appropriate non-$G$-cr subgroup. By Remark \ref{rem:a3d6-d4d8}, there are two $D_6$-classes of non-$D_6$-cr subgroups $A_{3} \le D_6$ via $010 + 010$; fix representatives $Y$ and $Z$ of these. It follows from Lemma~\ref{lem:BMR} that $Y$ and $Z$ are non-$G$-cr, since $D_6$ is a Levi subgroup of $G$. Their actions on the $56$-dimensional irreducible $G$-module are given in Table~\ref{tab:E7}; pick $Y$ and $Z$ such that $V_{56} \downarrow Y = 010^4 + T(200) + T(002)$ and $V_{56} \downarrow Z = 010^4 + T(101)^2$. Then $Y$ and $Z$ are not conjugate in $G$ since they are not $\GL(V_{56})$-conjugate, and $Z$ is not contained in a conjugate of $P_{13456}$ as its composition factors on $V_{56}$ are distinct from those of $X_0$. By Proposition~\ref{prop:h1}, each non-$G$-cr subgroup of type $A_3$ is contained in an $A_5$-parabolic subgroup. Thus $Z$ lies in a conjugate of $P_{24567}$ which we consider in a moment; since $Y$ is not $G$-conjugate to $Z$, it follows that $Y$ lies in a conjugate of $P_{13456}$.

We now calculate $C_{G}(Y)^{\circ}$. To begin, Lemma \ref{lem:rankofcents} shows that $C_G(Y)$ has rank at most $1$. Since $Y \le D_6$ we have $C_{G}(Y) \ge C_G(D_6) = \bar{A}_1$ \cite[p.333 Table~2]{MR1274094}, thus $C_G(Y)$ has rank exactly $1$. Now let $U$ be a maximal connected unipotent subgroup of $C_{G}(Y)$. By Lemma~\ref{lem:unip-cent} at least one of $P_{13456}$ and $P_{34567}$ contains a conjugate of $U$ in its unipotent radical. In the table above, each $Q$ has at most $2$ trivial composition factors, hence also by Lemma~\ref{lem:unip-cent} we have $\dim U \le 2$. On the other hand, above we have exhibited a $2$-dimensional unipotent subgroup of $Q = Q_{13456}$ centralised by all complements to $Q$ in $QX_{0}$. Thus $\dim U = 2$. We have shown that $C_{G}(Y)^{\circ}$ is a connected algebraic group of rank $1$ with a maximal connected unipotent subgroup of dimension $2$. It follows that $C_{G}(Y)^{\circ} = U_{1}\bar{A}_1$ for some $1$-dimensional connected unipotent group $U_1$.

We now consider the second association class of $A_5$-parabolic subgroups. Let $P = P_{24567} = QL$ and again let $X_0$ be a simple subgroup of type $A_3$, with $X_0 \le L'$ via $010$. The corresponding actions of $L'$ and $X_0$ on $Q$ are as follows
\begin{center}
\begin{tabular}{*{6}{c}}
 & $Q/Q(2)$ & $Q(2)/Q(3)$ & $Q(3)/Q(4)$ & $Q(4)/Q(5)$ & $Q(5)$ \\ \hline
$L'$ & $\lambda_2 + 0$ & $\lambda_2$ & $\lambda_4$ & $0$ & $0$ \\
$X_0$ & $101 + 0^2$ & $101 + 0$ & $101 + 0$ & $0$ & $0$ \\ 
\end{tabular}
\end{center}
Thus $\mathbb{V} \cong K^{3}$. The $X_0$-modules $101$ in levels $1$, $2$ and $3$ of $Q$ are respectively generated by the root groups $U_{\alpha_3}$, $U_{1010000}$ and $U_{1122100}$. By \cite[p.333, Table~3]{MR1274094}, $C_G(X_0)^\circ = G_2$ and thus $Q^{X_0}$ is $6$-dimensional. The trivial $L$-module in $Q/Q(2)$ is generated by $z_1(c) = x_{\alpha_1}(c)$, hence gives a subgroup of $Q^{X_0}$ inducing an $X_0$-module isomorphism from the summand $\lambda_2$ of $Q/Q(2)$ to the summand $\lambda_2$ of $Q(2)/Q(3)$. Furthermore the trivial $X_0$-submodule of $Q(2)/Q(3)$ is generated by $z_2(c) = x_{1112100}(c)x_{1111110}(c)x_{1011111}(c)$, inducing a non-zero homomorphism $Q/Q(2) \to Q(3)/Q(4)$, and this is non-zero on the summand $101$ since this trivial module lies outside of $Z(Q/Q(4)) = Q(3)/Q(4)$. Thus, as described in Section \ref{sec:rho}, when parametrising complements to $Q$ in $QX_0$ by elements $(k_1,k_2,k_3) \in \mathbb{V}$, we may assume that at most one of $k_{1}$, $k_{2}$ and $k_{3}$ is non-zero. Moreover, the element $n_{\alpha_{1}} \in N_{G}(T)/T$ centralises $L'$ and exchanges the roots of levels $1$ and $2$ whose roots give rise to the $L'$-modules of high weight $\lambda_2$. Finally, the following element of the Weyl group normalises $L'$ and swaps the roots $1010000$ and $1122100$:
\[ n_3 n_4 n_2 n_5 n_4 n_3 n_6 n_5 n_4 n_2 n_7 n_6 n_5 n_4 n_3. \]
By the discussion in Section \ref{sec:qgconj}, these elements fuse together classes of complements corresponding to $(1,0,0)$, $(0,1,0)$ and $(0,0,1)$ of $\mathbb{V}$. Thus there is at most one non-$G$-cr subgroup in $P_{24567}$, up to $G$-conjugacy, corresponding to $(0,0,1) \in \mathbb{V}$. Above, we found a non-$G$-cr subgroup $Z$ lying in this class of parabolic subgroups with irreducible image in $L$.

It remains to calculate $C_{G}(Z)^{\circ}$. To start, $Z$ is conjugate to a subgroup corresponding to $(0,0,1) \in \mathbb{V}$, and we may therefore assume that $Z$ is contained in $Q(3) X_0$, where $Q(3)$ is abelian. We claim that any complement to $Q(3)$ in $Q(3) X_0$ commutes with $Q^{X_0}$. To check this, we start by noting that $x_{1122100}(t)$ commutes with all elements of $Q^{X_0} \cap Q(3)$ since $Q(3)$ is abelian, so any complement to $Q(3)$ in $Q(3)X_0$ commutes with $Q^{X_0} \cap Q(3)$. Now $Q^{X_0}$ is generated by $Q(3) \cap Q^{X_0}$ together with $z_1(c), z_2(c)$ given above and $z_3(c) = x_{0112100}(c)x_{0111110}(c)x_{0011111}(c)$. It is then a routine check using the Chevalley commutator formula that $x_{1122100}(t)$ also commutes with these three generators. Alternatively, one need not resort to explicit calculations since elements of $Q^{X_0}$ induce $X_0$-module homomorphisms from the $X_0$-module generated by $x_{1122100}(t)$, which is of high weight $101$ and in level $3$, into levels $4$ and above. But these homomorphisms are necessarily trivial because there are no vectors of non-zero weight in levels $4$ or $5$. This implies that elements of $Q^{X_0}$ commute with $x_{1122100}(t)$. 

At this point, we know that $C_G(Z)^{\circ}$ has rank $1$ and contains a subgroup $\bar{A}_1$. By Lemma~\ref{lem:unip-cent} we also know that $C_G(Z)$ has a maximal connected unipotent subgroup of dimension at most $6$, and we have exhibited such a subgroup above, namely $Q^{X_0}$. It follows that $C_{G}(Z)^{\circ} = U_5 \bar{A}_1$, where $U_5$ is a $5$-dimensional connected unipotent subgroup.

\subsection{\texorpdfstring{$G$}{G} of type \texorpdfstring{$E_8$}{E8}}

\subsubsection{\texorpdfstring{$L'$}{L'} of type \texorpdfstring{$A_5$}{A5}}

The four standard $A_5$-parabolic subgroups of $G$ are $P_{13456}$, $P_{34567}$, $P_{45678}$ and $P_{24567}$. We start by considering $P = P_{13456} = QL$, with $X_0 = A_{3} \le L'$ via $010$. The action of $X_0$ on the levels of $Q$ is as follows: 

\begin{center}
\begin{tabular}{*{4}{c}}
 $Q/Q(2)$ & $Q(2)/Q(3)$ & $Q(3)/Q(4)$ & $Q(4)/Q(5)$   \\ \hline
$(010|(200+002)|010)+ 010 +0$ & $101 + 010 + 0^2$ & $101 + 010 + 0$ & $010^2$   
\end{tabular}
\end{center}

\begin{center}
\begin{tabular}{*{4}{c}}
$Q(5)/Q(6)$ & $Q(6)/Q(7)$ & $Q(7)/Q(8)$ & $Q(8)$ \\ \hline
$101$   & $010$  &  $0$ & $0$ 
\end{tabular}
\end{center}
The summands of high weight $101$ in levels $2,3,5$ are respectively generated by the images of the root groups $U_{01011110}$, $U_{01011111}$, and $U_{12232221}$. Furthermore, the elements $z_1(c) = x_{\alpha_8}(c)$, $z_2(c) = x_{01122210}(c)x_{11221110}(c)x_{11122110}(c)$ and $z_3(c) = x_{01122211}(c)x_{11221111}(c)x_{11122111}(c)$ commute with $X_0$. Then a standard Chevalley commutator formula calculation shows that for each $c \neq 0$, $z_1(c)$ induces an $X_0$-module isomorphism between the summands $101$ in levels $2$ and $3$. Similarly, non-trivial elements $z_2(c)$ and $z_3(c)$ induce isomorphisms between the summands $101$ in levels $3$ and $5$, and in levels $2$ and $5$, respectively. Parametrising complements to $Q$ in $QX_0$ by $(k_1,k_2,k_3) \in \mathbb{V}$, this allows us to assume that $k_1 k_2 = k_1 k_3 = k_2 k_3 = 0$. Moreover, $n_{\alpha_8}$ preserves the roots in $L_{13456}$ and swaps the roots $01011110$ and $01011111$, the element $n_{12232111} n_{01122211} n_{10111100} n_{2} n_{4} n_{3} n_{5} n_{4} n_{2} n_{8} n_{7}$ preserves the roots in $L_{13456}$ and sends $01011111 \mapsto 12232221 \mapsto 01011110$. Thus in this case, there is precisely one non-$G$-cr subgroup of type $A_{3}$ having irreducible image in the Levi factor.

If $P = QL$ is one of the three remaining standard $A_5$-parabolics then again, precisely three direct summands of some level of the unipotent radical are irreducible of high weight $101$ for a subgroup $A_{3} \le L$ via $010$. These are the only summands occurring with non-zero first cohomology group, and are respectively generated by the images of root subgroups corresponding to the following roots:
\begin{center}
\begin{tabular}{c|c}
Parabolic & Roots \\ \hline
$P_{34567}$ & $\alpha_2$, $11111111$, $12232111$ \\ 
$P_{45678}$ & $01110000$, $11110000$, $12232100$ \\ 
$P_{24567}$ & $\alpha_3$, $10100000$, $11221000$ 
\end{tabular}
\end{center}

Essentially identical calculations in these three parabolics show that, up to conjugacy, there is at most one non-$G$-cr subgroup of type $A_3$ in each, having irreducible image in the Levi factor. Finally, the following elements of the Weyl group of $G$ send the roots of $L_{13456}$ to those of the Levi factor in each other case, and send the root $01011110$ to the given root $\beta$ occurring in the corresponding unipotent radical:
\begin{center}
\begin{tabular}{c|c|c}
Parabolic & Root $\beta$ & Weyl group element \\ \hline
$P_{34567}$ & $\alpha_2$ & $n_{7} n_{6} n_{5} n_{4} n_{3} n_{1}$ \\
$P_{45678}$ & $01110000$ & $n_{7} n_{6} n_{5} n_{4} n_{3} n_{1} n_{8} n_{7} n_{6} n_{5} n_{4} n_{3}$ \\
$P_{24567}$ & $10100000$ & $n_{112} n_{101} n_{7} n_{6} n_{5} n_{4} n_{3}$
\end{tabular}
\end{center}
We conclude that up to $G$-conjugacy there is exactly one non-$G$-cr subgroup of type $A_3$ with irreducible image in a Levi subgroup of type $A_5$.

Let $Y = A_3 \le E_7$ represent one of the two non-$E_7$-cr subgroup classes (these are fused in $G$ since they lie in a subgroup $D_6$, on which $G$ induces a graph automorphism). Then $Y$ is non-$G$-cr by Lemma \ref{lem:BMR} and contained in an $A_5$-parabolic subgroup of $G$. 

It remains to calculate $C_{G}(Y)^\circ$. Since $C_G(E_7)^\circ = \bar{A}_1$ and $C_{E_7}(Y) = U_5 \bar{A}_1$, it follows that $U_5 \bar{A}_1^2 = C_{\bar{A}_1 E_7}(Y) \le C_G(Y)$. Checking the restriction of $L(G)$ to $Y$ in Table \ref{tab:E8p2}, we see that $\dim C_G(Y) \le 12$. We claim that $Y$ is not separable in $G$, from which it follows that $U_5 \bar{A}_1^2 = C_G(Y)$. To prove that $Y$ is not separable in $G$ we start by considering the restriction $L(G) \downarrow \bar{A}_1 E_7 = L(\bar{A}_1 E_7) + (1,V_{56})$. From this, and the restriction of $L(G) \downarrow Y$, we deduce that $\dim C_{L(\bar{A}_1E_7)}(Y) = 12$. But $C_{\bar{A}_1E_7}(Y)$ was already shown to be $11$-dimensional in the above calculation. Therefore, $Y$ is not a separable subgroup of $\bar{A}_1 E_7$. Now, $(\bar{A}_1 E_7,G)$ is a reductive pair and so by \cite[Theorem~1.4]{MR2608407}, $Y$ is not a separable subgroup of $G$, as required.

\subsubsection{\texorpdfstring{$L'$}{L'} of type \texorpdfstring{$A_3^2$}{A3A3}} \label{sec:a3a3}
The two standard $A_3^2$-parabolic subgroups $P = QL$ of $G$ are $P_{134678}$ and $P_{234678}$. The filtration of the unipotent radical of $P_{134678}$ has a module $(010,100)$, generated as an $X_0$-module by the image of the root group $U_{01011000}$, and a module $(100,010)$ generated by the image of the root group $U_{01122100}$. Thus if $A_3 \le A_3^2$ via $(100^{[r]},100^{[s]})$, we find that $\mathbb{V}$ is $1$-dimensional if $(r,s) = (1,0)$ or $(0,1)$, and $\mathbb{V} = \{0\}$ otherwise. We will show below that non-trivial cocycles $X_0 \to Q(2)/Q(3)$ and $X_0 \to Q(3)/(4)$ fail to lift to cocycles $X_0 \to Q$, so that $H^{1}(X_0,Q) = \{0\}$ and $P_{134678}$ contains no non-$G$-cr subgroups of type $A_{3}$. (This occurs because of the presence of a module in the filtration with non-zero second cohomology group.)

In $P_{234678}$, the unipotent radical again gives rise to $A_3^2$-modules $(010,100)$ and $(100,010)$, respectively generated by the images of the root subgroups $U_{\alpha_5}$ and $U_{11122100}$. The Weyl group element $n_{\alpha_1} n_{\alpha_3} n_{\alpha_4} n_{\alpha_2}$ maps the standard Levi subgroup $L_{134678}$ to $L_{234678}$, and respectively sends the root groups $U_{01011000}$ and $U_{11122100}$ to $U_{\alpha_5}$ and $U_{11122100}$. Thus any non-$G$-cr subgroup $A_{3}$ of $P_{234678}$ is conjugate to a subgroup contained in $P_{134678}$.

We now show that $H^{1}(X_0,Q) = \{0\}$ when $Q = R_{u}(P_{134678})$. For this we use explicit calculations with root elements in $Q$; we give details for the case $X_0 \le L_{134678}'$ via $(100,100^{[1]})$, so that a module with non-zero first cohomology group appears only in $Q(2)/Q(3)$. We omit essentially identical calculations for other embeddings of $X_0$, where the Frobenius twist appears on the other factor, or where the dual module $001$ is used in place of $100$ in either factor.

Fix a maximal torus $T_{X_0}$ of $X_0$, contained in the fixed maximal torus of $G$, and let $\beta_1$, $\beta_2$ and $\beta_3$ be a set of simple roots of $X_0$ with respect to $T_{X_0}$. If $V = V_{X_0}(210)$ is the irreducible $X_0$-module, then a non-trivial cocycle $X_0 \to V$ is determined by its restrictions to the three subgroups $U_{\beta_1}$, $U_{\beta_2}$, $U_{\beta_3}$, since these subgroups together generate the unipotent radical $U_{X_0}$ of a Borel subgroup $B_{X_0}$ of $X_0$, and we have isomorphisms $H^{1}(X_0,V) \cong H^{1}(B_{X_0},V) \cong H^{1}(U_{X_0},V)^{T_{X_0}}$ by \cite[4.7(c), 6.9(3)]{MR2015057}. Moreover, consider an indecomposable extension $W = 0|V$, generated by a $T_{X_0}$-stable vector $v$ of weight zero. Then the cocycles $X_0 \to V$, $x \mapsto \lambda(x \cdot v - v)$ give a complete set of representatives of the cohomology classes in $H^{1}(X_0,V)$. Thus each cocycle is cohomologous to some $\phi$ such that $\phi(x_{\beta_i}(1))$ is a sum of vectors of weight $c_{i} \beta_{i}$, $c_{i} \in \mathbb{N}$. Expressing the roots as elements of the weight lattice of $X_0$, an elementary calculation shows that each weight $\beta_1 = (2,-1,0)$, $\beta_2 = (-1,2,-1)$, $\beta_3 = (0,-1,2)$ occurs in $V(210)$ with multiplicity one, and that $c_{i}\beta_{i}$ does not occur if $c_{i} > 1$. In particular, there are weight vectors $v_{(2,-1,0)}$, $v_{(-1,2,-1)}$, $v_{(0,-1,2)}$ such that $\phi$ is fully specified by $\phi(x_{\beta_1}(a)) = a v_{(2,-1,0)}$, $\phi(x_{\beta_2}(b)) = b v_{(-1,2,-1)}$, $\phi(x_{\beta_3}(c)) = c v_{(0,-1,2)}$. Using the commutator relations for $x_{\beta_1}(a)$, $x_{\beta_2}(b)$ and $x_{\beta_3}(c)$, we calculate that
\[ \phi(x_{\beta_1}(a) x_{\beta_2}(b) x_{\beta_3}(c)) = a v_{(2,-1,0)} + b v_{(-1,2,-1)} + c v_{(0,-1,2)} + bc^{2} v_{(-1,0,3)} + ab^{2} v_{(0,3,-2)} + ab^{2}c^{2} v_{(0,1,2)} \]
for all $a,b,c \in K$, where each $v_{j}$ is a non-zero vector of $T_{X_0}$-weight $j$.

We now realise $T_{X_0}$ as a sub-torus of $T$. We realise $A_{3}$ as a subgroup of $A_3^2 = L_{134678}'$ via $(100,100^{[1]})$ by defining $x_{\beta_1}(c) = x_{1}(c) x_{6}(c^2)$, $x_{\beta_2}(c) = x_{3}(c) x_{7}(c^2)$, $x_{\beta_3}(c) = x_{4}(c) x_{8}(c^2)$ for all $c \in K$. Then $T_{X_0}$ is generated by elements $h_{\beta_{i}}(t) = n_{\beta_i}(t)n_{\beta_i}(-1)$, where we in turn define $n_{\beta_i}(t) = x_{\beta_i}(t) x_{-\beta_i}(-t^{-1})x_{\beta_i}(t)$, for $t \in K^{\ast}$.

Now, since $T_{X_0}$ is a sub-torus of $T$, the $T_{X_0}$-weight spaces of $Q(2)/Q(3)$ are root subgroups of $G$. Direct calculation allows us to identify the six vectors $v_{j}$ above as follows:
\begin{align*}
v_{(2,-1,0)} &= x_{01011111}(\lambda_1)Q(3),
&v_{(-1,2,-1)} &= x_{01121110}(\lambda_2)Q(3),
& v_{(0,-1,2)} &= x_{11221100}(\lambda_3)Q(3),\\
v_{(-1,0,3)} &= x_{01121111}(\lambda_4)Q(3),
&v_{(0,3,-2)} &= x_{11221110}(\lambda_5)Q(3),
&v_{(0,1,2)} &= x_{11221111}(\lambda_5)Q(3)
\end{align*}
for some non-zero scalars $\lambda_1$, $\ldots$, $\lambda_6$. Finally, note that for a cocycle $\phi \, : \, X_0 \to Q$ and an arbitrary unipotent element $u \in X_0$, the element $\phi(u)u$ has order at most four, since it is a unipotent element in a complement to $Q$ in $QX_0$. So we take the element $\phi(y(a,b,c))y(a,b,c)$ where $y(a,b,c) = x_{\beta_1}(a)x_{\beta_2}(b)x_{\beta_3}(c)$, and calculate:
\[ y(a,b,c)^4 = x_{11221111}(ab^{2}c\lambda_1 + abc^{2}\lambda_{2}) x_{12232221}(ab\lambda_{1}\lambda_{2}). \]
In particular, this is only zero for all $a$, $b$ and $c$ if $\lambda_1$ and $\lambda_2$ are zero, which is the case if and only if $\phi$ is a coboundary.

\subsubsection{\texorpdfstring{$L'$}{L'} of type \texorpdfstring{$D_7$}{D7}}

Let $P = P_{2345678} = QL$ be the unique standard $D_7$-parabolic subgroup of $G$. The unipotent radical has two levels, with $Q/Q(2) \cong \lambda_6$ and $Q(2) = Z(Q) \cong \lambda_1$ as $L'$-modules.

Let $X_0 = A_3 \le L'$ via $101$. Then weight-space calculations show that $V_{L'}(\lambda_6) \downarrow X_0 = V_{L'}(\lambda_7) \downarrow X_0 = V_{X_0}(111)$ (cf.\ proof of \cite[Proposition 2.12]{MR1329942}; this is also stated in \cite[p.\ 283, Case S7]{MR888704}), and these are the only modules of this high weight occurring in $L(G) \downarrow X_0$. In particular $\mathbb{V} = K$ and there is at most one non-$G$-cr subgroup of type $A_3$ minimally contained in $P$, up to $G$-conjugacy.

Now consider a subgroup $Y$ of type $A_3$, embedded in a subsystem subgroup $D_8$ via the 16-dimensional module $T(101) = 0|101|0$. Since $Y$ is reducible on the natural 16-dimensional module and $Y$ preserves a nondegenerate quadratic form on $101$ \cite[p.\ 283, Case S7]{MR888704}, it follows that $Y$ lies in a $D_7$-parabolic subgroup of $D_8$, hence in a conjugate of $P$, and we can assume that the image of $Y$ in $L'$ is $X_0$. Note that $Y$ is not $D_8$-conjugate to $X_0$ as they have incompatible actions on the natural $16$-dimensional module, thus $Y$ is non-$D_8$-cr.

In a $D_7$-parabolic of $D_8$ containing $Y$, the unipotent radical, call it $Q_{D_8}$, is abelian and isomorphic to $V_{D_7}(\lambda_1)$ as a module for $L'$. In $G$, we have $Q_{D_8} = Z(Q) = Q(2)$, and $Q_{D_8} \cong 101$ and $Q/Q_{D_8} \cong V_{Y}(111)$ as $X_0$-modules. Now the inclusion $Q_{D_8} \to Q$ induces a long exact sequence in cohomology, and since $Q/Q_{D_8}$ has no fixed points under the action of $X_0$, Corollary~\ref{cor:mcrgcr} tells us that $Y$ is non-$G$-cr.

Since $P$ is a maximal parabolic subgroup of $G$ and $Q^{X_0} = Q^Y = 1$, by Lemma~\ref{lem:maxparab} the connected centraliser of $Y$ is trivial.


\section{Type \texorpdfstring{$B_3$}{B3} and \texorpdfstring{$B_3^2$}{B3B3}} \label{sec:B3}

\subsection{\texorpdfstring{$G$}{G} of type \texorpdfstring{$F_4$}{F4}}

The classification of non-$G$-cr subgroups of type $B_3$ in $G$ is given in \cite[Lemma~4.4.3]{MR3075783}. There are two classes of subgroups, which are respectively $G$-conjugates of $Y_1 < D_4 < B_4$ via $T(100)$ and $Y_2 < \tilde{D}_4 < C_4$ via $T(100)$. Moreover, $Y_1$ is contained in a $B_3$-parabolic subgroup and $Y_2$ is contained in a $C_3$-parabolic subgroup.

It remains to calculate $C_G(Y_1)^\circ$ and $C_G(Y_2)^\circ$. It suffices to calculate just one of these, as the exceptional graph morphism swaps $Y_1$ and $Y_2$ and thus their connected centralisers are isomorphic as abstract groups. We will calculate $C_G(Y_2)^\circ$. From the action of $Y_2$ on $L(G)$, given in Table \ref{tab:F4}, we see that $\dim C_G(Y_2) \le 1$. On the other hand, $Y_2$ is contained in the parabolic subgroup $P_{234} = QL$, and $Q$ has two levels, the second being $Z(Q)$ which is $1$-dimensional and hence centralised by $L'$. Thus every complement to $Q$ in $QL'$ centralises this $1$-dimensional unipotent subgroup and so $C_G(Y_2)^\circ = U_1$. 

\subsection{\texorpdfstring{$G$}{G} of type \texorpdfstring{$E_6$}{E6}}

\subsubsection{\texorpdfstring{$L'$}{L'} of type \texorpdfstring{$A_5$}{A5}}
Let $P = P_{13456} = QL$ and $X_0 = C_3 \le L'$ via $100$. The unipotent radical $Q$ has two levels, with $Q(2) = Z(Q)$ a trivial $1$-dimensional $L'$-module, and $Q/Q(2)$ isomorphic to the irreducible $A_5$-module $\lambda_3$. By Lemma \ref{lem:c3-wedge-cube-h1}, $H^{1}(X_0,Q/Q(2))$ vanishes and $H^{1}(X_0,(Q/Q(2))^{[1]})$ is $1$-dimensional. Since $H^{1}(X_0,Q(2)) = \{0\}$ it follows from Lemma~\ref{lem:exactseq} that $H^{1}(X_0,Q)$ vanishes and $H^{1}(X_0,Q^{[1]})$ is at most $1$-dimensional. Any non-$G$-cr complements to $Q$ in $QX_0$ are of type $B_{3}$, by Lemma \ref{lem:bn-cn}. Taking into account the action of the torus $Z(L)$ on $H^{1}(X_0,Q^{[1]})$ we see that up to $G$-conjugacy there is at most one non-$G$-cr subgroup $B_3$ in $P$ with image $X_0$ under projection to $L$.

Let $Y$ be a non-$G$-cr subgroup of type $B_{3}$ contained in a $C_3$-parabolic subgroup of $F_4$, as exhibited in \cite[Theorem 1(A)]{MR3075783}. Note that $Y < \tilde{D}_4 < F_4$, and this is how it is listed in Table \ref{tab:E6}. An embedding $F_4 \to E_6$ maps a $C_3$-parabolic subgroup of $F_4$ into $P$. Let $R$ be the unipotent radical of the $C_3$-parabolic subgroup of $F_4$. Then $R$ has a filtration $R/R(2) \cong V_{C_3}(\lambda_3)$, $R(2) \cong V_{C_3}(0)$ as $C_3$-modules. Thus $R$ is an $X_0$-invariant normal subgroup of $Q$, and $(Q/R)^{[1]}$ is isomorphic to an $X_0$-module of shape $200|002$. In particular this quotient has no fixed points, and by Corollary~\ref{cor:mcrgcr} this non-$F_4$-cr subgroup $Y$ remains non-$G$-cr.

Since $P$ is a maximal parabolic subgroup, Lemma \ref{lem:maxparab} implies that $C_G(Y)^\circ = Q^Y$. Since $Q^{X_0} \le Z(Q)$ and is $1$-dimensional it follows that $C_G(Y)^\circ = U_1$. 

\subsection{\texorpdfstring{$G$}{G} of type \texorpdfstring{$E_7$}{E7}}

\subsubsection{\texorpdfstring{$L'$}{L'} of type \texorpdfstring{$A_5$}{A5}} \label{sec:a5ine7}

The three standard $A_{5}$-parabolic subgroups of $G$ are $P_{13456}$, $P_{34567}$ and $P_{24567}$. The unipotent radical of $P_{24567}$ involves only modules of high weight $0$, $\lambda_2$ and $\lambda_4$ for the Levi factor; in particular none of these modules have a non-zero first cohomology group for a subgroup $C_{3}$. As discussed in Section \ref{sec:A3}, the other two parabolic subgroups are associated, and the filtrations of their unipotent radicals each involve two modules of high weight $\lambda_1$ or $\lambda_5$, and one module of high weight $\lambda_3$, so $\mathbb{V} \cong K^{3}$ in each case.

To begin, let $P = P_{34567} = QL$, and let $X_0 = C_3 < L_{34567}'$ via $200$. Then $Q$ involves a single $L'$-module of each high weight $\lambda_1$, $\lambda_3$ and $\lambda_5$, respectively generated as an $L'$-module by the image of the root groups $U_{\alpha_1}$, $U_{1111000}$ and $U_{1223210}$. The unipotent radical $Q$ also involves an $L'$-module of high weight $\lambda_2$, generated by the image of the root group $U_{\alpha_2}$. As an $X_0$-module, this is isomorphic to $\bigwedge^{2}(200) = 020 + 0$, and this trivial module induces a non-zero map of $X_0$-modules $Q/Q(2) \to Q(2)/Q(3)$. As discussed in the proof of Lemma \ref{lem:c3-wedge-cube-h1}, this induces an isomorphism of cohomology groups $H^{1}(X_0,200) \to H^{1}(X_0, \bigwedge^{3}(200))$. Thus if we parametrise complements to $Q$ in $QX_0$ by $(k_1,k_2,k_3) \in \mathbb{V}$, then we may assume $k_1 k_2 = 0$.

If $k_1 = 0$ then a $Q$-conjugate of the corresponding complement to $Q$ in $QX_0$ is contained in the subgroup generated by $X_0$, $U_{1111000}$ and $U_{1223210}$. The Weyl group element $n_{7} n_{6} n_{5} n_{4} n_{3} n_{1}$ sends the roots in $L_{34567}$ to roots in $L_{13456}$, and sends the roots $1111000$ and $1223210$ to roots in $Q_{13456}$. If instead $k_2 = 0$ then the Weyl group element:
\[ n_{7} n_{6} n_{5} n_{4} n_{2} n_{3} n_{1} n_{4} n_{3} n_{5} n_{4} n_{2} n_{6} n_{5} n_{4} n_{3} n_{7} n_{6} n_{5} n_{4} n_{2} \]
sends the roots in $L_{34567}$ to roots in $L_{13456}$, and sends the roots $\alpha_1$ and $1111000$ to roots in $Q_{13456}$. Thus every non-$G$-cr complement to $Q$ in $QX_0$ is $G$-conjugate to a subgroup of $P_{13456}$.

Now let $P = P_{13456} = QL$. Then $Q$ has three levels and $X_0 = C_3 < L'$ acts on them as follows. 
\begin{center}
\begin{tabular}{*{3}{c}}
$Q/Q(2)$ & $Q(2)/Q(3)$ & $Q(3)$ \\ \hline
$(100|001|100) + 100$ & $010 + 0^2$ & 100
\end{tabular}
\end{center}

Thus $\mathbb{V} \cong K^3$, where the $X_0$-submodule of high weight $100$ in $100|001|100$ is generated by elements of the form $x_{0111000}(c)x_{0101100}(c)$ and the $X_0$-summands of high weight $100$ are generated by the image of the root groups $U_{\alpha_7}$ and $U_{1223211}$. 

The root group $U_{1223210}$ is a trivial $L'$-module, and induces a non-zero $L'$-module homomorphism between the two modules of high weight $\lambda_5$. Moreover, the corresponding Weyl group element $n_{1223210}$ swaps the roots $\alpha_7$ and $1223211$, and sends $0111000$ and $0101100$ to negative roots. Thus, parametrising complements to $Q$ in $QX_0$ by triples $(k_1,k_2,k_3) \in \mathbb{V}$, we may assume that $k_2 k_3 = 0$, and we may also swap $k_2$ and $k_3$, so long as $k_1 = 0$.

Taking into account the action of the $2$-dimensional torus $Z(L)$, non-$G$-cr complements to $Q$ in $QX_0$ correspond to one of the four triples $(1,0,0)$, $(0,0,1)$, $(1,1,0)$ and $(1,0,1)$. Thus there are at most four non-$G$-cr complements to $Q$ in $QX_0$, up to $G$-conjugacy.

We now present four non-$G$-cr subgroups of type $B_3$, and we see that they are all non-conjugate by considering their actions on $V_{56}$, given in Table \ref{tab:E7}. Let $Y_1 < E_6$ be the non-$E_6$-cr subgroup from the previous section and let $Y_2 < A_6$ be embedded via $W(100)$. Then $Y_1$ and $Y_2$ are non-$G$-cr by Lemma \ref{lem:BMR}. Subgroups $Y_3$ and $Y_4$ are embedded in the maximal rank subgroup $A_7$ via $T(100)$ and $001$, respectively. To prove these are non-$G$-cr we note that they both act on $V_{56}$ with two indecomposable summands of dimension $28$ and hence are not contained in any Levi subgroup of $G$. The subgroup $Y_3$ is contained in a parabolic subgroup of $A_7$ and hence of $G$. The subgroup $Y_4$ is contained in a parabolic subgroup of $G$ because $Y_4 < D_4 < C_4 < A_7$ and this $C_4$ subgroup is non-$G$-cr, as discussed in Section \ref{sec:C4}. Note that all $Y_i$ subgroups must be contained in a parabolic subgroup of type $A_5'$, since we have proved that a representative for every non-$G$-cr subgroup of type $B_3$ in $G$ is contained $P_{13456}$. 

It remains for us to calculate the connected centralisers of the four non-$G$-cr subgroups. Let us start by considering $Y_1$, which is contained in $F_4 < E_6$. We noted above that $C_{F_4}(Y_1)^\circ = U_1$. We also have that $C_G(F_4)^\circ = A_1$ (by \cite[p.333, Table 3]{MR1274094}) and so $U_1 A_1 \le C_G(Y_1)^\circ$. The socle series of the restriction of $L(G)$ to $Y_1$ given in Table \ref{tab:E7} shows that $\dim C_{L(G)}(Y_1) \le 4$. Thus $C_G(Y_1)^\circ = U_1 A_1$. 

Similarly, by considering the action of $Y_2$ on $L(G)$ we find that $\dim C_G(Y_2) \le 3$. The non-$G$-cr subgroups corresponding to $(1,0,0)$ and $(0,0,1)$ both centralise tori and so $Y_1$ and $Y_2$ must be representatives of their conjugacy classes. In fact, one can check that $Y_2$ is conjugate to the subgroup corresponding to $(0,0,1)$ by calculating that $C_G(S)^\circ$ is of type $A_6 T_1$ for the $1$-dimensional torus $S \le Z(L)$ centralising the root $1223211$. Considering the structure of the unipotent radical of $P_{13456}$ given above, we see that the non-$G$-cr subgroup corresponding to $(0,0,1)$, contained in $Q(2)X_0$, centralises a $2$-dimensional unipotent subgroup of $Q$ generated by the trivial $X_0$-modules in level $2$ generated by $x_{1122111}(c) x_{11122111}(c) x_{0112221}(c)$ and $x_{1223210}(c)$. Therefore, $C_G(Y_2)^\circ = U_2 T_1$. 

Finally, the subgroups $Y_3$ and $Y_4$ are not contained in any Levi subgroups, as shown above, and therefore their connected centralisers are unipotent and they are each conjugate to precisely one of the subgroups corresponding to $(1,1,0)$ and $(1,0,1)$ (though we do not yet know which is which). Note that $C_{L(G)}(Y_i)$ contains $Z(L(G))$, a $1$-dimensional subalgebra generated by a semisimple element, and hence $\dim C_G(Y_i) \le \dim C_{L(G)}(Y_i) -1$ for $i=3,4$. In particular, using the restrictions in Table \ref{tab:E7} we find that $\dim C_G(Y_3) \le 2$ and $\dim C_G(Y_4) \le 1$. Studying $P = P_{13456}$ as in the previous paragraph, we see that the subgroup corresponding to $(1,1,0)$ centralises a $1$-dimensional subgroup of $Q$, whereas the subgroup corresponding to $(1,0,1)$ centralises a $2$-dimensional unipotent subgroup of $Q$.  It now follows that $Y_3$ is conjugate to the subgroup corresponding to $(1,0,1)$ with $C_{G}(Y_3) = U_2$ and similarly, $Y_3$ is conjugate to the subgroup corresponding to $(1,1,0)$ with $C_{G}(Y_4) = U_1$. 

\subsection{\texorpdfstring{$G$}{G} of type \texorpdfstring{$E_8$}{E8}}

\subsubsection{\texorpdfstring{$L'$}{L'} of type \texorpdfstring{$A_5$}{A5}} \label{subsec:b3ine8a5}

The four standard $A_5$-parabolic subgroups of $G$ are $P_{13456}$, $P_{34567}$, $P_{45678}$ and $P_{24567}$. We argue in this section as follows.

For $P = QL$ equal to each of the four standard parabolic subgroups, with $X_0 = C_3 < L'$ via $100$, we find that $P$ contains only finitely many complements to $Q$ in $QX_0$, up to $P$-conjugacy. Moreover, each such subgroup is $G$-conjugate to a subgroup of $P_{34567}$. In fact, we can prove this latter claim without a complete analysis of the complements to $Q$ in $QX_0$. We then proceed to consider $P_{34567}$ in detail. 

First let $P = P_{13456}$, and let $X_0 = C_3 < L'$ via $100$. The $X_0$-modules with non-zero first cohomology group in the filtration of $Q$ are generated by the images of the root groups $U_{\alpha}$ for $\alpha \in \{ \alpha_7, \alpha_2, \alpha_7 + \alpha_8, 12232110, 01122221, 12232111, 13354321 \}$. Moreover there are four $1$-dimensional subgroups of $Q$ which centralise $X_0$ and induce non-zero homomorphisms between $X_0$-modules in levels of $Q$. We record these homomorphisms below. Writing $(\alpha,\beta)$ here means that a homomorphism is induced from the $X_0$-module generated by the image of $U_{\alpha}$, to the $X_0$-module generated by the image of $U_{\beta}$:
\begin{center}
\begin{tabular}{c|c}
Element & Induced homomorphism \\ \hline
$x_{8}(c)$ & $(\alpha_7,\alpha_7 + \alpha_8)$, $(12232110,12232111)$ \\
$x_{12232100}(c)$ & $(\alpha_7,12232110)$, $(\alpha_7 + \alpha_8,12232211)$, $(01122221,13354321)$ \\
$x_{01122210}(c) x_{11122110}(c) x_{11221110}(c)$ & $(\alpha_7 + \alpha_8,01122221)$, $(12232111,13354321)$ \\
$x_{01122211}(c) x_{11122111}(c) x_{11221111}(c)$ & $(\alpha_7,01122221)$, $(12232110,13354321)$
\end{tabular}
\end{center}

Working through all of the relations implied by these homomorphisms, as well as elements of $N_G(T)$ normalising $L$, it transpires that $P_{13456}$ contains at most eight non-$G$-cr complements to $Q$ in $QX_0$, up to $G$-conjugacy. However we do not need this; instead we show only that each such subgroup is conjugate to a subgroup of $P_{34567}$. For this, notice that the first homomorphism induced by $x_{8}(c)$ allows us to assume that $k_1 k_2 = 0$, if we take an appropriate basis of the $7$-dimensional space $\mathbb{V}$. Moreover, conjugation by the Weyl group element $n_{8}$ centralises $X_0$ and swaps $\alpha_7$ and $\alpha_7 + \alpha_8$, and also preserves the set of roots whose root groups in $Q$ give rise to $X_0$-modules with non-zero first cohomology group. It follows that we can assume that $k_1 = 0$. So every complement to $Q$ in $QX_0$ is $Q$-conjugate to a complement in $Q_{1}X_0$, where $Q_{1}$ is the smallest $X_0$-invariant subgroup of $Q$ containing $Q(2)$ and $U_{\alpha_7 + \alpha_8}$.

Now consider the element $n_{7} n_{6} n_{5} n_{4} n_{3} n_{1}$ of the Weyl group $W$. This sends the roots occurring in $L_{13456}$ to the roots occurring in $L_{34567}$, and sends all of the roots giving modules with non-zero first cohomology group, other than $\alpha_7$, to roots occurring in $Q_{34567}$. Thus this Weyl group element sends a $Q$-conjugate of each non-$G$-cr complement to $Q$ in $QX_0$ to a subgroup of $P_{34567}$, as claimed.

Entirely similar calculations occur in the parabolic subgroups $P_{45678}$ and $P_{24567}$. Up to conjugacy, we find that there are at most eight and three non-$G$-cr complements, respectively, to the unipotent radical in an appropriate semidirect product. For the case $P = P_{45678}$, these are all contained in the smallest $X_0$-invariant subgroup of $P_{45678}$ generated by $X_0$ and $U_{23454321}$, $U_{11222100}$, $U_{11222110}$, $U_{10100000}$, $U_{12343210}$, $U_{\alpha_2}$ and $U_{22343210}$. The Weyl group element $n_{3} n_{4} n_{5} n_{6} n_{7} n_{8}$ sends the roots in $L_{45678}$ to the roots in $L_{34567}$ and sends each of these other root groups to root groups in $Q_{34567}$. Thus each non-$G$-cr complement in $P_{45678}$ is $G$-conjugate to a subgroup of $P_{34567}$. Similarly each complement in $P_{24567}$ is $P_{24567}$-conjugate to a complement contained in the smallest $X_0$-invariant subgroup generated by $X_0$ and $U_{22454321}$, $U_{11232111}$, $U_{11222211}$, $U_{10111111}$ and $U_{12343211}$. The following Weyl group element:
\[ n_{0 0 1 1 1 1 1 1} n_{2} n_{4} n_{3} n_{5} n_{4} n_{2} n_{6} n_{5} n_{4} n_{3} n_{7} n_{6} n_{5} n_{4} n_{2} n_{8}\]
sends the roots in $L_{24567}$ to the roots occurring in $L_{34567}$, and sends each of the root elements above to roots arising in the unipotent radical of $P_{34567}$; again we conclude that each non-$G$-cr occurring in $P_{24567}$ is $G$-conjugate to a subgroup of $P_{34567}$.

So let $P = QL = P_{34567}$, and fix $X_0 = C_{3} < L'$. The seven relevant summands of $Q$ are generated by the images of the root groups $U_{\alpha_1}$, $U_{\alpha_8}$, $U_{01011111}$, $U_{11110000}$, $U_{12232100}$, $U_{22343211}$, $U_{23354321}$. All but the fourth of these give rise to an irreducible $A_5$-module with high weight $\lambda_1$ or $\lambda_5$, hence to a natural $6$-dimensional $X_0$-module $100$. The root group $U_{11110000}$ gives rise to an $A_5$-module $\lambda_3$, which restricts to $X_0$ as $100|001|100$. The submodule $100$ is generated by the images of elements of the form $x_{11121000}(c)x_{11111100}(c)$, and as discussed in Lemma \ref{lem:c3-wedge-cube-h1} the inclusion $100 \to 100|001|100$ induces an isomorphism in first cohomology (after applying a Frobenius twist).

Fix a basis of $\bigoplus_{i = 1}^{7} H^{1}(X_0,Q(i)/Q(i+1)) = \mathbb{V} \cong K^{7}$, whose $i$-th member spans the first cohomology group of the $i$-th summand above, so that complements to $Q$ in $QX_0$ are parametrised by $7$-tuples $(k_1,k_2,\ldots,k_7)$.

The following elements generate $Q^{X_0}$, and the first six of them induce non-zero homomorphisms between $X_0$-modules in the levels of $Q$. Again, $(\alpha,\beta)$ means that a non-zero homomorphism is induced from the $X_0$-module generated by the image of $U_{\alpha}$, to the $X_0$-module generated by the image of $U_{\beta}$:
{\small
\begin{center}
\begin{tabular}{c|c|c|c}
Name &Elements & Induced homomorphisms & Implied conditions\\ \hline
$z_1(c)$ & $x_{01011110}(c)x_{01111100}(c)x_{01121000}(c)$ & $(\alpha_1,11110000$ and $12232100)$, & $k_1 k_4 = 0$ or $k_1 k_5 = 0$;\\
& & $(\alpha_8,01011111)$, & if $k_1 = 0$ then $k_2 k_3 = 0$;\\
& & $(22343211,23354321)$ & if $k_1 = k_2 = 0$ then $k_6 k_7 = 0$ \\ \hline
$z_2(c)$ & $x_{10111111}(c)$ & $(12232100,22343211)$ & $k_5 k_6 = 0$ \\ \hline
$z_3(c)$ & $x_{11222211}(c)x_{11232111}(c)x_{11122221}(c)$ & $(12232100,23354321)$ & $k_5 k_7 = 0$ \\ \hline
$z_4(c)$ & $x_{22343210}(c)$ & $(\alpha_8,22343211)$, & $k_2 k_6 = 0$; \\
& & $(01011111,23354321)$ & if $k_2 = 0$ then $k_3 k_7 = 0$ \\ \hline
$z_5(c)$ & $x_{12343211}(c)x_{12243221}(c)x_{12233321}(c)$ & $(\alpha_1,22343211)$ & $k_1 k_6 = 0$ \\ \hline
$z_6(c)$ & $x_{13354321}(c)$ & $(\alpha_1,23354321)$ & $k_1 k_7 = 0$ \\ \hline
$z_7(c)$ & $x_{23465432}(c)$ & & 
\end{tabular}
\end{center}
}

Moreover, there are a number of elements of the Weyl group of $G$ which stabilise the set of simple roots in $L_{34567}$ and conjugate together various complements to $Q$ in $QX_0$. The element
\[ w = n_{2} n_{4} n_{3} n_{5} n_{4} n_{2} n_{6} n_{5} n_{4} n_{3} n_{7} n_{6} n_{5} n_{4} n_{2} \]
normalises $L'$ and swaps $U_{\alpha_1}$ with $U_{12232100}$; $U_{\alpha_8}$ with $U_{01011111}$; $U_{22343211}$ with $U_{23354321}$; and fixes $U_{11110000}$. This therefore swaps $k_1$ with $k_5$, and using the first induced homomorphism above, we can therefore take $k_1 = 0$. Hence using the same induced homomorphism, we can also assume $k_2 k_3 = 0$.

If $k_2 \neq 0$ then $k_3 = k_6 = 0$. From the above induced homomorphisms we also have $k_5 k_7 = 0$. If $k_5 = 0$ then applying the above Weyl group element lets us assume that $k_2 = 0$, which we deal with in a moment. So we can assume $k_5 \neq 0$, and we obtain the tuples $(0,1,0,1,1,0,0)$ and $(0,1,0,0,1,0,0)$, after considering the action of the three-dimensional torus $Z(L)$ on $H^{1}(X_0,Q^{[1]})$.

Now assume $k_2 = 0$. If $k_3 \neq 0$ then $k_7 = 0$ and $k_5 k_6 = 0$. Moreover the element $n_{10111111}$ swaps $k_5$ and $k_6$, while stabilising $k_3$, $k_4$ and $k_7$ (it also sends $\alpha_1$ and $\alpha_8$ to negative roots, but this does not matter as $k_1 = k_2 = 0$). We can thus assume $k_5 = 0$. If $k_4 = k_6 = 0$ then we can apply the Weyl group element:
\[ n_{22343211} n_{11122221} n_{00111110} n_{11110000} n_{11122100} n_{5} n_{4} n_{6} n_{5} n_{8} \]
which swaps $k_3$ and $k_6$, and therefore puts us in the $k_3 = 0$ case which we deal with in the next paragraph. Thus we assume that one of $k_4$ and $k_6$ is non-zero, and we obtain the tuples $(0,0,1,1,0,1,0)$, $(0,0,1,0,0,1,0)$ and $(0,0,1,1,0,0,0)$, after considering the action of $Z(L)$.

Now assume $k_1 = k_2 = k_3 = 0$. Then we have $k_6 k_7 = k_5 k_6 = k_5 k_7 = 0$, so at most one of $k_5$, $k_6$ and $k_7$ is non-zero. Again, the element $n_{10111111}$ swaps $k_5$ and $k_6$, so we can assume that $k_5 = 0$. Then the Weyl group element $w$ above lets us swap $k_6$ and $k_7$, so we can assume that $k_6 = 0$. We obtain the tuples $(0,0,0,1,0,0,1)$, $(0,0,0,1,0,0,0)$ and $(0,0,0,0,0,0,1)$, as well as $(0,0,0,0,0,0,0)$, corresponding to the $G$-cr class, after considering the action of $Z(L)$.

In summary, there are at most eight non-$G$-cr complements to $Q$ in $QX_0$ up to conjugacy in $G$.

We now give eight non-$G$-cr subgroups of type $B_3$. Their actions on $L(G)$, given in Table \ref{tab:E8p2}, show that they represent different $G$-classes. Moreover, their composition factors on $L(G)$ show that they are all contained in a parabolic subgroup of type $A_5$. We let $Y_1 < E_6$, $Y_2 < A_6$ via $W(100)$, and $Y_3, Y_4 < A_7' < E_7$ via $T(100)$ and $001$, respectively be the non-$E_7$-subgroups from Section \ref{sec:a5ine7}, which are non-$G$-cr by Lemma \ref{lem:BMR}. The subgroups $Y_5 < A_7$ via $T(100)$ and $Y_6 < D_7$ via $W(100) + W(100)^*$ are also non-$G$-cr by Lemma \ref{lem:BMR}. Next we let $Y_7 < D_8$ via $T(100)^2$, which is a non-$D_8$-cr subgroup. From Table \ref{tab:E8p2}, we see that $Y_7$ acts on $L(G)$ with five indecomposable summands, two of which have dimension $28$ and the other three have dimension $64$. It now follows that $Y_7$ is not contained in any Levi subgroup of $G$ and is hence non-$G$-cr. Finally, we let $Y_8$ be a non-$G$-cr subgroup of type $B_3$ from Lemma \ref{lem:B3inD4}, contained in a non-$G$-cr subgroup of type $D_4$. The action of $Y_8$, given in Table \ref{tab:E8p2}, shows that it is not conjugate to the previous seven $Y_i$ and not contained in any Levi subgroup of $G$ (it has five direct summands on $L(G)$ of dimensions $30,30,62,63,63$).  

It remains for us to calculate the connected centralisers of the subgroups $Y_i$. In this case, the following proposition gives us an explicit construction of the non-$G$-cr subgroups occurring here.

\begin{proposition} \label{prop:B3inE8gens}
Let $G$ be of type $E_8$, $p=2$ and $P = P_{34567} = QL$. Let $X_0$ be a subgroup of type $B_3$ contained in $P$, with irreducible image $X_0$ of type $C_3$ in $L$. Then there exists $\mathbf{v} = (a_1,\ldots,a_7) \in \mathbb{V} \cong K^7$ such that $X_0$ is conjugate to $X_{\mathbf{v}} = \langle y_{\pm i}(t) : t \in K,\, 1 \le i \le 3 \rangle$ for the following elements $y_{\pm i}(t)$:
\begin{align*}
y_{\pm1}(t) &= x_{\pm3}(t)x_{\pm7}(t), \\
y_{\pm2}(t) &= x_{\pm4}(t)x_{\pm6}(t), \\
y_3(t) &= x_{5}(t) x_{10111000}(a_1t) x_{00001111}(a_2t) x_{11222100}(a_3t) x_{11122110}(a_3t)  x_{01122111}(a_4t) \\ 
&\phantom{{}={}} x_{12233210}(a_5t) x_{22344321}(a_6t)  x_{23465321}(a_7t), \\   
y_{-3}(t) &= x_{-5}(t) x_{10110000}(a_1t) x_{00000111}(a_2t) x_{11221100}(a_3t)  x_{11121110}(a_3t)  x_{01121111}(a_4t) \\ 
&\phantom{{}={}} x_{12232210}(a_5t) x_{22343321}((a_6+a_1a_2a_5 + a_2a_3^2)t)  x_{23464321}((a_7+a_1a_4a_5+a_3^2a_4)t).
\end{align*}
\end{proposition}
\begin{proof}
By construction, the given elements lie in $P$, and when $a_1,\ldots,a_7$ are all zero these generate the $L$-irreducible subgroup of type $C_3$. It is also clear that the cocycles have image in the appropriate modules in the levels of $Q$. It remains, therefore, only to prove that each such subgroup is indeed a group of type $B_3$. This is now routine using \cite[Theorem 12.1.1]{MR0407163} (calculations assisted using \Magma).
\end{proof}

\begin{remark}
To obtain the form of the implicit cocycles above, consider the restriction map $H^{1}(X_0,200) \to H^{1}(Y_0,200 \downarrow Y_0)$, where $Y_0$ is a subgroup of type $A_1$ of $X_0$ generated by $x_{\pm 5}(t)$. Then $Y_0$ is the derived subgroup of a Levi subgroup of $X_0$. Now $200 \downarrow Y_0 = 2 + 0^{4}$. It follows that the map $H^{1}(X_0,200) \to H^{1}(Y_0, 200 \downarrow Y_0)$ is injective. Cocycles for subgroups of type $A_1$ have been described explicitly in \cite[Lemma 3.6.2]{MR3075783}. To lift this to a classification of cocycles $Y_0 \to Q$, one can use the relations defining groups of type $A_1$; see \cite[Lemma 3.6.1]{MR3075783}, and \cite[p.\ 45]{MR3075783} for an explicit example.
\end{remark}

Returning to the centraliser calculations, we first find the rank of $C_G(Y_i)$ by considering the smallest Levi subgroup in which $Y_i$ is contained. Indeed, within our list of non-$G$-cr subgroups we have all non-$L$-cr subgroups of type $B_3$ for every Levi subgroup $L$, recalling that the first four subgroups represent the four $E_7$-classes of non-$E_7$-cr subgroups. Thus the rank is $2$ for $Y_1$ and $Y_2$, since they are contained in $E_6$ and $A_6$, respectively. For $Y_3$, \ldots, $Y_6$, the rank is $1$ since they are contained in $E_7, E_7, A_7, D_7$, respectively. Finally, $Y_7$ and $Y_8$ are not contained in any Levi subgroup and hence the rank of its connected centraliser is $0$.

Given the construction of all subgroups of type $B_3$ contained in $P$ in Proposition \ref{prop:B3inE8gens} above, it is now routine to calculate $C_{Q^{X_0}}(X_{\mathbf{v}})$, since $Q^{X_0} = \langle z_i(c) : i =1, \ldots, 7\rangle$. In particular, notice that $a_1a_2a_5 + a_2a_3^2 = 0$ for $Z\mathbf{v} \in \mathbb{V}$ for each of our eight class representatives, and so the lift in the root group $U_{22343321}$ is trivial. 

We now show that $C_G(Y_1) = U_1 G_2$. We know that $C_{F_4}(Y_1) = U_1$ and $C_G(F_4) = G_2$. Since $\dim C_{L(G)}(Y_1) \le 15$, from Table \ref{tab:E8p2}, we are done. Next we consider $Y_2$, and start by noting that $C_G(A_6) = \bar{A}_1 T_1$. Furthermore, a routine calculation as above shows that the complements corresponding to $(0,0,0,1,0,0,0)$ and $(0,0,0,0,0,0,1)$ centralise $Q^{X_0}$. Therefore, $C_G(Y_2)^\circ$ contains a $7$-dimensional unipotent subgroup. Since $\dim C_{L(G)}(Y_2) = 10$ it follows that $C_G(Y_2) = U_6 A_1 T_1$. 

Next, we will find the connected centralisers of $Y_7 < D_8$ and $Y_8$. Recall that these centralisers are unipotent and so $Y_7$ and $Y_8$ must be conjugate to one of the subgroups corresponding to $(0,0,1,1,0,1,0)$ and $(0,1,0,1,1,0,0)$. We now check that the elements $z_2(c)$, $z_3(c)$, $z_5(c)$, $z_6(c)$, $z_7(c)$ centralise the complements corresponding to $(0,0,1,1,0,1,0)$ and $(0,1,0,1,1,0,0)$ and that $z_1(c) z_4(c)$ centralises the complement corresponding to $(0,0,1,1,0,1,0)$. We claim that $C_G(Y_7)^\circ = U_5$ and so $Y_7$ must be conjugate to the complement corresponding to $(0,1,0,1,1,0,0)$. To check the claim, we start by noting that $\dim{C_G(Y_7)} \le 6$, using the relevant restriction given in Table \ref{tab:E8p2}.  Since $Y_7 < D_8$, we have $\mathfrak{c} = C_{L(G)}(D_8) \subseteq C_{L(G)}(Y_7)$. Since the isogeny type of $D_8 < E_8$ is half-spin, it follows that $\mathfrak{c}$ is a $1$-dimensional subalgebra generated by a semisimple element. Since $C_G(Y_7)$ is unipotent, it follows that $\text{Lie}(C_G(Y_7))$ cannot contain $\mathfrak{c}$ and so $\text{Lie}(C_G(Y_7))$, hence $C_G(Y_7)$, has dimension at most $5$. It now follows that $Y_8$ is conjugate to the subgroup corresponding to $(0,0,1,1,0,1,0)$ and so $C_G(Y_8)^\circ$ has dimension at least $6$. Checking the restriction in Table \ref{tab:E8p2} shows that $\text{Lie}(C_G(Y_8))$ is $6$-dimensional and so $C_G(Y_8)^\circ = U_6$. 

The remaining subgroups $Y_3$, $Y_4$, $Y_5$, $Y_6$ are each conjugate to exactly one of the four complements corresponding to $(0,0,0,1,0,0,1)$, $(0,0,1,1,0,0,0)$, $(0,0,1,0,0,1,0)$ and $(0,1,0,0,1,0,0)$; recall that their centralisers all have rank $1$. A routine calculation shows that these complements centralise a $7$, $6$, $6$ and $4$-dimensional subgroup of $Q^{X_0}$, respectively. For the final two complements, only $5$ and $3$ generators of $Q^{X_0}$ given above, respectively, are in the centraliser. However, in both cases the product of two of the excluded generators makes up the final generator of the centraliser of the complements. 

Inspecting Table \ref{tab:E8p2}, we see that $\dim C_G(Y_i) = 10,9,8,6$ for $i=3,4,5,6$, respectively. Next, we show that $Y_3$ and $Y_4$ are not separable in $E_7T_1$, i.e.\ $\dim C_{E_7T_1}(Y_i) < \dim C_{L(E_7T_1)}(Y_i)$. Since $L(E_7T_1)$ is a direct summand of $L(E_8)$, it then follows from \cite[Theorem 1.5]{MR2608407} that $Y_3$ and $Y_4$ are not separable in $E_8$. From the calculations in $E_7$ we find $\dim C_{E_7T_1}(Y_3) = 3$ and $\dim C_{E_7T_1}(Y_4) = 2$. We can easily check that the corresponding centraliser dimension in $L(E_7T_1)$ is one more in each case, using the restrictions $L(E_8) \downarrow E_7T_1 = L(E_7T_1) + (V_{56},1) + (V_{56},-1) + (0,1) + (0,-1)$ and $V_{56} \downarrow Y_3, Y_4$ in Table \ref{tab:E7}. Thus $\dim C_G(Y_3) \le 9$ and $\dim C_G(Y_4) \le 8$. 

We claim that $C_{G}(Y_i)$ contains a subgroup $A_1$ in all four remaining cases. For $Y_3, Y_4$ this is clear, since they are contained in $E_7$, and $C_{G}(E_7)^{\circ} = \bar{A}_1$. The subgroup $Y_5$ acts as $T(100)$ on $V_{A_7}(\lambda_1)$, and since $T(100)$ is self-dual it follows that $Y_5 < C_4 < A_7$. By \cite[p.333 Table~3, Lemma~4.9]{MR1274094}, $C_G(C_4) = A_1$ and hence $A_1 \le C_G(Y_5)$. For $Y_6$, we note that $Y_6 < B_3^2 < B_6 < D_7$. Then, \cite[p.333 Table~3]{MR1274094} shows that $C_G(B_6) = A_1$ and hence $A_1 \le C_G(Y_6)$.

We now have enough information to conclude that $C_G(Y_3) = U_6 \bar{A}_1$,  $C_G(Y_4) = U_5 \bar{A}_1$,  $C_G(Y_5) = U_5 A_1$ and  $C_G(Y_6) = U_3 A_1$.      

\subsubsection{\texorpdfstring{$L'$}{L'} of type \texorpdfstring{$A_7$}{A7}}

Let $P = P_{1345678} = QL$ and $X_0 = B_3 < L' = A_7$ be embedded via $001$. Then the unipotent radical $Q$ has three levels, with $X_0$ actions
\begin{center}
\begin{tabular}{*{3}{c}}
$Q/Q(2)$ & $Q(2)/Q(3)$ & $Q(3)$ \\ \hline
$101 + 001$ & $0|100|010|100|0$ & $001$
\end{tabular}
\end{center}
By Proposition~\ref{prop:h1}, $H^{1}(X_0,Q(2)/Q(3)) = K$, so $\dim H^{1}(X_0,Q) \le 1$ and, considering the action of $Z(L)$, we find at most one non-$G$-cr complement to $Q$ in $QX_0$, up to conjugacy in $G$.

Let $Y$ be a subgroup of $B_3 < D_8$, embedded via $001 + 001$. By Proposition~\ref{prop:orthog_sum} and Remark~\ref{rem:a3d6-d4d8} there are two $D_8$-classes of such subgroups and both are non-$D_8$-cr. Now $D_8$ has two classes of Levi subgroup $A_7$, one of which is a Levi subgroup of $G$, and one of which is the subgroup $A_7'$. Considering the images of the non-$D_8$-cr subgroups in these Levi factors, then, one class arises from a $G$-cr subgroup contained in a subgroup $E_7$. So even if such subgroups are non-$G$-cr, they will arise from minimal parabolic subgroups with a Levi factor not conjugate to $L$. The other class, however, has irreducible image in $A_7$, which remains a Levi subgroup of $G$, so we only need to prove that these non-$D_8$-cr subgroups remain non-$G$-cr. Since $(Q/Q(2))^{X_0} = \{0\}$ and $Q(2)$ is the full unipotent radical of an $A_7$-parabolic of $D_8$, this follows from Corollary~\ref{cor:mcrgcr}.

By Lemma~\ref{lem:maxparab} we have $C_G(Y)^{\circ} = C_{Q}(Y)^{\circ}$. In an $A_7$-parabolic subgroup of $D_8$, the unipotent radical is the $28$-dimensional module $V_{A_7}(\lambda_2)$, and this is therefore a subgroup of $Q$ isomorphic to $Q(2)/Q(3)$. Since this has a $1$-dimensional space of fixed points under $Y$, we conclude that $C_{G}(Y)^{\circ} = U_1$.

\subsubsection{\texorpdfstring{$L'$}{L'} of type \texorpdfstring{$D_7$}{D7}}

In this case, $L'$ has irreducible subgroups of types $B_3$ and $B_3^2$ that we need to consider; we start with $B_3^2$. Let $P = P_{2345678} = QL$ and $X_0 = B_3^2 < L' = D_7$ via $0|( (100,0) + (0,100) )|0$. Then the unipotent radical $Q$ has two levels, with $Q/Q(2) \cong \lambda_6$ and $Q(2) \cong \lambda_1$ as $L'$-modules, and therefore $Q/Q(2) \downarrow X_0 \cong (001,001)$ and $Q(2) \downarrow X_0 = 0 | ((100,0) + (0,100)) | 0$.

From Proposition~\ref{prop:h1} we have $H^{1}(X_0,Q(2)) \cong K$ and $H^{1}(X_0,Q/Q(2)) = \{0\}$, and taking into account the action of $Z(L)$, we deduce that there is at most one non-$G$-cr complement to $Q$ in $QX_0$, up to conjugacy in $G$.

Consider the subgroup $Y = B_3^2 < D_8$ via the 16-dimensional $B_3^2$-module $(0|(100,0)|0) + (0|(0,100)|0)$. Each 8-dimensional factor has a nonsingular vector fixed by $B_3^2$. Thus the $2$-dimensional subspace spanned by these contains a singular vector fixed by $B_{3}$, and thus this subgroup $B_3^2$ is contained in a $D_7$-parabolic subgroup of $D_8$. Since this is the unique non-zero totally singular subspace fixed by $B_3^2$, we deduce that this subgroup is non-$D_8$-cr. Since $Q/Q(2)$ is an irreducible $Y$-module, Corollary \ref{cor:mcrgcr} applies, and $Y$ is also non-$G$-cr. A routine application of Lemma \ref{lem:maxparab} shows that $C_G(Y)^\circ = U_1$.

Now, if $X_0 = B_3 < L'$ via $100 \otimes 100^{[r]}$ ($r > 0$), identical arguments to the above show that $P$ contains at most one non-$G$-cr complement to $Q$ in $QX_0$, up to conjugacy. A representative is given by taking a subgroup $Z < D_8$ via $T(100) + T(100)^{[r]}$. Again, Corollary~\ref{cor:mcrgcr} tells us that this non-$D_8$-cr subgroup is non-$G$-cr, and Lemma~\ref{lem:maxparab} tells us that $C_{G}(Z)^{\circ} = U_1$. Note that, for each $r > 0$, the class representative $Z$ is a diagonal subgroup of $Y$.


\section{Type \texorpdfstring{$B_4$}{B4}} \label{sec:B4}

Let $P = P_{1345678} = QL$ and $X_0$ be of type $C_4$ embedded in $L'$ via $1000$. Then the action of $X_0$ on the levels of $Q$ are as follows: 
\begin{center}
\begin{tabular}{*{3}{c}}
$Q/Q(2)$ & $Q(2)/Q(3)$ & $Q(3)$ \\ \hline
$1000 + 0100$ & $\bigwedge\nolimits^{2}(1000)$ & $1000$
\end{tabular}
\end{center}
Since $0010$ and $\bigwedge^{2}(1000) = T(0100) = 0|0100|0$ are tilting, their first cohomology groups vanish, hence $\mathbb{V} \cong K^2$. The trivial submodule in level 2 lifts to a subgroup of $Q^{X_0}$, since $C_G(X_0)^\circ = A_1$ by \cite[p.333 Table~3]{MR1274094}. It follows easily that this induces a non-zero $X_0$-module homomorphism $Q/Q(2) \to Q(3)$. We deduce that complements are parametrised by $(k_1,k_2) \in \mathbb{V}$ with $k_1 k_2 = 0$. After considering the action of $Z(L)$, there are therefore at most two non-$G$-cr subgroups of type $B_{4}$ in such a parabolic subgroup, up to conjugacy.

Let $Y_1$ be the subgroup of type $B_4$ embedded in $A_8$ via $V_{A_8}(\lambda_1) \downarrow Y_1 = 1000 | 0$. Since $p \neq 3$ and $A_{8}$ is the centraliser of an element of order $3$, \cite[Corollary~3.1]{MR2178661} implies that $Y_1$ is non-$G$-cr. There are two $D_8$-conjugacy classes of subgroups of type $B_4$ acting as $V_{B_4}(\lambda_4)$ on the natural module for $D_8$. By \cite[Lemma 7.4]{Tho1}, one of these subgroups, denoted $B_4(\ddagger)$ in \emph{ibid}, is non-$G$-cr and non-conjugate to $Y_1$, let this subgroup be $Y_2$. Therefore $Y_1$ and $Y_2$ are representatives of two non-conjugate non-$G$-cr subgroup classes contained in $P$. 

Since $P$ is a maximal parabolic subgroup it follows from Lemma \ref{lem:maxparab} that $C_{G}(Y_1)^\circ$ and $C_{G}(Y_2)^\circ$ are equal to $Q^{Y_1}$ and $Q^{Y_2}$, respectively. In Table \ref{tab:E8p2} we see that $\dim C_{L(G)}(Y_i) = 1$ for $i=1,2$. Since $Y_2 < D_8$ and $C_{L(G)}(D_8)$ is a $1$-dimensional subalgebra generated by a semisimple vector, it follows that $C_G(Y_2)^\circ$ is trivial. Now, the complement to $Q$ in $P$ corresponding to $(0,1) \in \mathbb{V}$ centralises the lift of the trivial submodule in level $2$, since $Q(2)$ is abelian. It follows that $Y_1$ is conjugate to this non-$G$-cr subgroup and $C_G(Y_1)^\circ = U_1$.


\section{Type \texorpdfstring{$C_3$}{C3}} \label{sec:C3}

Recall that in Proposition~\ref{prop:h1}, this is the only case where $p = 3$. Let $P = P_{2345678} = QL$ be the unique standard $D_7$-parabolic subgroup of $G$ and let $X_0 = C_3 < L'$ via $010$. Then $Q/Q(2) \downarrow X_0 = 110 + 001$ and $Q(2) \downarrow X_0 = 010 + 0$. By Proposition~\ref{prop:h1}, we have $\mathbb{V} \cong K$ and hence by considering the action of $Z(L)$ we have at most one $G$-conjugacy class of non-$G$-cr complements to $Q$ in $QX_0$. 

Let $Y$ be a subgroup of type $C_3$ embedded in $D_8$ via $V_{D_8}(\lambda_1) = T(010) + 0$. Then $Y$ is non-$D_8$-cr and hence non-$E_8$-cr by Lemma \ref{lem:BMR}, since $D_8$ is the centraliser of an involution in $E_8$ when $p \neq 2$. It follows that $Y$ must be contained in $P$ since it is the only parabolic that possibly contains non-$G$-cr subgroups of type $C_3$ when $p=3$.    

Since $P$ is maximal, using Lemma~\ref{lem:maxparab} we have $C_{G}(Y)^{\circ} = C_{Q}(X_0) = Q(2)^{X_0} = U_1$.


\section{Type \texorpdfstring{$C_4$}{C4}} \label{sec:C4}

The case $G = E_7$ is covered by \cite[Lemma 2.7]{MR1367085}, yielding one conjugacy class of non-$G$-cr subgroups, with representative $Y = C_4 < A_7$ embedded via $1000$. The connected centraliser of $Y$ is $U_1$, as given in \cite[p.333 Table~3]{MR1274094}.  So we may assume that $G = E_8$.

Let $P = P_{123456} = QL$ be the unique standard $E_6$-parabolic subgroup of $G$ and let $X_0$ of type $C_4$ be a representative of the unique class of $E_6$-irreducible subgroups. Then $X_0$ acts on the levels of $Q$ as follows: 
\begin{center}
\begin{tabular}{*{5}{c}}
$Q/Q(2)$ & $Q(2)/Q(3)$ & $Q(3)/Q(4)$ & $Q(4)/Q(5)$ & $Q(5)$ \\ \hline
$0100 + 0^2$ & $0100 + 0$ & $0100 + 0$ & $0$ & $0$
\end{tabular}
\end{center}

By Proposition~\ref{prop:h1}, we have $\mathbb{V} \cong K^3$. The images of $U_{\alpha_7}$, $U_{\alpha_7 + \alpha_8}$, and $U_{01122221}$ generate the $X_0$-module $0100$ in levels 1, 2 and 3, respectively. Moreover, three trivial summands in levels 1 and 2 are respectively generated by the image of the elements $x_{\alpha_8}(c)$,  $x_{11221110}(c)$ $x_{11122110}(c)$ $x_{01122210}(c)$ and $x_{11221111}(c)$ $x_{11122111}(c)$ $x_{01122211}(c)$. It then follows that we may assume that $k_1 k_2 = k_1 k_3 = k_2 k_3 = 0$. The element $n_{8} \in N_G(T)$ centralises $X_0$ whilst swapping $U_{\alpha_7}$ and $U_{\alpha_7 + \alpha_8}$, and the element
\[ n_7 n_6 n_5 n_4 n_3 n_2 n_4 n_5 n_6 n_7 n_1 n_3 n_4 n_5 n_6 n_2 n_4 n_5 n_3 n_4 n_1 n_3 n_2 n_4 n_5 n_6 n_7\]
centralises $X_0$ whilst swapping $U_{\alpha_7 + \alpha_8}$ and $U_{01122221}$. We therefore conclude that there is at most one $G$-conjugacy class of non-$G$-cr complements to $Q$ in $Q X_0$. 

Let $Y$ be the non-$G$-cr subgroup of type $C_4$ contained in $E_7$. Then $Y$ is non-$G$-cr by Lemma \ref{lem:BMR} and hence a conjugate of $Y$ is contained in $P$. The connected centraliser of $Y$ is calculated in \cite[Lemma~4.9]{MR1274094} and is $U_5 A_1$.


\section{Type \texorpdfstring{$D_4$}{D4}} \label{sec:D4}

In this section we complete our analysis of non-$G$-cr subgroups by considering those of type $D_4$, so that according to Lemma~\ref{lem:subtypes} we have $p = 2$ and $G = E_{7}$ or $E_{8}$.

\subsection{\texorpdfstring{$G$}{G} of type \texorpdfstring{$E_7$}{E7}} \label{subsec:d4e7}

In this case, according to \cite[Theorem 1]{MR1367085}, there are infinitely many classes of subgroups of type $D_{4}$ contained in an $E_{6}$-parabolic subgroup of $G$, each having irreducible image in the Levi factor. Exactly one such class consists of $G$-completely reducible subgroups. Let $P = QL$ be an $E_{6}$-parabolic subgroup. Then $Q$ is abelian, and is an irreducible $L'$-module of high weight $\lambda_1$. Thus the non-$G$-cr subgroups of type $D_{4}$ occurring are complements to $Q$ in $QX_0$, where $X_0$ is generated by the short root subgroups of $F_{4} < L'$. We have $Q \downarrow X_0 = 0100 + 0$, hence $H^{1}(X_0,Q)$ is $2$-dimensional. Fixing a basis so that each $(a,b) \in K^2$ gives a complement to $Q$ in $QX_0$, the action of the $1$-dimensional torus $Z(L)$ shows that complements corresponding to $(a\mu, b\mu)$, as $\mu$ varies over $K^{\ast}$, are all $G$-conjugate. Next note that $L'$ contains a subgroup $D_4.S_3$. This outer automorphism group $S_3$ acts on $Q$ and thus on the collection of complements to $Q$ in $QX_0$. Concretely, this can be realised as an element of order $3$ sending a complement corresponding to $(a,b)$ to a complement corresponding to $(b,a+b)$, and an element of order $2$ swapping complements corresponding to $(a,b)$ and $(b,a)$. Thus all $G$-classes of non-$G$-cr complements correspond to a pair $(1,a)$, as $a$ varies over $K$. 

Next, we consider potential overgroups of these non-$G$-cr subgroups. First, we note that by Lemma \ref{lem:rankofcents}, the connected centraliser of every such subgroup is unipotent. Suppose that $M$ is a reductive maximal connected subgroup of $G$ containing a non-$G$-cr subgroup $Z$ of type $D_4$. Then each simple factor of $M$ must have rank at least $4$, otherwise $Z$ would project trivially to some simple factor $M_0$ and give $M_0 < C_G(Z)$. Thus $M$ must be of type $A_7$ by \cite[Theorem~1]{MR2063402}. Let $Y = D_{4} < A_{7} < G$ via $1000$. Since $N_{G}(A_7)/A_7$ has order $2$, $Y$ centralises an involution in $G$ and thus lies in a parabolic subgroup of $G$. From Table \ref{tab:E8p2}, we see that $Y$ acts with two $28$-dimensional indecomposable summands on $V_{56}$. It follows that $Y$ is not contained in any Levi subgroup of $G$ and is therefore non-$G$-cr.

Thus we have shown that precisely one of the non-$G$-cr subgroups is contained in a proper reductive overgroup in $G$, and the remaining infinitely many classes are MR.

It remains to calculate the connected centralisers. As $Q$ is abelian, $Q^{X_0} = Q^Z$ for any non-$G$-cr subgroup $Z$ contained in $P$. It follows from Lemma \ref{lem:maxparab}\ref{maxparab-ii} that $C_G(Z)^\circ = U_1$. 

\subsection{\texorpdfstring{$G$}{G} of type \texorpdfstring{$E_8$}{E8}} \label{sec:D4inE8}

\subsubsection{\texorpdfstring{$L'$}{L'} of type \texorpdfstring{$A_7$}{A7}} 
Let $P = QL$ with $L'$ of type $A_{7}$ and $X_0 < L'$ embedded via $1000$. Then $Q$ has three levels, with $X_0$-actions
\begin{center}
\begin{tabular}{ccc}
$Q/Q(2)$ & $Q(2)/Q(3)$ & $Q(3)$ \\ \hline
$1000 + 0011$ & $0|0100|0$ & $1000$
\end{tabular}
\end{center}

Proposition~\ref{prop:h1} implies that $\mathbb{V} \cong K$ and thus there exists at most one class of non-$G$-cr complement to $Q$ in $QX_0$. By Proposition \ref{prop:orthog_sum} and Remark \ref{rem:a3d6-d4d8}, there are two $D_8$-classes of non-$D_8$-cr subgroups of type $D_4$ which act as $1000 + 1000$ on $V_{D_8}(\lambda_1)$. Embedding $D_8 < G$, since the subgroups lie in a proper parabolic subgroup of $D_8$, they lie in a proper parabolic subgroup of $G$. We now prove that these subgroups are both non-$G$-cr. We start by noting that the action of each subgroup on $L(G)$, given in Table \ref{tab:E8p2}, has three indecomposable summands of dimension $64$ and no indecomposable summand of dimension less than $28$. For a contradiction, suppose that $Y$ is a subgroup in either class and that $L$ is minimal among proper Levi subgroups of $G$ containing $Y$. Since $L'$ has rank at most $7$ and $Y$ has rank $4$, it follows that $L'$ is simple, and in fact $L'$ has type $A_7$, $D_4$, $D_5$, $D_6$, $D_7$, $E_6$ or $E_7$. Knowledge of high weights on $L(G)$ \cite[Table~10.1]{MR1329942} shows that only $D_7$ can potentially have a subgroup acting with three indecomposable summands of dimension $64$. But in this case, $L(G) \downarrow D_7 = \lambda_1^2 + T(\lambda_2) + \lambda_6 + \lambda_7$ and any subgroup of $D_7$ acts on $L(G)$ with an indecomposable summand of dimension at most $14$. This contradiction proves that both classes of non-$D_8$-cr subgroups of type $D_4$ are in no proper Levi subgroup of $G$, hence are non-$G$-cr.

The two class representatives have different composition factors on $L(G)$. By Proposition~\ref{prop:h1}, only $A_7$-parabolics and $E_6$-parabolics can contain non-$G$-cr subgroups of type $D_4$, and it follows that each class of parabolics gives rise to precisely one of these non-$G$-cr subgroup classes. The fact that the subgroup $X_0 < A_7$ has a composition factor of high weight $0011$ shows that the first subgroup in Table \ref{tab:E8p2} is contained in $P$.

Finally, we calculate $C_{G}(Y)^{\circ}$ for $Y$ in this class of non-$G$-cr subgroups. Since $P$ is maximal, we can apply Lemma \ref{lem:maxparab}. There is a $C_4$ overgroup of $X_0$ in $A_7$ and $C_G(C_4)^\circ = A_1$ by \cite[p.333 Table~3]{MR1274094}). It follows that the trivial submodule in $Q(2)/Q(3)$ lifts to a subgroup of $Q$ and so $Q^{X_0}$ is $1$-dimensional. Moreover, $Y < Q(2) X_0$ and $Q^{X_0} \le Q(2)$ which is abelian, so $Q^{Y} = Q^{X_0}$. Hence $C_G(Y) = U_1$. 

\subsubsection{\texorpdfstring{$L'$}{L'} of type \texorpdfstring{$E_6$}{E6}}
Let $P = QL$ be the standard $E_{6}$-parabolic of $G$ and let $X_0$ be the $E_6$-irreducible subgroup of type $D_4$ generated by short root subgroups in an $F_4$ subgroup in $E_6$. Then $Q$ has five levels, with actions of $L'$ and $X_0$ as follows. 
\begin{center}
\begin{tabular}{*{6}{c}}
& $Q/Q(2)$ & $Q(2)/Q(3)$ & $Q(3)/Q(4)$ & $Q(4)/Q(5)$ & $Q(5)$ \\ \hline
$L'$ & $\lambda_1 + 0$ & $\lambda_1$ & $\lambda_6$ & $0$ & $0$ \\
$X_0$ & $0100 + 0^2$ & $0100 + 0$ & $0100 + 0$ & $0$ & $0$
\end{tabular}
\end{center}
The trivial summand in $Q/Q(2)$, generated by the image of the root group $U_{\alpha_8}$, induces a non-zero $E_6$-module homomorphism $Q/Q(2) \to Q(2)/Q(3)$. In the action of $X_0$ on $\lambda_1$ and $\lambda_6$, the modules of high weight $0100$ in the three levels are respectively generated by the images of the root groups $U_{\alpha_7}$, $U_{\alpha_7 + \alpha_8}$ and $U_{01122221}$. Note that $X_0$ is contained in a subgroup $Y_0 = C_{4} < E_{6}$, and as described in Section~\ref{sec:C4}, the additional trivial modules arising in $Q/Q(2)$ and $Q(2)/Q(3)$ respectively induce non-zero $Y_0$-module homomorphisms $\psi : Q(2)/Q(3) \to Q(3)/Q(4)$ and $\psi' : Q/Q(2) \to Q(3)/Q(4)$.

Now $H^{1}(X_0,0100) \cong K^{2}$, hence $\mathbb{V} \cong K^{6}$ and we denote elements of $\mathbb{V}$ by $(\phi_{1},\phi_{2},\phi_{3})$, where $\phi_{i} \in H^{1}(X_0,Q(i)/Q(i+1))$. We identify elements in the different copies of $H^{1}(X_0,0100)$ such that $x_{\alpha_8}(c)$ induces the map \[(\phi_1,\phi_2,\phi_3) \mapsto (\phi_1,\phi_2 + c\phi_1, \phi_3)\] and the additional trivial modules from $Q/Q(2)$ and $Q(2)/Q(3)$ respectively induce maps
\[ (\phi_1,\phi_2,\phi_3) \mapsto (\phi_1,\phi_2,\phi_3 + c\phi_2), \qquad (\phi_1,\phi_2,\phi_3) \mapsto (\phi_1,\phi_2,\phi_3 + c\phi_1). \]
Again, since $X_0 < Y_0$ it follows that $X_0$ is centralised by $n_8 \in N_G(T)$, which swaps the root groups $U_{\alpha_7}$ and $U_{\alpha_7 + \alpha_8}$, and centralises $U_{01122221}$. Moreover, the element
\[ w = n_7 n_6 n_5 n_4 n_3 n_2 n_4 n_5 n_6 n_7 n_1 n_3 n_4 n_5 n_6 n_2 n_4 n_5 n_3 n_4 n_1 n_3 n_2 n_4 n_5 n_6 n_7 \]
also centralises $X_0$, and swaps $U_{\alpha_7 + \alpha_8}$ and $U_{01122221}$ whilst mapping $U_{\alpha_7}$ to a negative root subgroup. Similarly, the element
\[ w' = n_{7} n_{6} n_{5} n_{4} n_{2} n_{3} n_{1} n_{4} n_{3} n_{5} n_{4} n_{2} n_{6} n_{5} n_{4} n_{3} n_{1} n_{7} n_{6} n_{5} n_{4} n_{2} n_{3} n_{4} n_{5} n_{6} n_{7}  \]
centralises $X_0$, swaps $U_{\alpha_7}$ with $U_{01122221}$ and maps $U_{\alpha_7 + \alpha_8}$ to a negative root subgroup.

We now use the above elements to parametrise complements to $Q$ in $QX_0$, including $X_0$ itself, by elements of $\mathbb{P}^1 \cup \{ \text{3 points} \}$. To do this, we divide into a number of cases.

In the case that $(\phi_1,\phi_2) = (0,0)$, conjugating by $w'$ sends the corresponding complement to one corresponding to $(\phi_3,0,0) \in \mathbb{V}$. Such a subgroup is contained in a Levi subgroup of $G$ of type $E_7$, and is non-$E_7$-cr. We will shortly show that such subgroups are $G$-conjugate if and only they are $E_7$-conjugate, hence there are infinitely many such subgroups, as described in Section~\ref{subsec:d4e7}.

Now suppose that $(\phi_1,\phi_2) \neq (0,0)$ but that $\phi_1$ and $\phi_2$ are multiples of one another under the above identification of the copies of $H^{1}(X_0,0100)$. Then, conjugating by $n_8$ if necessary, we can assume that $\phi_1 \neq 0$. Conjugating by an appropriate element $x_{\alpha_8}(c)$ then allows us to assume that $\phi_2 = 0$. This complement therefore corresponds to an element of the form $(\phi_1,0,\phi_3) \in \mathbb{V}$ where $\phi_1 \neq 0$. If $\phi_3$ is a multiple of $\phi_1$ under the identification above, then conjugating by an appropriate $X_0$-fixed point of $Q(2)/Q(3)$ shows that the complement corresponds to $(\phi_1,0,0)$, which again lies in an $E_7$ Levi subgroup of $G$ as above. Thus we may assume that $\{\phi_1,\phi_3\}$ form a basis of $H^{1}(X_0,0100)$. Now the element $w'$ together with the $X_0$-fixed points of $Q(2)/Q(3)$ and $Z(L)$, fuse together all classes corresponding to any such basis of $H^{1}(X_0,0100)$ (this is essentially the fact that ${\rm GL}_2(K)$ is transitive on ordered bases of its natural module). Thus this case gives rise to at most one new class of subgroups. In Lemma~\ref{lem:d4actione8} we will show that the indecomposable direct summands of such a group on $L(G)$ are incompatible with $L(G) \downarrow E_7$, so that these subgroups are not contained in a subgroup of type $E_7$.

Finally, suppose that $\{\phi_1,\phi_2\}$ form a basis of $H^{1}(X_0,0100)$. Conjugating by the $X_0$-fixed points which give rise to the maps $\psi$ and $\psi'$ above, we may assume that $\phi_3 = 0$. Then the element $n_8$, together with $U_{\alpha_8}$ and $Z(L)$, again fuses together all classes corresponding to triples $(\phi_1,\phi_2,0)$ where $\{\phi_1,\phi_2\}$ is a basis of $H^1(X_0,0100)$. Thus in this case we find at most one more non-$G$-cr subgroup, up to conjugacy. Again, in Lemma~\ref{lem:d4actione8} we derive the actions of these subgroups on $L(G)$, and see that these are incompatible with the subgroup being contained in $E_7$, or being conjugate to any of the non-$G$-cr subgroups found above. Thus these do indeed give rise to a distinct class of non-$G$-cr subgroups.

In summary, supposing that all potential complements above actually exist, we have shown that non-$G$-cr subgroups correspond to one of: $(\phi,0,0)$ with $\phi$ a non-zero point of $H^{1}(X_0,0100)$; or to $(\phi,0,\psi)$ or to $(\phi,\psi,0)$ for a fixed basis $\{\phi,\psi\}$ of $H^{1}(X_0,0100) \cong K^2$. In the former case, these are parametrised by elements $(a,b) \in K^2$ modulo the action of $Z(L)$ and a subgroup $S_3$, as in \ref{subsec:d4e7} above. In each of the latter two cases we get at most one extra non-$G$-cr subgroup.

Let us first consider the infinitely many non-$G$-cr subgroups corresponding to elements $(\phi,0,0) \in \mathbb{V}$, each contained in a Levi subgroup $M$ of type $E_7$. We will show that two such subgroups are $G$-conjugate if and only if they are $M'$-conjugate. Let $Y$ be such a non-$G$-cr subgroup. To start, note that $M'$ centralises a connected simple subgroup $\bar A_1$ (generated by root subgroups corresponding to $\pm \alpha_0$), hence so does $Y$. Moreover, note that conjugation by $w'$ sends $Y$ to a subgroup corresponding to $(0,0,\phi) \in \mathbb{V}$; since this is contained in $Q(3)X_0$, it follows that $Y$ commutes with the $6$-dimensional unipotent subgroup $Q^{X_0}$ (the induced $Y$-module homomorphisms $Q(3) \to Q$ induced by elements of $Q^{X_0}$ are all zero as their image lies in $Q(4)$, and neither $Q(4)/Q(5)$ nor $Q(5)$ have a non-trivial composition factor). By Lemma \ref{lem:unip-cent}, $Q^{X_0}$ is a maximal unipotent subgroup of $C_G(Y)$, for all subgroups $Y$ corresponding to an element $(0,0,\phi) \in \mathbb{V}$. And Lemma \ref{lem:rankofcents} shows that $C_G(Y)$ is rank $1$. Since $Q^{X_0}$ meets the subgroup $\bar A_1$ in a $1$-dimensional connected unipotent group, we conclude that $C_G(Y)^\circ = U_5 \bar A_1$. 

Thus if $Y$ and $Z$ are two non-$G$-cr complements to $Q$ in $QX_0$, each contained in $M'$ (a fixed copy of $E_7$), we have $C_G(Y)^{\circ} = C_G(Z)^{\circ} = UH$, where $U$ is a $5$-dimensional unipotent subgroup and $H$ is simple of type $A_1$. Now suppose $g \cdot Y = Z$ for some $g \in G$. Then $g$ normalises $C_G(Y)^{\circ}$ and so $H$ and $g \cdot H$ are simple subgroups of $UH$ of type $A_1$. A quick calculation shows that $UH$ is $G$-conjugate to the subgroup generated by $Q^{X_0}$ and $U_{-\alpha_8}$, with the subgroup $A_1$ generated by the root subgroups $U_{\pm \alpha_8}$. Using the commutator relations in $G$, it follows that $U$ has a filtration by $H$-modules of high weight $0$ and $1$. Each of these modules has zero first cohomology group, hence $H^{1}(H,U)$ vanishes and $g \cdot H$ is $U$-conjugate to $H$ (note that although $H^1(H,1^{[1]})$ is $1$-dimensional, the resulting non-trivial complements cannot be $G$-conjugate to $H$ since they are not algebraically isomorphic to $H$; the projection $UH \to H$ has a scheme-theoretic kernel on restriction to $H$). Therefore $(ug) \cdot H = H$ for some $u \in U$. Since $U \le C_G(Y) = C_G(Z)$, we have $(ug) \cdot Y = u \cdot Z = Z$, and $ug \in N_G(H) = H E_7$. Finally, this means we may take some $h \in H$ such that $(hug) \cdot Y = h \cdot Z = Z$ with $hug \in E_7$; hence $Y$ and $Z$ are conjugate under an element of $E_7$.

We now claim there are indeed two distinct classes of non-$G$-cr subgroups of type $D_4$ contained in $P$ but not contained in $E_7$; we exhibited the first of these in the previous case of $P_{1345678}$. Indeed, there is a non-$D_8$-cr subgroup $Y_1$ acting as $1000 + 1000$ on $V_{D_8}(\lambda_1)$ which is contained in $P$. Now let $Y_2$ be the subgroup generated by the set of eight Chevalley generators $y_{\pm \beta_i}(t)$ for $i=1, \ldots, 4$ in Proposition \ref{prop:D4inE8gens} with $a_1=a_4=1$ and $a_i = 0$ for $i = 2,3,5,6$. This is a subgroup of type $D_4$ contained in $P_{1345678}$ by construction. Furthermore, in Lemma~\ref{lem:d4actione8}, we show that $Y_2$ acts with an indecomposable summand of dimension $188$. No maximal reductive subgroup of $G$ or proper Levi subgroup of $G$ has an indecomposable summand of this dimension or higher, hence $Y_2$ is non-$G$-cr and maximal among proper connected reductive subgroups of $G$. 

Finally we calculate $C_{G}(Y_1)^\circ$ and $C_{G}(Y_2)^\circ$. These are unipotent groups because neither $Y_1$ nor $Y_2$ is contained in a Levi subgroup of $G$. Inspecting Table \ref{tab:E8p2}, we find that $\dim C_G(Y_1) \le 5$. Since $Y_1 < D_8$, it follows (as for the subgroup $Y_7$ in Section \ref{subsec:b3ine8a5}) that $\dim C_G(Y_1) \le 4$. Again inspecting Table \ref{tab:E8p2}, we find that $\dim C_G(Y_2) \le 5$. We know from the above analysis that these two extra subgroups come from the cases $(\phi_1,0,\phi_3) \in \mathbb{V}$, which are conjugate to $(0,\phi_1,\phi_3) \in \mathbb{V}$, and $(\phi_1,\phi_2,0) \in \mathbb{V}$, where $\phi_1, \phi_3$ and $\phi_1,\phi_2$ form a basis of $K^2$, respectively (though we do not yet know which corresponds to $Y_1$ or $Y_2$). It is routine to check that a subgroup corresponding to $(0,\phi_1,\phi_3)$ centralises a $5$-dimensional subgroup of $Q^{X_0}$, and such a subgroup is therefore conjugate to $Y_2$ and $C_G(Y_2)^\circ = U_5$. Another routine check shows that a subgroup corresponding to $(\phi_1,\phi_2,0)$ centralises a $4$-dimensional subgroup of $Q^{X_0}$; this is therefore a conjugate of $Y_1$ and $C_G(Y_1)^\circ = U_4$.


\section{Actions of non-\texorpdfstring{$G$}{G}-cr subgroups} \label{sec:actions}

In this section we determine the restrictions $L(G) \downarrow X$ and $\Vmin \downarrow X$ for the non-$G$-cr subgroups $X$ arising in Theorem~\ref{THM:MAIN}. When $X$ is not of type $D_4$, it has a proper reductive overgroup in $G$ and we use this to deduce the required restrictions. More details about this method can be found in \cite[Section~10]{Litterick2018}. The case when $X$ has type $D_4$ and $p = 2$ is by far the most difficult to analyse; the remainder of this section is dedicated to giving a full treatment of this scenario.

So now let $p = 2$ and suppose firstly that $G = E_7$, with $X$ contained in $P = QL$ with $L'$ of type $E_6$, such that the image $X_{0}$ of $X$ in $L$ lies in a subgroup of type $F_4$ and is generated by short root subgroups of this. Recall from Section~\ref{sec:D4} that $Q$ is abelian, with $Q \downarrow X_0 = 0100 + 0$, and the $G$-conjugacy class of $X$ corresponds to $\phi \in K^2$ where we can take $\phi = (1,a)$ with $a \neq 0$. Furthermore, precisely one of these infinitely many classes consists of subgroups lying in a proper reductive overgroup of type $A_7$, and the rest are MR (maximal among proper connected reductive subgroups of $G$).

\begin{lemma} \label{lem:b3-in-d4}
In the above set-up, let $Y < X$ be a subgroup of type $B_3$. Then $Y$ is non-$G$-cr, and is conjugate to the subgroup $B_3 < A_7$ via $001$ given in Table~\ref{tab:E7}.
\end{lemma}

\begin{proof}
Let $X_0$ be the ($G$-cr) image of $X$ in the Levi subgroup $L$, and let $Z_0 < Y_0 < X_0$ where $Z_0$ is simple of type $A_{3}$ and $Y_{0}$ is the image of $Y$ (simple of type $B_3$). We will show that $Z_0$ can be chosen such that the restriction map $H^{1}(X_0,0100) \to H^{1}(Z_0,0100)$ sends the non-zero cohomology class giving rise to $X$, to a non-zero class. It then follows that the corresponding complement, $Z$, is non-$G$-cr, and also restriction of this class to $H^{1}(Y_0,0100)$ is non-zero, so that $Y$ is also non-$G$-cr.

In a simply connected group, the derived subgroup of a Levi factor is again simply connected. Therefore, the Lie algebra of an $A_3$-Levi subgroup in a simply connected group of type $D_4$ has socle series $101|0$ as an $A_3$-module. Since $L(A_3)$ is a submodule of $L(D_4) = 0100|0^2$, and since $H^{1}(D_4,0100) \cong K^2$ and $H^{1}(A_3,101) \cong K$ it follows that the restriction map $H^{1}(D_4,0100) \to H^{1}(L',0100)$ has a $1$-dimensional kernel, for each choice of $A_3$-Levi subgroup $L$. Moreover the group of graph automorphisms $S_3$ acts on the centre of $L(D_4)$ and on $H^{1}(D_4,0100)$, giving each the structure of an irreducible $2$-dimensional $S_3$-module. In particular this action preserves no $1$-space, and it follows that for each non-trivial cohomology class in $H^{1}(D_4,0100)$, there is a simple subgroup $A_3$ whose class is not in the kernel of the restriction map.

As above, let $Z_0 < Y_0 < X_0$ where $Y_0$ has type $B_3$ and $Z_0$ has type $A_3$. Choosing $Z_0$ so that the cohomology class giving rise to $X$ does not lie in the kernel of $H^{1}(X_0,Q) \to H^{1}(Z_0,Q)$, the class remains non-trivial for both $Y_0$ and $Z_0$. Therefore the corresponding complements $Z < Y < X$ are respectively not $Q$-conjugate to $Z_0$, $Y_0$ and $X_0$, hence are non-$G$-cr. Considering the composition factors of $Z$ and $Y$ on $L(G)$ and on $\Vmin$, we identify $Z$ as the subgroup in Table~\ref{tab:E7} with $\Vmin \downarrow Z = 010^{4} + T(101)^2$. Then $\dim (\Vmin^{Z}) = 2$. Looking at the first three subgroups $B_3$ in Table~\ref{tab:E7}, the corresponding subgroups $A_3$ will have a four-dimensional fixed-point space on $\Vmin$, and therefore $Y$ is a conjugate of the final subgroup $B_3$, embedded in the subsystem subgroup of type $A_7$ of $G$ via $001$, as claimed. 
\end{proof}

\begin{proposition} \label{prop:D4restrictionMin} In the above set-up:
\begin{enumerate}
	\item If $X < A_7$ then $\Vmin \downarrow X = (0|\lambda_2|0)^2$. \label{prop:d4-action-i}
	\item If $X$ is MR then $\Vmin \downarrow X = 0^2 | \lambda_2^2 | 0^2$ is indecomposable. \label{prop:d4-action-ii}
\end{enumerate}
\end{proposition}

\begin{proof}
The restriction for the subgroup $X < A_7$ follows readily since $A_7$ acts on $\Vmin$ with two direct summands: the alternating squares of the natural $8$-dimensional $A_7$-module and its dual. 

So suppose that $X$ is MR and take a subgroup $Y < X$ of type $B_3$; by Lemma~\ref{lem:b3-in-d4} this is a non-$G$-cr subgroup contained in a subsystem subgroup $H$ of type $A_7$, via $001$. Inspecting Table~\ref{tab:E7}, the module $\Vmin \downarrow Y$ is a direct sum of two indecomposable $28$-dimensional modules, both of which are $H$-submodules. Since $H$ is maximal among proper connected subgroups of $G$, it follows that $H$ is the identity component of the full stabiliser in $G$ of this direct-sum decomposition of $\Vmin$. Since $X$ is not contained in such a subsystem subgroup of $G$, we see that $X$ cannot stabilise such a direct-sum decomposition and so $\Vmin \downarrow X$ is indecomposable as claimed. To obtain the socle series in this case, note that $\Vmin \downarrow X$ is self-dual with composition factors $0^{4}/\lambda_2^2$. There are no uniserial $X$-modules of shape $\lambda_2|0^{i}|\lambda_2$ for $i \ge 0$, since these would be generated by a highest weight vector and would hence be an image of the Weyl module $\lambda_2|0^2$, which is absurd. Since $\Vmin \downarrow X$ is indecomposable, it follows that it cannot have any irreducible quotient $\lambda_2$; since $\Ext^{1}(\lambda_2|0^{2},0^{2}|\lambda_2) = \{0\}$ by \cite[Proposition 4.13]{MR2015057}, the stated socle series is the only remaining possibility.
\end{proof}

\begin{proposition} \label{prop:D4restrictionAdj} In the above set-up:
\begin{enumerate}
		\item If $X < A_7$ then \label{prop:d4-ad-action-i}
		\[ L(G) \downarrow X = (\lambda_2|(2\lambda_1 + 0^2)|\lambda_2) +  (\lambda_2|(2\lambda_3 + 2\lambda_4 + 0)|(\lambda_2+0)) + 0. \]
	\item \label{prop:d4-ad-action-ii} If $X$ is MR then $L(G) \downarrow X$ is indecomposable with socle series 
	\[ \lambda_2^2|(2\lambda_1 + 2\lambda_3 + 2\lambda_4 + 0^3)|(0^2 + \lambda_2^2). \]
\end{enumerate}
\end{proposition}

\begin{proof}
In Proposition~\ref{prop:D4restrictionMin} we showed that the subgroups $X$ in \ref{prop:d4-ad-action-i} lie in a subgroup $A_7$, which acts on $L(G)$ as a direct sum $V_{A_7}(\lambda_1 + \lambda_7) + V_{A_7}(\lambda_4) + 0$. Take $Y < X$ of type $B_3$, recalling from Lemma~\ref{lem:b3-in-d4} that this is a conjugate of the final subgroup in Table~\ref{tab:E7}). The indecomposable summands of $X$ on $L(G)$ have the dimensions of the stated modules. The module structure itself follows at once from inspecting $L(G) \downarrow Y$ and comparing this with the action of $Y$ on each $X$-composition factor.

For \ref{prop:d4-ad-action-ii}, we again take $Y < X$ of type $B_3$ and consider its action on $L(G)$ in Table~\ref{tab:E7}. This has three indecomposable summands, denoted $M_1 = Z(L(A_7)) = Z(L(G))$, $M_{62}$ and $M_{70}$, where the subscript denotes the dimension. Now, $Y$ lies in a subsystem subgroup $H$ of type $A_7$ in $G$, and the stabilisers of $M_{62}$, $M_{70}$ and $M_{62}+M_{70}$ all contain $H$. Since $H$ is maximal in $G$, it in fact equals each of these stabilisers. However, we proved in Section~\ref{sec:D4} that $X$ is not contained in any subgroup of type $A_7$, so $X$ does not stabilise $M_{62}$, $M_{70}$ or $M_{62} + M_{70}$. This shows that $L(G) \downarrow X$ is indecomposable. Similarly, since $X$ does not stabilise either of the indecomposable $Y$-direct summands of the $132$-dimensional $A_7$-module $L(G)/M_1 \cong M_{62} + M_{70}$, we see that $X$ is indecomposable on $L(G)/Z(L(G))$. Finally, consider the $G$-module $T(\lambda_2)$. This has a $133$-dimensional indecomposable submodule $L(G)$, hence $X$ is either indecomposable on $T(\lambda_2)$ or $X$ preserves a $1$-dimensional complement to $L(G)$. However, note that as an $H$-module, $T(\lambda_2)$ has indecomposable summands of dimension $1$, $1$, $63$ and $70$, and the subgroup $Y$ preserves only these direct summands and no more. Thus each $X$-module complement to $L(G)$ in $T(\lambda_2)$ is an $H$-submodule, and again $H$ is the full stabiliser of such a submodule. Since $X$ is not contained in a subgroup of type $A_7$, we again conclude that $T(\lambda_2)$ is an indecomposable $X$-module.

To determine the socle series, we first work with the $G$-module $T_{E_7}(\lambda_2)$, which is self-dual, indecomposable and contains $L(G)$ as a submodule. To begin, since $T_{E_7}(\lambda_2)$ has exactly one $X$-composition factor of each high weight $2\lambda_1$, $2\lambda_3$ and $2\lambda_4$, none of these can appear as a submodule or quotient, since the self-duality of $T_{E_7}(\lambda_2)$ would imply the existence of a direct summand with one of these weights.

Next, $L(G)$ has a $1$-dimensional centre since $p = 2$ and $G$ is simply connected, and $L(Q) \downarrow X = 0 + \lambda_2$, which does not meet $Z(L(G))$ and thus furnishes a second trivial submodule and a submodule $\lambda_2$; these latter submodules consist entirely of nilpotent elements of $L(G)$. The subgroup $Y < X$ does not have a $3$-dimensional trivial submodule in its action on $L(G)$, and we conclude that the fixed-point space of $X$ on $L(G)$ is $2$-dimensional. Next, note that for a rational cocycle $\phi$, the maps $x \mapsto \phi(x)x$ and $\phi(x)x \mapsto x$ are mutually inverse isomorphisms of algebraic groups between $X$ and its image in the $E_6$-Levi subgroup of $G$. This image in the Levi subgroup is of adjoint type, since its action on $V(\lambda_7)$ is a direct sum $0^2 + \lambda_2^2$. Thus $X$ is adjoint, so $L(G)$ has a submodule $L(X) = 0^2 | \lambda_2$, and moreover this submodule $\lambda_2$ does not consist entirely of nilpotent elements, hence is not equal to the submodule $\lambda_2$ furnished by $L(Q)$. Now, note that since $T_{E_7}(\lambda_2)$ has exactly four $X$-composition factors of high weight $\lambda_2$, we cannot have three of these in the socle, since this would imply the existence of a proper direct summand. Thus $L(G) \downarrow X$ has socle $0^2 + \lambda_2^2$. It follows that $T(\lambda_2)$ also has socle $0^2 + \lambda_2$, otherwise it would have socle $0^3 + \lambda^3$, which would imply that $T(\lambda_2) = L(G) + 0$, whereas we know it is in fact indecomposable for $X$.

Using the fact that $T_{E_7}(\lambda_2)$ admits a quotient $\lambda_2^2$, together with the fact that $2\lambda_1$, $2\lambda_3$ and $2\lambda_4$ have trivial first cohomology, we see that all of the composition factors $2\lambda_1$, $2\lambda_3$ and $2\lambda_4$ must appear as submodules once we factor out the socle of $T_{E_7}(\lambda_2)$, as well as all of the remaining trivial composition factors. Hence the same is also true for $L(G)$. Finally, by self-duality of $T_{E_7}(\lambda_2)$, it follows that the remaining two composition factors $\lambda_2^2$ do not appear in this second socle layer of $T(\lambda_2)$, hence they do not appear in the second socle layer of $L(G)$. The given socle series for $L(G) \downarrow X$ now follows, as well as
\[ T_{E_7}(\lambda_2) \downarrow X = \lambda_2^2|(2\lambda_1 + 2\lambda_3 + 2\lambda_4 + 0^4)|(0^2 + \lambda_2^2). \]
\end{proof}

\begin{remark}
Note that combining $T_{E_7}(\lambda_2) \downarrow X$ above with Proposition \ref{prop:D4restrictionMin}(ii), we immediately deduce the module structure of $L(E_8) \downarrow X$.  
\end{remark}

Finally, we need to determine the action of $X$ of type $D_4$ in $G = E_8$ when $X$ is contained in no reductive overgroup. To do this, we require an explicit construction of such a subgroup in an $E_6$-parabolic $P$ of $G$. For completeness, we construct of a conjugate of every non-$G$-cr subgroup of type $D_4$ in $P$. 

\begin{proposition} \label{prop:D4inE8gens}
Let $G$ be of type $E_8$, $p=2$ and let $P = QL$ be the standard $E_6$-parabolic subgroup of $G$. Let $X$ be a subgroup of type $D_4$ contained in $P$, with irreducible image $X_0$ in $L$. Then, identifying $\mathbb{V} = (H^{1}(X_0,\lambda_2))^3 = K^{6}$, there exists $\mathbf{v} = (\phi_1,\phi_2,\phi_3) = (a_1,\ldots,a_6) \in \mathbb{V}$ such that $X$ is conjugate to $X_\mathbf{v} = \langle y_{\pm i}(t) : t \in K,\, 1 \le i \le 4 \rangle$ for the following elements $y_{\pm i}(t)$.
\begin{align*}
y_{\pm1}(t) &= x_{\pm3}(t)x_{\pm5}(t), \\
y_{\pm2}(t) &= x_{\pm1}(t)x_{\pm6}(t), \\
y_3(t) &= x_{0011000}(t)x_{0001100}(t)x_{11232110}(a_1t) x_{11232111}(a_3t) x_{12354321}(a_5t),\\ 
y_{-3}(t) &= x_{-0011000}(t)x_{-0001100}(t)x_{11111110}(a_1t) x_{11111111}(a_3t) x_{12233321}(a_5t) ,\\
y_4(t) &= x_{0111000}(t)x_{0101100}(t)x_{12232110}(a_2t) x_{12232111}(a_4t) x_{13354321}(a_6t), \\
y_{-4}(t) &= x_{-0111000}(t)x_{-0101100}(t)x_{10111110}(a_2t) x_{10111111}(a_4t) x_{11233321}(a_6t).
\end{align*}
\end{proposition}

\begin{proof}
By construction, the given elements lie in $P$, and when $a_1,\ldots,a_6$ are all zero these generate the $L$-irreducible subgroup $X_0$, as shown in \cite[p.\ 444]{MR1367085}. It is also clear that the cocycles have image in the appropriate modules in the levels of $Q$. It remains, therefore, only to prove that each such subgroup is indeed a group of type $D_4$. This is now routine using \cite[Theorem 12.1.1]{MR0407163}; our calculations were assisted using \Magma.
\end{proof}

\begin{remark}
The form of the cocycles in the above construction was derived using the restriction $H^{1}(X_0,0100) \to H^{1}(Y_0,0100 \downarrow Y_0)$, where $Y_0$ is the subgroup of type $A_1^2$ of $X_0$ generated by $x_{\pm 0011000}(t)x_{\pm 0001100}(t)$ and $x_{\pm 0111000}(t)x_{\pm 0101100}(t)$. Then $Y_0$ is the derived subgroup of a Levi subgroup of $X_0$. Since $1000 \downarrow Y_0 = (1,1) + 0^{4}$, it follows that $0100 \downarrow Y_0 = (2,0) + (0,2) + (1,1)^4 + 0^6$, a completely reducible module. As mentioned earlier, in a simply connected semisimple algebraic group, the derived subgroup of any Levi subgroup is again simply connected. The Lie algebra of a simply connected simple group of type $D_4$ has shape $0100|0^2$, and the Lie algebra of a simply connected semisimple group of type $A_1 A_1$ has shape $((2,0)|0) + ((0,2)|0)$. It follows that the $X_0$-module $0100|0^2$ restricts to $Y_0$ as $(2,0)|0 + (0,2)|0 + (1,1)^4 + 0^6$, and in particular the map $H^{1}(X_0,0100) \to H^{1}(Y_0, 0100 \downarrow Y_0)$ is injective. Finally, we use the explicit description of cocycles for subgroups of type $A_1$ in \cite[Lemma 3.6.2]{MR3075783}.
\end{remark}

With this construction in hand, the following result allows us to deduce the action of the algebraic group using a well-chosen finite subgroup. 

\begin{lemma} [{\cite[Proposition~1.4]{MR1458329}}] \label{lem:fixfinite}
Let $X$ be a connected simple algebraic group and let $Y$
be a finite subgroup of $X$. Suppose $V$ is a finite-dimensional $X$-module
satisfying the following conditions:
\begin{enumerate}
\item Every $X$-composition factor is irreducible for $Y$; 
\item for each pair of $X$-composition factors $M$, $N$ of $V$, restriction $\Ext_X^1(M,N) \rightarrow \Ext^1_{Y}(M,N)$ is an injective map;
\item for each pair of $X$-composition factors $M$, $N$ of $V$, if $M \downarrow Y \cong N \downarrow Y$ then $M \cong N$ as $X$-modules.
\end{enumerate}
Then $X$ and $Y$ fix exactly the same subspaces of $V$. 
\end{lemma}

\begin{lemma} \label{lem:d4actione8}
Let $G=E_8$, $p=2$ and $X$ be an MR non-$G$-cr subgroup of type $D_4$. Then $L(G) \downarrow X$ has the following socle series:   
 \[T(\lambda_2)^2 + (0 | \lambda_2^3 | (0^4 + 2\lambda_1 + 2\lambda_3 + 2\lambda_4) | (\lambda_2^2 + 0^2) | (\lambda_2 + 0) ). \] 
\end{lemma}

\begin{proof}
We apply Lemma \ref{lem:fixfinite} to $X$ and its finite subgroup $Y:=X(4) \cong \SO_8^+(4)$, acting on $L(G)$. To do so, we must check the three conditions hold. Conditions (i) and (iii) are immediate from inspection of the $X$-composition factors of $L(G)$. Condition (ii) follows from \cite[Theorem~7.4]{MR0439856}, noting that $\langle \lambda, \alpha_j\rangle \le 3$ for $j \in \{1, \ldots, 4\}$ for all highest weights $\lambda$ occurring in $L(G) \downarrow X$.

Therefore, $X$ and $Y$ stabilise exactly the same subspaces of $L(G)$. In particular, the claimed socle series for $L(G) \downarrow X$ follows from the socle series for $L(G) \downarrow Y$. The latter socle series can be calculated in \Magma, using the generators in Proposition \ref{prop:D4inE8gens} with $a_1 = a_4 = 1$ and $a_i = 0$ for $i=2,3,5,6$.  
\end{proof}

\begin{lemma} \label{lem:B3inD4}
Let $G$ be of type $E_8$, $p=2$ and $X$ be an MR non-$G$-cr subgroup of type $D_4$. If $Y < X$ is of type $B_3$ then $Y$ is non-$G$-cr and acts with five direct summands on $L(G)$ of dimensions 30, 30, 62, 63, 63.  
\end{lemma}

\begin{proof}
A group of type $D_4$ contains three conjugacy classes of maximal subgroups of type $B_3$. From the construction of $X$ in Proposition \ref{prop:D4inE8gens}, we can write down generators for the three classes. Using \Magma{} we find they each contain a finite subgroup $\SO_7(4)$ which acts with five direct summands on $L(G)$ of dimensions $30, 30, 62, 63, 63$. Another routine application of Lemma \ref{lem:fixfinite} yields that these are actually the direct summands of each maximal subgroup of type $B_3$. That these are non-$G$-cr now follows since no Levi subgroup can contain a subgroup acting with summands of the given dimensions.    
\end{proof}

\begin{remark}
It follows from the classification of non-$G$-cr subgroups of type $B_3$ in $G = E_8$ when $p=2$ that all subgroups of type $B_3$ in an MR subgroup of type $D_4$ are conjugate to one another. We do not state this in the previous lemma, because the lemma itself used in the proof of the classification.  
\end{remark}


\section{Tables of embeddings for Theorem~\ref{THM:MAIN}} \label{sec:tables}

We now give the tables listing non-$G$-cr subgroups $X$ of simple algebraic groups $G$ as in Theorem~\ref{THM:MAIN}, which have been enumerated in Sections~\ref{sec:A3}--\ref{sec:D4}. Within each table, each given $X$ indicates a unique $G$-conjugacy class of subgroups, except where explicitly stated otherwise. 

The first column in each table gives the Lie type of $X$. For each type, we use a horizontal line starting from Column 2 to distinguish between non-$G$-cr subgroups minimally contained in different association classes of parabolic subgroups. For convenience, the order is the same as the order in which we have considered the parabolic subgroups in the proof.

The second column either gives a proper reductive overgroup $M$ of $X$ or `MR', indicating that $X$ is maximal among connected reductive subgroups of $G$, cf.\ Corollaries\ \ref{cor:MR} and \ref{cor:overgroups}. When $M$ is classical we use the notation of Section~\ref{sec:embed} to indicate its embedding. When $M$ is exceptional, it turns out that $X$ is always non-$M$-cr and we give the same overgroup of $X$ as in the table for $M$ (when $X$ is MR in $M$, we just list $M$).

The next columns give the action of $X$ on a non-trivial $G$-module of least dimension, and on the adjoint module $L(G)$ (when these modules are different); the structure of the Weyl modules and tilting modules in these columns is given in Appendix~\ref{sec:ancillary}.

The final column gives the structure of $C_{G}(X)^\circ$ using the notation in Section~\ref{sec:centnot}. The symbol $(\dagger)$ indicates that $X$ is separable in $G$, i.e.\ $L_{C_G(X)}(X) = L(C_G(X))$. We recall that $G$ is simply connected, so in the case $G = E_7$, where $G$ is not a separable subgroup of itself, the inseparability of many subgroups $X$ is an artefact of this choice.  

\begin{landscape}
\begin{longtable}{lllll}
\caption{Non-$G$-cr subgroups of $G = F_4$, $p = 2$} \label{tab:F4} \\
\hline
$X$ & Embedding of $X$ & $V_{26} \downarrow X$ & $L(G) \downarrow X$ & $C_G(X)^\circ$ \\
\hline

$B_3$ & $D_4$ via $T(100)$ & $T(100) + 001^2 + 0^2$ & $(100|010|100|0^2) + T(100) + 001^2$ & $U_1$ \\ 
\cline{2-5}
 & $\tilde{D}_4$ via $T(100)$ & $100|010|100$ & $((002 + 0^2) | (200 + 100) | (010 + 002 + 0) | 100 ) + 0$ & $U_1$ $(\dagger)$ \\ 
\hline
\end{longtable}

\begin{longtable}{lllll}
\caption{Non-$G$-cr subgroups of $G = E_6$, $p = 2$} \label{tab:E6} \\
\hline
$X$ & Embedding of $X$ & $V_{27} \downarrow X$ & $L(G) \downarrow X$ & $C_G(X)^\circ$ \\
\hline

$B_3$ & $\tilde{D}_4 < F_4$ & $(100|010|100) + 0$ & $100| (010 + 002 + 0^2) | (200 + 100^2) | (010 + 002 + 0) | (100 + 0)$ & $U_1$ $(\dagger)$ \\ 

\hline
\end{longtable}

{\small
\begin{longtable}{llrll}
\caption{Non-$G$-cr subgroups of $G = E_7$, $p = 2$} \label{tab:E7} \\
\hline
$X$ & Embedding of $X$ & & Module actions & $C_G(X)^\circ$ \\
\hline
$A_3$ & $D_6$ via $010 + 010$ & $V_{56} \downarrow X =$	& $010^4 + T(200) + T(002)$, & $U_5\bar{A}_1$ \\
& & $L(G) \downarrow X =$ 	& $W(101) + W(101)^* +  T(101)^4 + T(020) + 0^3$ \\
\cline{2-5}
& $D_6$ via $010 + 010$  & $V_{56} \downarrow X =$ 	& $010^4 + T(101)^2$ &  $U_1\bar{A}_1$ \\
& & $L(G) \downarrow X =$ 	& $W(101) + W(101)^* + T(200)^2 + T(002)^2 + T(020) + 0^3$ \\ 

\hline
$B_3$ & $E_6$ & $V_{56} \downarrow X =$ & $(100|010|100)^2 + 0^4$ & $U_1A_1$ $(\dagger)$ \\
& & $L(G) \downarrow X =$ & $(100| (010 + 001 + 0^2) | (200 + 100^2) | (010 + 001 + 0) | (100 + 0)) + (100|010|100)^2 + 0^3$ \\
& $A_6$ via $W(100)$ & $V_{56} \downarrow X =$ & $W(100) + W(100)^* + ((0 + 010)|100) + (100|(0 + 010))$ & $U_2 T_1$ $(\dagger)$ \\
& & $L(G) \downarrow X=$ & $(100|(010+0)|200|(010+0)|100) + W(100) + W(100)^* + (100|002|100|(010 + 0)) +$ \\
& & & $(010|100|(002+0)|100) + 0$ \\ 
& $A_7$ via $T(100)$ & $V_{56}\downarrow X =$ & $(100|(010+0)|(100 + 0))^2 $ & $U_2$ \\ 
& & $L(G)\downarrow X=$ & $(100|(010+0^2)|(200+100)|(100+010+0)|(100+0)) +$ \\
& & & $(100|(010+002)|100^2|(010+002+0)|(100+0)) + 0$ \\ 
& $A_7$ via $001$ & $V_{56} \downarrow X =$ & $(0|100|010|100|0)^2$ & $U_1$ \\ 
 & & $L(G) \downarrow X=$ & $(100|010|100|(002 + 0)|100|(010+0)|100) +$ \\
 & & & $(100|010|100|(002+0^2)|(200+100)|(010+0)|(100+0)) + 0$ \\ 
\hline
$C_4$ & $A_7$ via $1000$ & $V_{56} \downarrow X=$ & $T(0100)^2$ & $U_1$ \\
& & $L(G) \downarrow X=$ & $(0100|0|2000|0|0100) + (0100|(0001 + 0)|(0100+0)) + 0$\\ 
\hline
$D_4$	& MR ($\infty$ classes) & $V_{56} \downarrow X=$ & $0^2|0100^2|0^2$  & $U_1$ \\
 		& & $L(G) \downarrow X=$ & $0100^2|(2000 + 0020 + 0002 + 0^3)|(0^2 + 0100^2)$ \\ 
& $A_7$ via $1000$ & $V_{56}\downarrow X=$ & $(0|0100|0)^2$ & $U_1$ \\
& & $L(G) \downarrow X=$ & $(0100|(2000 + 0^2)|0100) +  (0100|(0020 + 0002 + 0)|(0100+0)) + 0$ \\ 
\hline
\end{longtable}
}

\begin{longtable}{llll}
\caption{Non-$G$-cr subgroups of $G = E_8$, $p=3$} \label{tab:E8p3} \\
\hline
$X$ & Embedding of $X$ & $L(G) \downarrow X$ & $C_G(X)^\circ$\\
\hline

$C_3$ & $D_8$ via $T(010) + 000$ & $200 + T(010) + T(101) + W(110) + W(110)^*$ & $U_1$ $(\dagger)$ \\

\hline
\end{longtable}

\begin{longtable}{llll}
\caption{Non-$G$-cr subgroups of $G = E_8$, $p=2$} \label{tab:E8p2} \\
\hline
$X$ & Embedding of $X$ & $L(G) \downarrow X$ & $C_G(X)^\circ$\\
\hline

$A_3$ & $D_6$ via $010 + 010$ & $T(200)^2 + T(101)^4 + W(101) + W(101)^* + T(020) + 010^8 + T(002)^2 + 0^6$ & $U_5 \bar{A}_1^2$ $(\dagger)$ \\

\cline{2-4} 

& $D_8$ via $T(101)$ & $(101|(020+0^2)|(210 + 012 + 101)|(020 + 0)|(101^2+0)) + 111^2$ & $1$ \\ 

\hline

$B_3$ & $E_6$ & $(100| (010 + 001 + 0^2) | (200 + 100^2) | (010 + 001 + 0) | (100 + 0)) + (100|010|100)^6 + 0^{14}$ & $U_1 G_2$ $(\dagger)$ \\

& $A_6$ via $W(100)$ & \makecell[l]{$(100|(010+0)|200|(010+0)|100) + W(100)^3 + (W(100)^*)^3 + ((0 + 010)|100)^2 +$ \\ $ (100|(0 + 010))^2 + (100|002|100|(010 + 0)) + (010|100|(002+0)|100) + 0^4$} & $U_6 \bar{A}_1 T_1$ $(\dagger)$ \\

& $A_7$ via $T(100)$ & $T(200) + T(100)^4 + (100|(010 + 0)|(100 + 0))^2 + (010|002|100|010)^2$ & $U_5A_1$ $(\dagger)$ \\

& $A_7' < E_7$ via $T(100)$ & \makecell[l]{$(100|(010+0^2)|(200+100)|(100+010+0)|(100+0)) + (100|(010+0)|(100 + 0))^4 +$ \\ $(100|(010+002)|100^2|(010+002+0)|(100+0)) + 0^4$} & $U_6\bar{A}_1$ \\ 

& $A_7' < E_7$ via $001$ & \makecell[l]{$(100|010|100|(002 + 0)|100|(010+0)|100) + (0|100|010|100|0)^4 +$ \\ $  (100|010|100|(002+0^2)|(200+100)|(010+0)|(100+0)) + 0^4$} & $U_5\bar{A}_1$ \\ 

& $D_7$ via $W(100) + W(100)^*$ & \makecell[l]{$W(100)^2 + (W(100)^*)^2 + (100|(010+0)|200|(010+0)|100) + ((0+010)|100|0) +$ \\ $ (0|100|(0+010)) + T(002)^2$} & $U_3 A_1$  $(\dagger)$ \\ 

& $D_8$ via $T(100)^2$ & $(100|010|100|0^2) + (0|100|(010+0)|100) + T(200) + T(002)^2$& $U_5$ \\ 

& non-$G$-cr MR $D_4$  & $T(010)^2 + (100|(010+0^2)|(200+100)|(100+010+0)|(100+0)) +$ & $U_6$ $(\dagger)$ \\
&                               & $(100|010|100|(002+0)|100|(010+0)|100|0) + (0|100|010|100|(002+0)|100|(010+0)|100)$  \\

\cline{2-4} 

& $D_8$ via $001^2$ & $(0|100|010|100|0)^2 + 101^2 + 001^4 + T(002)$& $U_1$ $(\dagger)$ \\ 

\cline{2-4}

& \makecell[l]{$D_8$ via \\ $T(100) + T(100)^{[r]}$ $(r \neq 0)$} & \makecell[l]{$(0|(100 + 100^{[r]})|(010 + 010^{[r]} + 0)|(100 + 100^{[r]})|0^2) +$ \\ $(0|(100 + 100^{[r]})|(100 \otimes 100^{[r]} + 0^2)|(100 + 100^{[r]})|0) + (001 \otimes 001^{[r]})^2$} & $U_1$ \\

\hline

$B_3^2$ & \makecell[l]{$D_8$ via \\ $(0|(100,0)|0) + (0|(0,100)|0)$} & \makecell[l]{$ (0|((100,0) + (0,100))|((010,0) + (0,010) + 0)|((100,0) + (0,100))|0^2) +$ \\ $(0|((100,0) + (0,100))|((100,100) + 0^2)|((100,0) + (0,100))|0) + (001,001)^2$}  & $U_1$ $(\dagger)$ \\

\hline

$B_4$ & $A_8$ via $W(1000)$ &\makecell[l]{$(1000|0|0100|0|2000|0|0100|0|1000) + (1000|(0010 + 0)|0100|0) +$ \\ $(0|0100|0|(1000+0010))$} & $U_1$ $(\dagger)$ \\

& $D_8$ via $0001$ & \makecell[l]{$(1000|0|0100|(0010 + 0)|0100|0|(1000 + 0)) +$ \\ $(1000|0|0100|0|2000|(0010+0)|0100|0|1000)$} & $1$ \\

\hline

$C_4$ & $A_7 < E_7$ & $(0100|0|2000|0|0100) + T(0100)^4 + (0100|(0001 + 0)|(0100+0)) + 0^4$ & $U_{5}\bar{A}_1$ \\

\hline

$D_4$ & $D_8$ via $1000+1000$ &  $(0^2|0100^2|0^2) + T(2000) + T(0011)^2$ & $U_1$ \\ 

\cline{2-4}

& $E_7$ ($\infty$ classes) & $(0100^2|(2000 + 0020 + 0002 + 0^4)|(0^2 + 0100^2)) + (0^2|0100^2|0^2)^2 + 0^2$ & $U_{5}\bar{A}_{1}$ $(\dagger)$ \\

& $A_7 < E_7$ & $(0100|(2000+0^2)|0100) + (0|0100|0)^4 + (0100|(0020 + 0002 + 0)|(0100+0)) + 0^4$ & $U_{5}\bar{A}_{1}$ \\

& $D_8$ via $1000+1000$ & $(0100|0^2) + (0^2|0100) + T(2000) + T(0020) + T(0002)$ & $U_4$ \\ 

& MR & $T(0100)^2 + (0 | 0100^3 | (0^4 + 2000 + 0020 + 0002) | (0100^2 + 0^2) | (0100 + 0) )$ & $U_5$ $(\dagger)$ \\ 

\hline
\end{longtable}

\end{landscape}

\appendix

\section{Ancillary data} \label{sec:ancillary}

Here we provide the socle series for the Weyl and tilting modules occurring in Tables \ref{tab:F4}--\ref{tab:E8p2}. These mostly follow from the Weyl character formula and knowledge of the weights of low-dimensional irreducible modules, as given for instance in \cite{MR1901354}. Explicit calculations are also facilitated using S.~Doty's software package \cite{Dot1}. For tilting modules, attention is focused on indecomposable modules $T(\lambda)$, since direct sums and summands of tilting modules are tilting; moreover tensor products of tilting modules are tilting. Finally, the uniqueness of $T(\lambda)$ for each dominant weight $\lambda$ implies that $T(\lambda) \cong T(-w_0 \lambda)^{\ast}$ where $w_0$ is the longest element of the Weyl group; for us this is often sufficient information to determine the submodule structure of $T(\lambda)$.

Two of the most complicated tilting modules occurring in this paper arise for groups of type $B_3$ in characteristic $2$, specifically $T(200)$ and $T(002)$. The former is $T(100) \otimes T(100)$ and the latter is $001 \otimes 001$. Using a sufficiently large finite subgroup, we can use Lemma \ref{lem:fixfinite} to perform explicit computations, e.g.\ in \Magma, to obtain the precise module structure. By \cite[Corollary 7.5]{MR0439856}, a subgroup $\Omega_7(4)$ is `sufficiently large' for our purposes.

{\small
\begin{center}
\begin{tabular}{c|c|c|c}
$X$ & $p$ & $\lambda$ & Socle Series of $W(\lambda)$ \\ \hline
$A_3$ & $2$ & $101$ & $101|0$ \\
	  \hline
$B_3$ &$2$ & $100$ & $100 | 0$ \\ \hline
$C_3$ & $3$ & $110$ & $110|001$ \\
\hline
$B_4$ & $2$ & $1000$ & $1000 | 0$ \\
\end{tabular}
\end{center}
}

{\small
\begin{center}
\begin{tabular}{c|c|c|c}
$X$ & $p$ & $\lambda$ & Socle Series of $T(\lambda)$ \\ \hline
$A_3$ & $2$ & $101$ & $0|101|0$ \\
	  & & $020$ & $101|(020 + 0)|(101 + 0)$ \\ 
	  & & $200$ & $010|200|010$ \\
	  & & $002$ & $010|002|010$ \\
	  \hline
$B_3$ &$2$ & $100$ & $0 | 100 | 0$ \\
      & & $010$ & $0|100|(010 + 0)|100|0^2$ \\
	  & & $200$ & $100|(010 + 0^2)|(200+100+0)|(100+010+0)|(100+0^2)$ \\
	  & & $002$ & $0 | 100 | 101 | 100 | (002 + 0) | 100 | (101 + 0) | 100 | 0$ \\ \hline
$C_3$ & $3$ & $010$ & $0 | 010|0$ \\ 
	  & 	& $101$ & $010|(101+0)|010$ \\
\hline
$C_4$ & $2$ & $0100$ & $0 | 0100 | 0$ \\
\hline
$D_4$ & $2$ & $2\lambda_i$, $i=1,3,4$ & $0 | 0100 | 0 | (2\lambda_i + 0) | 0100 | 0$ \\
	  & & $0100$ & $0^2 | 0100 | 0^2$ \\ 
	  & & $0011$ & $1000 | 0011 | 1000$ \\ 
	  
\end{tabular}
\end{center}
}

\section*{Acknowledgements} 

The authors thank the London Mathematical Society for support through a Scheme 4 grant, and the MFO for a Research in Pairs visit. They also thank the Isaac Newton Institute for Mathematical Sciences for support and hospitality during the programme \emph{Groups, Representations and Applications: New perspectives}, when much of the work on this paper was undertaken. This was supported by: EPSRC grant number EP/R014604/1.

The first author acknowledges support from the Alexander von Humboldt Foundation, Germany, as well as a Scheme 9 \emph{Research Reboot} grant from the London Mathematical Society.

The second author is supported by EPSRC grant EP/W000466/1.

\bibliographystyle{amsplain}
\bibliography{biblio-rank-at-least-3}

\end{document}